\documentclass[a4paper,12pt,leqno]{article}

\usepackage[latin1]{inputenc}
\usepackage[T1]{fontenc}
\usepackage{lmodern}

\usepackage[english]{babel}
\usepackage{amsfonts}
\usepackage{ifthen}
\usepackage{xspace}
\usepackage{xcolor}
\usepackage{lmodern}
\usepackage{comment}
\usepackage{multicol}
\usepackage{adjustbox}
\usepackage{setspace}

\usepackage{enumitem}

\usepackage{amsmath}
\usepackage{amssymb}
\usepackage{mathtools}
\usepackage{stmaryrd}
\usepackage{amsthm}
\usepackage{bm} 
\usepackage{tikz}
\usepackage[b]{esvect} 

\usepackage[top=1.63in, bottom=1.63in, left=1.37in, right=1.37in]{geometry}

\allowdisplaybreaks[1]

\theoremstyle{plain}
\newtheorem{theorem}{Theorem}[section] 

\newtheorem{lemma}[theorem]{Lemma} 
\newtheorem{corollary}[theorem]{Corollary} 
\newtheorem{proposition}[theorem]{Proposition} 

\newtheorem{claim}{Claim}

\newtheorem{innercustomclaim}{Claim}
\newenvironment{customclaim}[1]
  {\renewcommand\theinnercustomclaim{#1}\innercustomclaim}
  {\endinnercustomclaim}

\theoremstyle{definition}
\newtheorem{definition}{Definition}[section] 
\newtheorem{example}{Example}[section] 

\newtheorem{observation}{Observation}[section]

\newtheorem*{remark}{Remark}

\allowdisplaybreaks[1]

\newcommand{\tuple}[1]{\vec{#1}} 
\newcommand{\rec}[2]{[#1\triangleright#2\,]} 

\DeclareMathOperator{\true}{\vDash}	
\DeclareMathOperator{\ntrue}{\nvDash} 
\DeclareMathOperator{\dom}{dom} 
\DeclareMathOperator{\ran}{ran} 
\DeclareMathOperator{\im}{ran} 

\DeclareMathOperator{\vr}{Vr} 
\DeclareMathOperator{\fr}{Fr} 
\DeclareMathOperator{\subf}{Sf}  
\DeclareMathOperator{\Tset}{T_\text{$L$}} 
\DeclareMathOperator{\FOset}{FO_\text{$L$}} 
\DeclareMathOperator{\INCset}{INC_\text{$L$}} 
\DeclareMathOperator{\EXCset}{EXC_\text{$L$}} 
\DeclareMathOperator{\INEXset}{INEX_\text{$L$}} 
\DeclareMathOperator{\ESOset}{ESO_\text{$L$}} 

\DeclareMathOperator{\ESO}{ESO}
\DeclareMathOperator{\EMSO}{EMSO}
\DeclareMathOperator{\INC}{INC}
\DeclareMathOperator{\EXC}{EXC}
\DeclareMathOperator{\INEX}{INEX}

\DeclareMathOperator{\inc}{\text{$\subseteq$}} 
\DeclareMathOperator{\exc}{\text{$\mid$}} 
\DeclareMathOperator{\ince}{\text{$\subseteq^{e}$}} 
\DeclareMathOperator{\dep}{\text{$=$}} 

\DeclareMathOperator{\Ee}{\exists} 
\DeclareMathOperator{\Es}{\exists} 
\DeclareMathOperator{\Ae}{\forall} 

\newcommand*{\abs}[1]{\left\lvert#1\right\rvert}   

\newcommand{\apal}[1]{}

\begin{document}

\title{Capturing $k$-ary Existential Second Order Logic \\ with $k$-ary Inclusion-Exclusion Logic\thanks{This paper is an extended version of  \cite{Ronnholm18} with additional technical details.}}
\author{Raine Rönnholm \\ \small University of Tampere}
\date{}
\maketitle

\begin{abstract}
In this paper we analyze $k$-ary inclusion-exclusion logic, INEX[$k$], which is obtained by extending first order logic with $k$-ary inclusion and exclusion atoms. We show that every formula of INEX[$k$] can be expressed with a formula of $k$-ary existential second order logic, ESO[$k$]. Conversely, every formula of ESO[$k$] with at most $k$-ary free relation variables can be expressed with a formula of INEX[$k$]. From this it follows that, on the level of sentences, INEX[$k$] captures the expressive power of ESO[$k$]. 

We also introduce several useful operators that can be expressed in INEX[$k$]. In particular, we define inclusion and exclusion quantifiers and so-called term value preserving disjunction which is essential for the proofs of the main results in this paper. Furthermore, we present a novel method of relativization for team semantics and analyze the duality of inclusion and exclusion atoms.

\end{abstract}



\section{Introduction}

The origin of inclusion and exclusion logics lies in the notion of dependence and imperfect information in logic. First approaches in this area were \emph{partially ordered quantifiers} by Henkin \cite{Henkin61} and \emph{IF-logic} (independence friendly logic) by Hintikka and Sandu \cite{Hintikka89}. The truth for IF-logic was originally defined by using semantic games of imperfect information (\cite{Hintikka97}), but an equivalent compositional semantics was presented later by Hodges \cite{Hodges97}. However, in the compositional approach it is not sufficient to consider single assignments, but instead sets of assignments which are called \emph{teams}. 

Teams can be seen as parallel positions in a semantic game, or can be interpreted as information sets or as databases (\cite{Vaananen07}).
By using similar \emph{team semantics} as Hodges, Väänänen \cite{Vaananen07} introduced \emph{dependence logic} which extends first order logic with new atomic formulas called \emph{dependence atoms}. Later Grädel and Väänänen \cite{Gradel12} presented \emph{independence logic} by analogously adding \emph{independence atoms} to first order logic. The truth conditions for these atoms are defined by dependencies/independencies of the values of terms in a team. These logics have been recently studied actively with an attempt to formalize the dependency phenomena in different fields of science. There has been research 
in several areas such as database dependency theory (\cite{Kontinen13b}), belief presentation~(\cite{Galliani12a}) and quantum mechanics (\cite{Paolini15}). 


Inclusion and exclusion logics were first presented by Galliani \cite{Galliani12b}. They extend first order logic with \emph{inclusion and exclusion atoms} as dependence atoms in dependence logic. Suppose that $\tuple{\vphantom\wedge\smash{t}}_1,\tuple{\vphantom\wedge\smash{t}}_2$ are $k$-tuples of terms and $X$ is a team. The $k$-ary inclusion atom $\tuple{\vphantom\wedge\smash{t}}_1\inc\tuple{\vphantom\wedge\smash{t}}_2$ says that the values of $\tuple{\vphantom\wedge\smash{t}}_1$ are included in the values of $\tuple{\vphantom\wedge\smash{t}}_2$ in the team $X$. The $k$-ary exclusion atom $\tuple{\vphantom\wedge\smash{t}}_1\,|\,\tuple{\vphantom\wedge\smash{t}}_2$ analogously says that $\tuple{\vphantom\wedge\smash{t}}_1$ and $\tuple{\vphantom\wedge\smash{t}}_2$ get distinct values in $X$. These are simple and natural dependencies in database theory (\cite{Galliani12b}), and thus it is reasonable to consider such atoms in a team semantical setting. 


Inclusion and exclusion atoms have some natural complementary properties. Exclusion logic is known to be \emph{closed downwards} (\cite{Galliani12b}), i.e. if a team satisfies some formula, then also all of its subteams satisfy it. Inclusion logic, on the other hand, is known to be \emph{closed under unions} (\cite{Galliani12b}), i.e. if each team in a set of teams satisfies a formula, then also their union satisfies it. However, neither of these logics is both closed downwards and under unions. Therefore the combination of these logics, \emph{inclusion-exclusion logic}, has neither of these properties.

Exclusion logic is equivalent with dependence logic (\cite{Galliani12b}) which captures \emph{existential second order logic}, ESO, on the level of sentences (\cite{Vaananen07}). Inclusion logic is not comparable with dependence logic in general (\cite{Galliani12b}), but captures \emph{positive greatest fixed point logic} on the level of sentences, as shown by Galliani and Hella~\cite{Hella13}. Hence exclusion logic captures NP, and inclusion logic captures PTIME over finite structures with linear order. Inclusion-exclusion logic has been shown to be equivalent with independence logic by Galliani~\cite{Galliani12b}. Galliani has also shown in \cite{Galliani12b} that with inclusion-exclusion logic it is possible to define exactly those properties of teams which are definable in ESO.  Thus we can say that inclusion-exclusion logic captures ESO \emph{on the level of formulas}. 

\medskip

By these earlier results, we see that the expressive power of inclusion-exclusion logic is rather strong. Instead of studying this whole logic, we will consider its weaker fragments. One of the most canonical approaches is to restrict the arities of inclusion and exclusion atoms. In particular, unary atoms are much simpler than inclusion and exclusion atoms in general. Hannula \cite{Hannula14} has shown that inclusion logic has a strict arity hierarchy over graphs, but it is still open what is the exact fragment of ESO that corresponds to $k$-ary inclusion logic, INC[$k$]. Before our work similar research has not been done for exclusion- nor for inclusion-exclusion logic. Our main research question for this paper was to examine whether there is some natural fragment of ESO that corresponds to \emph{unary} inclusion-exclusion logic, INEX[$1$].

Similar research has been done on the related logics: Durand and Kontinen \cite{Durand12} have shown that, on the level of sentences, $k$-ary dependence logic captures the fragment of  ESO in which at most ($k$$-$$1$)-ary functions can be quantified. Galliani, Hannula and Kontinen \cite{Kontinen13a} have shown that the same result holds also for $k$-ary independence logic. The arity hierarchy of ESO (over arbitrary vocabulary) is known to be strict, as shown by Ajtai \cite{Ajtai83} in 1983. Consequently dependence and independence logics have a strict arity hierarchy over sentences. 

These earlier results, however, do not tell much about the expressive power of $k$-ary exclusion logic, EXC[$k$], and $k$-ary inclusion-exclusion logic, INEX[$k$], since the known translations from them to dependence and independence logics do not respect the arities of atoms. Also, since these results are proven on the level of sentences, we do not know much how does the arity affect the expressive power of these logics on the level of formulas.

\medskip

We will show in Subsection \ref{ssec: Expressing INEX} that every formula of EXC[$k$] can be expressed with a formula of $k$-ary ESO, ESO[$k$]. The idea of this compositional translation is that for each occurrence of an exclusion atom $\tuple{\vphantom\wedge\smash{t}}_1\,|\,\tuple{\vphantom\wedge\smash{t}}_2$ we quantify a separate $k$-ary relation variable that gives limits to the values that the tuple $\tuple{\vphantom\wedge\smash{t}}_1$ can get and $\tuple{\vphantom\wedge\smash{t}}_2$ cannot. We can formulate a similar, yet more complex, translation for INC[$k$] and then merge these two translations to create a translation from INEX[$k$] to ESO[$k$].

In Subsection \ref{ssec: Expressing ESO} we will show that all ESO[$k$]-formulas that contain at most $k$-ary free relation variables can be expressed with a formula of INEX[$k$]. The translation we use here is compositional, very natural and uses inclusion and exclusion atoms in a dualistic way: The quantified $k$-ary relation variables $P_i$ are just replaced with $k$-tuples $\tuple w_i$ of quantified first order variables. Then we simply replace atomic formulas of the form $P_i\tuple{\vphantom\wedge\smash{t}}$ with inclusion atoms $\tuple{\vphantom\wedge\smash{t}}\inc\tuple w_i$ and formulas of the form $\neg P_i\tuple{\vphantom\wedge\smash{t}}$ with exclusion atoms $\tuple{\vphantom\wedge\smash{t}}\;|\,\tuple w_i$.

In order to get make this last translation compositional, we also need a new operator called \emph{term value preserving disjunction} which is introduced in Subsection \ref{ssec: Term value preserving disjunction}. We will show that this operator can be expressed with inclusion and exclusion atoms, and furthermore when preserving values of $k$-tuples, it can be defined in INEX[$k$]. We will also explain in Subsection~\ref{ssec: Term value preserving disjunction} why this is a useful operator for the framework of team semantics in general.

From our results it follows that, on the level of sentences, INEX[$k$] captures the expressive power of ESO[$k$]. In particular, by using only unary inclusion and exclusion atoms we get the expressive power of \emph{existential monadic second order logic}, EMSO. This special case should be noted for the following reason: As a consequence of the results mentioned above (\cite{Durand12,Kontinen13a}), if we extend FO with 1-ary dependence (or~independence) atoms, the expressive power stays inside~FO. But if we extend FO with 2-ary dependence (or~independence) atoms, the expressive power becomes already stronger than EMSO. Thus INEX[$1$] deserves extra recognition by capturing this important fragment of ESO that has not yet been characterized in the framework of team semantics.

\medskip

In addition to our main results, we also analyze the nature of inclusion and exclusion logics and their relationship more deeply. Even though inclusion and exclusion atoms are not contradictory negations of each other, we claim that they can be seen as duals of each other and thus they make a natural pair. This is one more reason why inclusion-exclusion logic can be seen as a quite canonical logic for the framework of team semantics.

We also analyze inclusion and exclusion relations from an another perspective by introducing \emph{inclusion and exclusion quantifiers}.  This can be seen as a step back to the origin of these logics, since dependence logic was inspired by IF-logic, in which dependencies were handled with quantification.
In Subsection \ref{ssec: Quantifiers} we first define natural semantics for inclusion and exclusion quantifiers and then show that we can express them in inclusion-exclusion logic. We also show reversely that, by extending first order logic with these quantifiers, we obtain an equivalent logic with inclusion-exclusion logic. However, there are still some small, yet intriguing, differences between these two approaches.

By using several of our new operators -- term value preserving disjunction and both existential and universal inclusion quantifiers -- we can introduce a novel method of \emph{relativization} for team semantics. This technique is introduced in Subsection~\ref{ssec: Relativization} and later, in Section \ref{sec: Examples}, we present further examples on how it can be applied. In Section \ref{sec: Examples} we also present some other concrete examples where we show how to use our translations and new operators to express some classical properties of models and teams in a rather straightforward way. 

\medskip

The structure of this paper is as follows: In Section \ref{sec: Preliminaries} we review team semantics for FO and define inclusion and exclusion logics. In Section \ref{sec: New operators} we define several useful operators for inclusion-exclusion logic -- such as inclusion and exclusion quantifiers and term value preserving disjunction. In Section~\ref{sec: Translations} we present our translations between $\INEX[k]$ and $\ESO[k]$, and in Section \ref{sec: Examples} we present some further examples. After the conclusion in Section \ref{sec: Conclusion}, there is an appendix where we present a single long and technical proof that has been omitted from the main text.


\newpage

\section{Preliminaries}\label{sec: Preliminaries}

In this section we will first define the syntax and the semantics for first order logic. Instead of the usual Tarski semantics we will present team semantics which turns out to be an essentially equivalent way of defining the truth in the first order case.
Then we present inclusion and exclusion logics, define team semantics for them and review some of their know properties.


\subsection{Syntax and team semantics for first order logic}


Let $\{v_i\,|\, i\in\mathbb{N}\}$ be a set of \emph{variables}. We use symbols $\{x,y,z,\dots\}$ to denote meta variables ranging over the set of variables. A \emph{vocabulary} $L$ is a set of \emph{relation symbols}~$R$, \emph{function symbols} $f$ and \emph{constant symbols} $c$.
We denote the set of $L$-terms by $\Tset$. If $\tuple{\vphantom\wedge\smash{t}}=t_1\dots t_k$ and $t_i\in\Tset$ for each $i\leq k$, we write $\tuple{\vphantom\wedge\smash{t}}\in\Tset$. The set of variables occurring in a term $t$ is denoted by $\vr(t)$. For a tuple $\tuple{\vphantom\wedge\smash{t}}=t_1\dots t_k$ of $L$-terms we write $\vr(\tuple{\vphantom\wedge\smash{t}}\,) := \vr(t_1)\cup\cdots\cup\vr(t_k)$.
Next we define the syntax for first order logic (FO):
\begin{definition}\label{def: FOset}
The language $\FOset$ is the smallest set $\mathcal{S}$ satisfying the following conditions:
\begin{itemize}
\item If $t_1,t_2\in\Tset$, then $t_1=t_2\in\mathcal{S}$ and $\neg t_1=t_2\in\mathcal{S}$.
\item If $\tuple{\vphantom\wedge\smash{t}}\in\Tset$ is a $k$-tuple and $R\in L$ is a $k$-ary relation symbol, \\ then $R\,\tuple{\vphantom\wedge\smash{t}}\in\mathcal{S}$ and $\neg R\,\tuple{\vphantom\wedge\smash{t}}\in\mathcal{S}$.
\item If $\varphi,\psi\in\mathcal{S}$, then $(\varphi\wedge\psi)\in\mathcal{S}$ and $(\varphi\vee\psi)\in\mathcal{S}$.
\item If $\varphi\in\mathcal{S}$ and $x$ is a variable, then $\Es x\,\varphi\in\mathcal{S}$ and $\Ae x\,\varphi\in\mathcal{S}$.
\end{itemize}
$\FOset$-formulas of the form $t_1=t_2$, $\neg t_1=t_2$, $R\,\tuple{\vphantom\wedge\smash{t}}$ and $\neg R\,\tuple{\vphantom\wedge\smash{t}}$ are called \emph{literals}. Note that we only allow formulas in the \emph{negation normal form}.
\end{definition}
We denote the set of subformulas of an $\FOset$-formula $\varphi$ by $\subf(\varphi)$, the set of variables occurring in $\varphi$ by $\vr(\varphi)$ and the set of free variables of $\varphi$ by $\fr(\varphi)$. 
If we have $\fr(\varphi)=\{x_1,\dots,x_n\}$, we can emphasize this by writing $\varphi$ as $\varphi(x_1\dots x_n)$.
\begin{remark}
When we say that $\tuple x$ is tuple of \emph{fresh variables} we mean that all variables in $\tuple x$ are distinct and not occur in the variables of any formulas or terms that we have mentioned in the assumptions.
\end{remark}

An \emph{$L$-model} $\mathcal{M}$ is a pair $(M,\mathcal{I})$, where the \emph{universe} $M$ is a nonempty set and the \emph{interpretation} $\mathcal{I}$ is a function defined in the vocabulary~$L$. The interpretation $\mathcal{I}$ maps constant symbols to elements in $M$, $k$-ary relation symbols to $k$-ary relations in $M$ and $k$-ary function symbols to functions $M^k\rightarrow M$. For all $k\in L$ we write $k^\mathcal{M}:=\mathcal{I}(k)$.

Let $\mathcal{M}=(M,\mathcal{I})$ be an $L$-model. An \emph{assignment} $s$ for $M$ is a function that is defined in some set of variables and ranges over the universe $M$. The domain of $s$ is denoted by $\dom(s)$. 
A \emph{team} $X$ for $M$ is any set of assignments for $M$ with a common domain, denoted by $\dom(X)$. In the literature usually only teams with finite domains have been considered, but for this paper there is no need to assume the domains of teams to be finite. If $X$ is a team for the universe of $\mathcal{M}$ we can also say that $X$ is a team for the model $\mathcal{M}$.
Note that we also allow the empty assignment $s=\emptyset$ and the empty team $X=\emptyset$. For the empty team we allow \emph{any of set variables} to be interpreted as its domain (this is practical for certain technical reasons). The empty team is not to be confused with the team $X=\{\emptyset\}$ which has a special role with $\FOset$-sentences.

Let $s$ be an assignment and $a\in M$. The assignment $s[a/x]$ is defined in $ \dom(s)\cup\{x\}$, and it maps the variable $x$ to $a$ and all other variables as the assignment $s$.  
Let $X$ be a team, $A\subseteq  M$ and $F:X\rightarrow \mathcal{P}(M)$. We write
\begin{align*}
	X[A/x] &:= \{s[a/x] \mid s\in X, \,a\in A\} \\
	X[F/x] &:= \{s[a/x] \mid s\in X, \,a\in F(s)\}.
\end{align*}

Next we generalize these notations for tuples of variables.
Let $s$ be an assignment, $\tuple x:= x_1\dots x_k$ a tuple of variables and $\tuple a := (a_1,\dots,a_k)\in M^k$. We use the notation $s[\tuple a/\tuple x\,] := s[a_1/x_1,\dots,a_k/x_k]$. For a team $X$, a set $A\subseteq M^k$ and a function $\mathcal{F}:X\rightarrow\mathcal{P}(M^k)$ we write
\begin{align*}
	X[A/\tuple x\,]  &:= \{s[\tuple a/\tuple x\,]\mid s\in X,\,\tuple a\in A\} \\
	X[\mathcal{F}/\tuple x\,]  &:= \{s[\tuple a/\tuple x\,]\mid s\in X,\,\tuple a\in\mathcal{F}(s)\}.
\end{align*}

Let $\mathcal{M}$ be an $L$-model, $s$ an assignment and $t\in\Tset$ s.t. $\vr(t)\subseteq\dom(s)$. The interpretation of $t$ with respect to $\mathcal{M}$ and $s$, $t^\mathcal{M}\langle s\rangle$, is denoted simply by $s(t)$. 
For a team $X$ and $t\in\Tset$ s.t. $\vr(t)\subseteq\dom(X)$ we write
$X(t) :=\{s(t) \mid s\in X\}$.
Let $\tuple{\vphantom\wedge\smash{t}}:= t_1\dots t_k$ be a tuple of $L$-terms and let $X$ be a team s.t. $\vr(\tuple{\vphantom\wedge\smash{t}\,})\subseteq\dom(X)$. We write
\begin{align*}
	s(\tuple{\vphantom\wedge\smash{t}}\,) := \,(s(t_1),\dots,s(t_k)) \quad\text{ and }\quad
	X(\tuple{\vphantom\wedge\smash{t}}\,) := \{s(\tuple{\vphantom\wedge\smash{t}}\,)\mid s\in X\}.
\end{align*}
Note that $s(\tuple{\vphantom\wedge\smash{t}}\,)$ is a vector in $M$ and $X(\tuple{\vphantom\wedge\smash{t}}\,)$ is a $k$-ary relation in $M$.  
We use the notation $\mathcal{P}^*(A)$ for the power set of $A$ excluding the empty set (that is $\mathcal{P}^*(A) := \mathcal{P}(A)\,\setminus\,\{\emptyset\}$). 
We are now ready to define \emph{team semantics} for FO.

\begin{definition}\label{def: Team semantics}
Let $\mathcal{M}$ be an $L$-model, $\varphi\in\FOset$ and $X$ a team such that $\fr(\varphi)\subseteq\dom(X)$. We define the \emph{truth of $\varphi$ in $\mathcal{M}$ and~$X$}, denoted by $\mathcal{M}\true_X\varphi$:
\begin{itemize}[leftmargin=*]
\item $\mathcal{M}\true_X\ t_1\!=\!t_2$\, iff \,$s(t_1)=s(t_2)$ for all $s\in X$.
\item $\mathcal{M}\true_X \neg t_1\!=\!t_2$\, iff \,$s(t_1)\neq s(t_2)$ for all $s\in X$.
\item $\mathcal{M}\true_X R\,\tuple{\vphantom\wedge\smash{t}}$\, iff \,$s(\tuple{\vphantom\wedge\smash{t}}\,)\in R^{\mathcal{M}}$ for all $s\in X$.
\item $\mathcal{M}\true_X \neg R\,\tuple{\vphantom\wedge\smash{t}}$\, iff \,$s(\tuple{\vphantom\wedge\smash{t}}\,)\notin R^{\mathcal{M}}$ for all $s\in X$.
\item $\mathcal{M}\true_X \psi\wedge\theta$\, iff \,$\mathcal{M}\true_X \psi$ and $\mathcal{M}\true_X \theta$.
\item $\mathcal{M}\true_X \psi\vee\theta$\, iff \,there are $Y,Y'\subseteq X$ s.t. $Y\cup Y'=X$, 
$\mathcal{M}\true_Y \psi$ and $\mathcal{M}\true_{Y'} \theta$.
\item $\mathcal{M}\true_X\Es x\,\psi$\, iff \,there is $F:X\rightarrow\mathcal{P}^*(M)$
	 such that $\mathcal{M}\true_{X[F/x]}\psi$.
\item $\mathcal{M}\true_X\Ae x\,\psi$\, iff \,$\mathcal{M}\true_{X[M/x]}\psi$.
\end{itemize}
\end{definition}
\begin{remark}
In the truth definition above we introduced so-called \emph{lax semantics} for existential quantifier. In this definition the quantified variable can be given several witnesses. From the perspective of game-theoretic semantics this can be interpreted as the verifying player having \emph{a non-deterministic strategy} when choosing a value for the quantified variable (\cite{Galliani12a}). 
An alternative semantics, so-called \emph{strict semantics}, is to allow only a single witness for each assignment. In first the order case these two truth definitions are equivalent\footnote{Also note that, in the general case, the lax version is not stronger since we can always turn a strict quantifier into the corresponding lax quantifier by adding a ``dummy'' universal quantifier before it in the formula. That is, if $z$ is a fresh variable, then the formula $\Ee x\,\varphi$ has same truth condition with lax semantics as the formula $\Ae z\Ee x\,\varphi$ with strict semantics.} (\cite{Galliani12b}), but this does not hold when we extend FO with inclusion atoms.
\end{remark}
By their definitions, conjunction and disjunction are both associative, and for $\FOset$-formulas $\varphi_i$, $i\in\{1,\dots,n\}$, we have
\begin{align*}
	&\mathcal{M}\true_X\bigwedge_{i\leq n}\varphi_i\, 
	\text{ \,iff\, } \mathcal{M}\true_X\varphi_i \,\text{ for each } i\leq n. \\
	&\mathcal{M}\true_X\bigvee_{i\leq n}\varphi_i\, 
	\text{ \,iff\, there exist } Y_1,\dots,Y_n\subseteq X \text{ such that } \bigcup_{i\leq n} Y_i=X \\[-0,1cm] 
	&\hspace{7cm} \text{ and } \mathcal{M}\true_{Y_i}\varphi_i \,\text{ for each } i\leq n.
\end{align*}
For tuples $\tuple{\vphantom\wedge\smash{t}}:= t_1\dots t_k$ and $\tuple{t'}:= t_1'\dots t_k'$ of $L$-terms we write $\tuple{\vphantom\wedge\smash{t}} \!=\! \tuple{t'} := \bigwedge_{i\leq k} t_i\!=\!t_i'$ and $\tuple{\vphantom\wedge\smash{t}}\!\neq\!\tuple{t'} := \bigvee_{i\leq k}\neg t_i\!=\!t_i'$.
It is easy to see that the following equivalences hold.
\begin{align*}
	\mathcal{M}\true_X\tuple{\vphantom\wedge\smash{t}} = \tuple{t'}\, 
	&\;\text{ iff }\; s(\tuple{\vphantom\wedge\smash{t}}\,) = s(\tuple{t'}) \text{ for all } s\in X \\
	\mathcal{M}\true_X\tuple{\vphantom\wedge\smash{t}} \neq \tuple{t'}\, 
	&\;\text{ iff }\; s(\tuple{\vphantom\wedge\smash{t}}\,) \neq s(\tuple{t'}) \text{ for all } s\in X. 
\end{align*}

For $\varphi\in\FOset$ and $\tuple x:= x_1\dots x_k$, we write $\Es\tuple x\,\varphi := \Es x_1\dots\Es x_k\varphi$ and $\Ae\tuple x\,\varphi := \Ae x_1\dots\Ae x_k\varphi$.
By Definition \ref{def: Team semantics}, consecutive quantifications modify the team after the evaluation of each quantifier. Nevertheless, as shown by the following easy proposition, it is equivalent to quantify several elements in $M$ one after another and to quantify a single vector in $M$.

\begin{proposition}\label{the: Quantification of vectors}
For any $k$-tuple $\tuple x$ and $\varphi\in\FOset$ we have
\begin{enumerate}
\item[a)] $\mathcal{M}\true_X\Es\tuple x\,\varphi$\, iff there exists $\mathcal{F}:X\rightarrow\mathcal{P}^*(M^k)$ 
	 such that $\mathcal{M}\true_{X[\mathcal{F}/\tuple{x}\,]}\varphi$.
\item[b)] $\mathcal{M}\true_X\Ae\tuple x\,\varphi$\, iff $\mathcal{M}\true_{X[M^k/\tuple{x}\,]}\varphi$.
\end{enumerate}
\end{proposition}
\noindent
Note that with lax semantics for existential quantifier, when we quantify a $k$-tuple of variables, we can actually quantify a $k$-ary relation in $M$. 

First order logic with team semantics has so-called \emph{flatness}-property:

\begin{proposition}[\cite{Vaananen07}, Flatness]\label{the: Flatness}
Let $X$ be a team and $\varphi\in\FOset$. Then
\[
	\mathcal{M}\true_X\varphi\, \text{ iff }\, \mathcal{M}\true_{\{s\}} \varphi \text{ for all } s\in X.
\]
\end{proposition}
\noindent 
We write $\true_s^\text{T}$ and $\true^\text{T}$ for truth with the standard Tarski semantics. The following proposition shows how team semantics is related to Tarski semantics.

\begin{proposition}[\cite{Vaananen07}]\label{the: Tarski}
Let $\varphi\in\FOset$ and let $s$ be an assignment. Then for all $\FOset$-formulas we have
$\mathcal{M}\true_s^\text{\emph{T}}\varphi\, \text{ iff }\, \mathcal{M}\true_{\{s\}}\varphi$.
In particular, for all $\FOset$-sentences, we have
$\mathcal{M}\true^\text{\emph{T}}\!\varphi$ \,iff\, $\mathcal{M}\true_{\{\emptyset\}}\varphi$.
\end{proposition}
\noindent
Note that, by flatness, $\mathcal{M}\true_X\varphi$ iff $\mathcal{M}\true_s^\text{T}\varphi$ for all $s\in X$. In this sense we can say that team semantics for FO is a generalization of Tarski semantics.

By Proposition \ref{the: Tarski} it is natural to write $\mathcal{M}\true\varphi$ when we mean $\mathcal{M}\true_{\{\emptyset\}}\varphi$. Note that $\mathcal{M}\true_\emptyset\varphi$ holds trivially for all $\FOset$-formulas $\varphi$ by Definition \ref{def: Team semantics}. In general we say that any logic $\mathcal{L}$ with team semantics has \emph{empty team property} if $\mathcal{M}\true_\emptyset\varphi$ holds for all $\mathcal{L}$-formulas $\varphi$. 

We say that a logic $\mathcal{L}$ is \emph{local} if the truth of formulas is determined only by the values of the free variables in a team, i.e. the following holds for all $\mathcal{L}$-formulas $\varphi$.
\[
	\mathcal{M}\true_X\varphi\, \;\text{ iff }\; \mathcal{M}\true_{X\upharpoonright\fr(\varphi)} \varphi,
\]
where $X\upharpoonright\fr(\varphi):=\{s\!\upharpoonright\!\fr(\varphi)\mid s\in X\}$ and $s\!\upharpoonright\!\fr(\varphi)$ is an assignment such that $\dom(s\!\upharpoonright\!\fr(\varphi))=\fr(\varphi)$ and $(s\!\upharpoonright\!\fr(\varphi))(x)=s(x)$ for each $x\in\fr(\varphi)$.
FO is clearly local by Propositions \ref{the: Flatness} and \ref{the: Tarski}. 
Also note that if a logic $\mathcal{L}$ is local and has empty team property, then the following holds for all $\mathcal{L}$-sentences:
\[
	\mathcal{M}\true\varphi \;\text{ iff } \mathcal{M}\true_X\varphi \text{ for all teams } X.
\]
We define two more important properties for any logic $\mathcal{L}$ with team semantics.
\begin{definition}
Let $\mathcal{L}$ be a logic with team semantics. We say that
\begin{itemize}
\item $\mathcal{L}$ is \emph{closed downwards} if the following implication holds:
\[
	\text{If } \mathcal{M}\true_X\varphi \text{ and } Y\subseteq X, \text{ then } \mathcal{M}\true_Y\varphi. \hspace{1,32cm}
\]
\item $\mathcal{L}$ is \emph{closed under unions} if the following implication holds:
\[
	\text{If } \mathcal{M}\true_{X_i}\varphi \text{ for every } i\in I, \text{ then } \mathcal{M}\true_{\cup_{i\in I}X_i}\varphi.
\]
\end{itemize}
\end{definition}
\noindent
By flatness, FO is both closed both downwards and under unions.


\subsection{Inclusion and exclusion logics}\label{ssec: Inclusion and exclusion logics}


Inclusion and exclusion logics are obtained by adding inclusion and exclusion atoms, respectively, to FO with team semantics. By allowing the use of the both of these atoms we get inclusion-exclusion logic which is our main topic of interest in this paper. We first present the syntax and semantics for inclusion logic (INC).

\begin{definition}\label{def: INCset}
If $\tuple{\vphantom\wedge\smash{t}}_1,\tuple{\vphantom\wedge\smash{t}}_2$ are $k$-tuples of $L$-terms, $\tuple{\vphantom\wedge\smash{t}}_1\inc\tuple{\vphantom\wedge\smash{t}}_2$ is a \emph{$k$-ary inclusion atom}. We define $\fr(\tuple{\vphantom\wedge\smash{t}}_1\inc\tuple{\vphantom\wedge\smash{t}}_2)=\vr(\tuple{\vphantom\wedge\smash{t}}_1)\cup\vr(\tuple{\vphantom\wedge\smash{t}}_2)$. The language $\INCset$ is defined by adding the following condition to the definition of $\FOset$ (Definition \ref{def: FOset}).
\begin{itemize}
\item If $\tuple{\vphantom\wedge\smash{t}}_1,\tuple{\vphantom\wedge\smash{t}}_2$ are tuples of $L$-terms of the same length, then $\tuple{\vphantom\wedge\smash{t}}_1\subseteq\tuple{\vphantom\wedge\smash{t}}_2\in\mathcal{S}$.
\end{itemize}
\end{definition}
\noindent
Note that we do not allow negation to appear in front of inclusion atoms. For literals, connectives and quantifiers we use the same semantics as for FO with team semantics. Inclusion atoms have the following truth condition:
\begin{definition}\label{def: Inclusion atom}
Let $\mathcal{M}$ be a model and $X$ a team s.t. $\vr(\tuple{\vphantom\wedge\smash{t}}_1\tuple{\vphantom\wedge\smash{t}}_2)\subseteq\dom(X)$. We define the truth of $\tuple{\vphantom\wedge\smash{t}}_1\subseteq\tuple{\vphantom\wedge\smash{t}}_2$ in the model $\mathcal{M}$ and the team $X$:
\vspace{-0,1cm}\[
	\mathcal{M}\true_X \tuple{\vphantom\wedge\smash{t}}_1\subseteq\tuple{\vphantom\wedge\smash{t}}_2\;\, 
	\text{ iff \,for all } s\in X \text{ there exists } s'\in X 
	\text{ s.t. } s(\tuple{\vphantom\wedge\smash{t}}_1)=s'(\tuple{\vphantom\wedge\smash{t}}_2).
\vspace{-0,1cm}\]
This truth condition can be written equivalently as follows:
\vspace{-0,1cm}\[
	\mathcal{M}\true_X \tuple{\vphantom\wedge\smash{t}}_1\subseteq\tuple{\vphantom\wedge\smash{t}}_2\;\; 
	\text{ iff }\; X(\tuple{\vphantom\wedge\smash{t}}_1)\subseteq X(\tuple{\vphantom\wedge\smash{t}}_2).
\vspace{-0,1cm}\]
\end{definition}

\begin{example}\label{ex: Full relation}
Let $\tuple{\vphantom\wedge\smash{t}}_1,\dots,\tuple{\vphantom\wedge\smash{t}}_m$ be $k$-tuples of $L$-terms and $\tuple x$ a $k$-tuple of fresh variables. Now the following holds for all \emph{nonempty} teams $X$:
\[
	\mathcal{M}\true_X\Ae\tuple x\,\bigl(\bigvee_{i\leq m}\tuple x\inc\tuple{\vphantom\wedge\smash{t}}_i\bigr) 
	\;\text{ iff }\, \bigcup_{i\leq m} X(\tuple{\vphantom\wedge\smash{t}}_i) = M^k.
\]
In particular, for $t\in\Tset$ and $X\neq\emptyset$ we have $\mathcal{M}\true_X \Ae x\,(x\inc t)$ iff $X(t) = M$.
Note that this property is not closed downwards and thus it cannot be expressed in dependence logic (which is closed downwards as shown in \cite{Vaananen07}).
\end{example}
\noindent 
Next we present the syntax and semantics for exclusion logic (EXC).

\begin{definition}\label{def: EXCset}
If $\tuple{\vphantom\wedge\smash{t}}_1,\tuple{\vphantom\wedge\smash{t}}_2$ are $k$-tuples of $L$-terms, $\tuple{\vphantom\wedge\smash{t}}_1\,|\,\tuple{\vphantom\wedge\smash{t}}_2$ is a \emph{$k$-ary exclusion atom}. We define $\fr(\tuple{\vphantom\wedge\smash{t}}_1\exc\tuple{\vphantom\wedge\smash{t}}_2)=\vr(\tuple{\vphantom\wedge\smash{t}}_1)\cup\vr(\tuple{\vphantom\wedge\smash{t}}_2)$. The language $\EXCset$ is defined by adding the following condition to Definition~\ref{def: FOset}.
\begin{itemize}
\item If $\tuple{\vphantom\wedge\smash{t}}_1,\tuple{\vphantom\wedge\smash{t}}_2$ are tuples of $L$-terms of the same length, then $\tuple{\vphantom\wedge\smash{t}}_1\mid\tuple{\vphantom\wedge\smash{t}}_2\in\mathcal{S}$.
\end{itemize}
\end{definition}

\begin{definition}\label{def: Exclusion atom}
Let $\mathcal{M}$ be a model and $X$ a team s.t. $\vr(\tuple{\vphantom\wedge\smash{t}}_1\tuple{\vphantom\wedge\smash{t}}_2)\subseteq\dom(X)$. We define the truth of $\tuple{\vphantom\wedge\smash{t}}_1\mid\tuple{\vphantom\wedge\smash{t}}_2$ in the model $\mathcal{M}$ and the team $X$:
\begin{align*}
	\mathcal{M}\true_X \tuple{\vphantom\wedge\smash{t}}_1\mid\tuple{\vphantom\wedge\smash{t}}_2\;\, \text{ iff \,for all } s,s'\in X: \, s(\tuple{\vphantom\wedge\smash{t}}_1)\neq s'(\tuple{\vphantom\wedge\smash{t}}_2).
\end{align*}
This truth condition can be written equivalently as follows:
\[
	\mathcal{M}\true_X \tuple{\vphantom\wedge\smash{t}}_1\mid\tuple{\vphantom\wedge\smash{t}}_2\;\; \text{ iff }\; X(\tuple{\vphantom\wedge\smash{t}}_1)\cap X(\tuple{\vphantom\wedge\smash{t}}_2) = \emptyset.
\]
\end{definition}

Inclusion-exclusion logic (INEX) is defined simply by combining inclusion and exclusion logics:
\begin{definition}\label{def: INEXset}
Language $\INEXset$ is defined by adding both inclusion and exclusion atoms to first order logic.
\end{definition}

INC and EXC have both been shown local\footnote{Exclusion logic has been shown equivalent with dependence logic (\cite{Galliani12b}) which is known to be local (\cite{Vaananen07}). Inclusion logic has been shown local by Galliani \cite{Galliani12b}, but for this proof the lax semantics is required. With strict semantics the locality of $\INC$ is lost, which is one of the reasons why the lax semantics is considered to be a more natural choice to be used in team semantics. Inclusion logic with strict semantics has also been studied (see for example \cite{Hannula15}).}. By the truth definitions of inclusion and exclusion atoms, it is easy to see that INC and EXC both satisfy empty team property. Hence also INEX satisfies these properties. 
Neither inclusion nor exclusion logic has flatness-property. Galliani \cite{Galliani12b} has shown that INC is closed under unions, but not downwards. On the other hand, EXC is closed downwards but not under unions (\cite{Vaananen07}). Hence INEX is not closed downwards nor under unions.

In this paper we are particularly interested in the effect of arity of atoms with respect to the expressive power. For this purpose we define \emph{$k$-ary fragments} of these logics.

\begin{definition}\label{def: $k$-ary logics}
If $\varphi\in\INEXset$ contains at most $k$-ary inclusion and exclusion atoms, we say that $\varphi$ is an $\INEXset[k]$-formula. By allowing only the use of these formulas, we obtain \emph{$k$-ary inclusion-exclusion logic}, denoted by INEX[$k$]. Furthermore, \emph{$k$-ary inclusion logic} (INC[$k$]) and \emph{$k$-ary exclusion logic} (EXC[$k$]) are defined analogously.
\end{definition}

Note that the exclusion atom $\tuple{\vphantom\wedge\smash{t}}_1\mid \tuple{\vphantom\wedge\smash{t}}_2$ is not the \emph{contradictory negation} of the inclusion atom $\tuple{\vphantom\wedge\smash{t}}_1\!\subseteq\!\tuple{\vphantom\wedge\smash{t}}_2 $, and that the former is symmetric while the latter is not (that is, $\tuple{\vphantom\wedge\smash{t}}_1\mid \tuple{\vphantom\wedge\smash{t}}_2\equiv\tuple{\vphantom\wedge\smash{t}}_2\mid \tuple{\vphantom\wedge\smash{t}}_1$ but $\tuple{\vphantom\wedge\smash{t}}_1\inc \tuple{\vphantom\wedge\smash{t}}_2\not\equiv\tuple{\vphantom\wedge\smash{t}}_2\inc \tuple{\vphantom\wedge\smash{t}}_1$). 
%
%
The contradictory negations of $k$-ary inclusion and exclusion atoms can be defined in $\INEX[k]$ for \emph{nonempty} teams, as shown by the following example.

\begin{example}
Let $\mathcal{M}$ be a model, $X$ a nonempty team, $\tuple{\vphantom\wedge\smash{t}}_1,\tuple{\vphantom\wedge\smash{t}}_2\in\Tset$ $k$-tuples and $\tuple x$ a $k$-tuple of variables. It is easy to see that we have
\begin{align*}
	\mathcal{M}\ntrue_X\tuple{\vphantom\wedge\smash{t}}_1\mid\tuple{\vphantom\wedge\smash{t}}_2\, &\,\text{ iff }\, 
	\mathcal{M}\true_X\Ee \tuple x\,(\tuple x\subseteq\tuple{\vphantom\wedge\smash{t}}_1\wedge\tuple x\subseteq\tuple{\vphantom\wedge\smash{t}}_2) \\
	\mathcal{M}\ntrue_X\tuple{\vphantom\wedge\smash{t}}_1\subseteq\tuple{\vphantom\wedge\smash{t}}_2\, &\,\text{ iff }\, 
	\mathcal{M}\true_X\Ee \tuple x\,(\tuple x\subseteq\tuple{\vphantom\wedge\smash{t}}_1\wedge\tuple x\mid\tuple{\vphantom\wedge\smash{t}}_2).
\end{align*}
\end{example}

If we would use \emph{negated} inclusion/exclusion atoms with the semantics of the contradictory negation in $\INEX$, we would lose empty team property since the contradictory negations of these atoms are false in the empty team. But for nonempty teams, this extension would not give us any more expressive power.

\begin{observation}\label{obs: contradictory negation}
In team semantics contradictory negation is not equivalent with the negation $\neg$ that is used with literals. This is because, if $\varphi$ is of the form $\tuple{\vphantom\wedge\smash{t}}_1=\tuple{\vphantom\wedge\smash{t}}_2$ or $R\,\tuple{\vphantom\wedge\smash{t}}$, the claims $\mathcal{M}\ntrue_X\varphi$ and $\mathcal{M}\true_X\neg\varphi$ are not necessarily equivalent when $\abs X>1$.
Since inclusion and exclusion atoms are atomic formulas as (non-negated) literals, their \emph{negations} should behave similarly as the negations of literals.
Therefore, if we would define negated inclusion or exclusion atoms, the semantics of contradictory negation would not be a natural choice for it. We will discuss further the issue of sensible semantics for negated atoms in the end of of Section~\ref{sec: Translations}.
\end{observation}


\section{Defining new operators for INEX}\label{sec: New operators}

In this section we will define several useful operators for INEX[$k$]. First we will define \emph{constancy atoms} and \emph{intuitionistic disjunction}. Then we will introduce \emph{inclusion and exclusion quantifiers} which present a new approach to inclusion and exclusion dependencies. 
Then we define a new operator called \emph{term value preserving disjunction} which will be essential for our translation from $\ESO[k]$ to $\INEX[k]$ in the next section. Finally we will introduce a method called \emph{relativization} that is an application which uses several of the new operators defined in this section.


\subsection{Constancy atoms and intuitionistic disjunction}


\emph{Constancy atom} $\dep(t)$ is a unary dependence atom (\cite{Vaananen07}). It simply says that the term $t$ has a constant value in a (nonempty) team. Galliani \cite{Galliani12b} has shown that this atom can be expressed by using unary exclusion atom. Thus we can define this atom as an abbreviation in $\INEXset[k]$ for any $k\geq 1$.

\begin{definition}[\cite{Galliani12b}]\label{def: Constancy atom}
Let $t\in\Tset$ and $x$ a fresh variable. We define \emph{constancy atom} $\dep(t)$, as an abbreviation, as follows: \;
$\dep(t) \;:=\; \Ae x\left(x = t\,\vee\, x \mid t\,\right)$.
\end{definition}

\begin{proposition}[\cite{Galliani12b}]
With the assumptions of the previous definition, we obtain the following truth condition: \, $\mathcal{M}\true_X\dep(t) \; \text{ iff } \; \abs{X(t)}=1 \text{ or } X=\emptyset$.
\end{proposition}

\emph{Intuitionistic disjunction} $\sqcup$ is obtained by lifting the Tarski semantics of disjunction from single assignments to teams. That is, $\varphi\sqcup\psi$ is true in a team~$X$ if either $\varphi$ or $\psi$ is true in $X$. Galliani \cite{Galliani12a} has shown that this operator can be expressed with constancy atoms in any logic with empty team property. We will define this operator in INEX here in the same way -- with the addition of the special case of single element models.

\begin{definition}[\cite{Galliani12a}]\label{def: Intuitionistic disjunction}
Let $\varphi,\psi\in\INEXset$. We define \emph{intuitionistic disjunction} $\varphi\sqcup\psi$, as an abbreviation, in the following way:
\begin{align*}
	\varphi\sqcup\psi := \;\bigl(\gamma_{=1}\wedge(\varphi\vee\psi)\bigr)
	\vee\Es z_1\Es z_2&\,\big(\dep(z_1)\wedge\dep(z_2) \\[-0,15cm]
	&\quad\wedge((z_1=z_2\wedge\varphi)\vee(z_1\neq z_2\wedge\psi))\bigr),
\end{align*}
where $z_1,z_2$ are fresh variables and $\gamma_{=1}$ is a shorthand for $\Ae z_1\Ae z_2\,(z_1\!=\!z_2)$.
\end{definition}

\begin{proposition}[\cite{Galliani12a}]\label{the: Intuitionistic disjunction}
With the assumptions of the previous definition, we obtain the following truth condition: \,
$\mathcal{M}\true_X \varphi\sqcup\psi \; \text{ iff } \; \mathcal{M}\true_X\varphi \text{ or } \mathcal{M}\true_X\psi$.
\end{proposition}

The idea for defining intuitionistic disjunction in this way is to require that the splitting of a team $X$ must be done in a way that either of the sides becomes empty, whence the other side must be the whole initial team $X$. When a logic has empty team property, then the requirement, that the splitting must be done in this way, becomes equivalent with the truth definition above. 
But note that if our logic would not have empty team property, then we could not define intuitionistic disjunction by using this approach.


\subsection{Inclusion and exclusion quantifiers}\label{ssec: Quantifiers}


Here we will consider inclusion and exclusion relations from a new perspective. Instead of having atomic formulas that express them, we embed these relations to the truth conditions of quantifiers. By this approach, we are aiming to obtain a logic that has similar relationship with INEX, as there is between IF-logic and dependence logic. 
We will define inclusion and exclusion versions for both existential and universal quantifiers. We will also show that we can express them by using inclusion and exclusion atoms, and thus use them freely as abbreviations in INEX. Before giving the actual definitions, we first consider what kind of semantics would be intuitive for such operators. 

In independence friendly logic we can use so-called IF-quantifiers which state that the values given for a quantified variable are independent of the values of certain other variables. It would be essentially equivalent to define ``dependence friendly'' quantifiers (\cite{Vaananen07}) which state that the values for a quantified variable is allowed to depend only on a certain set of variables. Dependence atoms of dependence logic (\cite{Vaananen07}) state a the same property about the values of variables in a team on an atomic level.
We take a reverse approach here: Instead of stating that inclusion or exclusion relation holds for certain variables in a team, we say that inclusion or exclusion holds for a certain variable when it is quantified. Syntactically this would give us quantifiers of the form $(\Ee x\inc y)$, $(\Ae x\inc y)$, $(\Ee x\exc y)$ and $(\Ae x\exc y)$.
\begin{remark}
Since inclusion is not a symmetric relation, one could also consider semantics for quantifiers of the form $(\Ee x\!\supseteq\! y)$ and $(\Ae x\!\supseteq\! y)$. 
For the first one of these we see at least two non-equivalent natural semantical approaches, but the meaning of the latter one seems to become trivial.
This question is not examined further in this paper, but a reader is encouraged to consider intuitive semantics for such quantifiers after reading this section. 
\end{remark}
Before considering natural semantics for these quantifiers, we introduce so-called \emph{storing operator} that is needed in the definitions later. The idea for it is simply that we copy the values of a given tuple $\tuple{\vphantom\wedge\smash{t}}$ of terms into a given tuple $\tuple u$ of variables. This way it is possible to refer to the old values of $\tuple{\vphantom\wedge\smash{t}}$, even if they change later in (re)quantifications.
\begin{definition}\label{def: Recording formula}
Let $\varphi\in\INEXset$, $\tuple{\vphantom\wedge\smash{t}}\in\Tset$ a $k$-tuple and $\tuple u$ a $k$-tuple of variables. The \emph{$\tuple{\vphantom\wedge\smash{t}}$ to $\tuple u$ storing operator}, $\rec{\tuple{\vphantom\wedge\smash{t}}}{\tuple u}$, is defined as:
\[
	\rec{\tuple{\vphantom\wedge\smash{t}}}{\tuple u}\varphi := 
	\Ee\tuple u\,(\tuple u = \tuple{\vphantom\wedge\smash{t}}\,\wedge\,\varphi).
\]
\end{definition}
For this operator to work as desired, we need to set a requirement that the variables in the tuple $\tuple u$ do not occur in the tuple $\tuple{\vphantom\wedge\smash{t}}$. However, naturally we must allow the variables in the tuple $\tuple u$ to be free variables in the formula $\varphi$. 

The following lemma for storing operator is obvious.

\begin{lemma}\label{the: Storing operator}
Let $\varphi\in\INEXset$, $\tuple{\vphantom\wedge\smash{t}}\in\Tset$ a $k$-tuple and let $\tuple u$ be a $k$-tuple of variables that are not in $\vr(\tuple{\vphantom\wedge\smash{t}}\,)$. Let $X' := \{s[s(\tuple{\vphantom\wedge\smash{t}}\,)/\tuple u\,]\mid s\in X\}$. Now we have $X(\tuple{\vphantom\wedge\smash{t}}\,)=X'(\tuple u\,)$ and the following holds: \, $\mathcal{M}\true_X\rec{\tuple{\vphantom\wedge\smash{t}}}{\tuple u}\varphi \,\text{ iff }\, \mathcal{M}\true_{X'}\varphi$.
\end{lemma}

We are now ready to start defining inclusion and exclusion quantifiers.


\subsubsection*{Existential inclusion and exclusion quantifiers}


We begin by defining semantics for existential inclusion and exclusion quantifiers $(\Ee x\inc y)$ and $(\Ee x\exc y)$. We take here a slightly more general approach by allowing the variable $y$ to be any $L$-term $t$. 
A natural reading for \emph{existential inclusion quantifier} $(\Ee x\!\subseteq\!t)$ is that ``there exists an $x$ within the values of $t$''. This kind of truth condition can be achieved simply by modifying the standard truth condition of existential quantifier in such a way that the values given by the choice function $F$ are restricted to the values of $t$ in a team $X$. We then obtain the following truth condition:
\begin{align*}
	\mathcal{M}\true_X (\Ee x\subseteq t)\,\varphi\;
	\text{ iff there is } &F:X\rightarrow\mathcal{P}^*(X(t))
	\text{ s.t. } \mathcal{M}\true_{X[F/x]}\varphi.
\end{align*}
Another natural language interpretation for quantifier $(\Ee x\!\subseteq\!t)$ is that the values given for $x$ must be \emph{possible} for the term $t$. From the perspective of semantic games we may say that \emph{verifying} player's allowed moves are restricted on the set $X(t)$ instead of the whole universe of a model.\footnote{Note that the setting here is quite different than in IF-logic (or dependence logic). In IF-logic the verifying player is allowed to choose any values, but values for certain variables are ``hidden'' from him/her when making the choice. Here the player may see the values of all variables, but only certain values are admissible to be chosen. In the former case the \emph{domain} of the strategy function is restricted and in the  latter case only its \emph{range} is restricted.}

Similarly we read \emph{existential exclusion quantifier} $(\Ee x\exc t)$ as ``there exists an $x$ outside the values of $t$''. To achieve this, we simply restrict values given by the choice function~$F$ to the complement, $\overline{X(t)}=M\!\setminus\!X(t)$, of $X(t)$:
\begin{align*}
	\mathcal{M}\true_X (\Ee x\mid t)\,\varphi\;
	\text{ iff there is } &F:X\rightarrow\mathcal{P}^*\left(\overline{X(t)}\right)
	\text{ s.t. } \mathcal{M}\true_{X[F/x]}\varphi.
\end{align*}
This kind of quantification dually must give such values for $x$ that are \emph{not possible} for $t$. Or in a semantic game we can say that the values in the set $X(t)$ are ``banned'' from the verifier when (s)he chooses a value for $x$.  


Next we define these operators, as abbreviations, by using inclusion and exclusion atoms. We want their truth conditions to be as described above, but we give the definitions in a more general form by allowing the quantification of tuples instead of just single variables.

\begin{definition}\label{def: Existential quantification}
Let $\varphi\in\INEXset$, $\tuple{\vphantom\wedge\smash{t}}\in\Tset$ a $k$-tuple and let $\tuple x,\tuple u$ be $k$-tuples of variables s.t. the variables in $\tuple u$ are not  in $\vr(\tuple{\vphantom\wedge\smash{t}}\,)$. We use the following notations:
\begin{align*}
	(\Ee\tuple x\subseteq\tuple{\vphantom\wedge\smash{t}}\,)\,\varphi &:= \rec{\tuple{\vphantom\wedge\smash{t}}}{\tuple u}
	\Ee\tuple x\,(\tuple x\subseteq\tuple u\wedge\varphi) \\
	(\Ee\tuple x\mid\tuple{\vphantom\wedge\smash{t}}\,)\,\varphi &:= \rec{\tuple{\vphantom\wedge\smash{t}}}{\tuple u}
	\Ee\tuple x\,(\tuple x\mid\tuple u\wedge\varphi).
\end{align*}
\end{definition}
Note that the lengths of quantified tuples match the arities of atoms, i.e. if $\varphi\in\INEXset[k]$ and $\tuple x$, $\tuple{\vphantom\wedge\smash{t}}$ are $k$-tuples, then $(\Ee\tuple x\inc\tuple{\vphantom\wedge\smash{t}}\,)\,\varphi,(\Ee\tuple x\exc\tuple{\vphantom\wedge\smash{t}}\,)\,\varphi\in\INEX[k]$. Since only one type of atom is needed for each quantifier, $(\Ee\tuple x\inc\tuple{\vphantom\wedge\smash{t}}\,)\,\varphi\in\INCset[k]$ when $\varphi\in\INCset[k]$ and $(\Ee\tuple x\exc\tuple{\vphantom\wedge\smash{t}}\,)\,\varphi\in\EXC[k]$ when $\varphi\in\EXCset[k]$. Also note that $\fr((\Ee\tuple x\inc\tuple{\vphantom\wedge\smash{t}}\,)\,\varphi)=(\fr(\varphi)\setminus\{x\})\cup\vr(\tuple{\vphantom\wedge\smash{t}})=\fr((\Ee\tuple x\exc\tuple{\vphantom\wedge\smash{t}}\,)\,\varphi)$.

The next proposition presents the truth conditions given by Definition~\ref{def: Existential quantification}. 
This result might seem quite obvious, since the definition is so straightforward, but we nevertheless we present a proof here with all the technical details -- also considering the use of storing operator $\rec{\tuple{\vphantom\wedge\smash{t}}}{\tuple u}$.
In the proof of the next proposition, and from now on, we will write $\im(F):=\{F(s)\mid s\in X\}$ for any function~$F$ that is defined in some team $X$.

\begin{proposition}\label{the: Existential quantification}
With the same assumption as in Definition~\ref{def: Existential quantification}, we obtain the following truth conditions:
\begin{enumerate}
\item[a)] $\mathcal{M}\true_X (\Ee\tuple x\subseteq\tuple{\vphantom\wedge\smash{t}}\,)\,\varphi$\, iff there is $\mathcal{F}:X\rightarrow\mathcal{P}^*\left(X(\tuple{\vphantom\wedge\smash{t}}\,)\right)$ s.t. $\mathcal{M}\true_{X[\mathcal{F}/\tuple x\,]}\varphi$.
\item[b)] $\mathcal{M}\true_X (\Ee\tuple x\mid\tuple{\vphantom\wedge\smash{t}}\,)\,\varphi$\, iff there is $\mathcal{F}:X\rightarrow\mathcal{P}^*\left(\,\overline{X(\tuple{\vphantom\wedge\smash{t}}\,)}\,\right)$ s.t. $\mathcal{M}\true_{X[\mathcal{F}/\tuple x\,]}\varphi$.
\end{enumerate}
\end{proposition}

\begin{proof}
By locality we may assume that the variables in $\tuple u$ are not in $\dom(X)$.
\medskip
\\
a) Suppose first that $\mathcal{\mathcal{M}}\true_X\,(\Ee\tuple x\inc\tuple{\vphantom\wedge\smash{t}}\,)\,\varphi$. Since $(\Ee\tuple x\inc\tuple{\vphantom\wedge\smash{t}}\,)\,\varphi=\rec{\tuple{\vphantom\wedge\smash{t}}}{\tuple u}\Ee\tuple x\,(\tuple x\!\subseteq\!\tuple u\wedge \varphi)$ by Lemma \ref{the: Storing operator} we have $\mathcal{M}\true_{X'}\Ee\tuple x\,(\tuple x\subseteq\tuple u\wedge \varphi)$, where $X'=\{s[s(\tuple{\vphantom\wedge\smash{t}}\,)/u]\mid s\in X\}$. Thus there exists $\mathcal{F}':X'\rightarrow\mathcal{P}^*(M^k)$ s.t. $\mathcal{M}\true_{X'[\mathcal{F}'/\tuple x\,]}\tuple x\subseteq\tuple u\wedge\varphi$. Let
\vspace{-0,1cm}\[
	\mathcal{F}:X\rightarrow\mathcal{P}^*(M), \;\; s\mapsto\mathcal{F}'(s[s(\tuple{\vphantom\wedge\smash{t}}\,)/\tuple u\,]).
\vspace{-0,1cm}\]
Since $\mathcal{M}\true_{X'[\mathcal{F'}/\tuple x\,]}\varphi$, by locality it is easy to see that $\mathcal{M}\true_{X[\mathcal{F}/\tuple x\,]}\varphi$. We still need to show that $\im(\mathcal{F})\subseteq\mathcal{P}(X(\tuple{\vphantom\wedge\smash{t}}\,))$. Since by Lemma \ref{the: Storing operator} we have $X'(\tuple u\,)=X(\tuple{\vphantom\wedge\smash{t}}\,)$, this amounts to showing that $\im(\mathcal{F}')\subseteq \mathcal{P}(X'(\tuple u\,))$:
Let $\mathcal{F'}(s)\in\im(\mathcal{F}')$ for some $s\in X'$ and let $\tuple a\in\mathcal{F'}(s)$. Let $r:=s[\tuple a/\tuple x\,]$ whence $r\in X'[\mathcal{F'}/\tuple x\,]$. Since $\mathcal{M}\true_{X'[\mathcal{F}'/\tuple x\,]}\tuple x\subseteq\tuple u$, there exists $r'\in X'[\mathcal{F}'/\tuple x\,]$ such that $r'(\tuple u)=r(\tuple x)$. Furthermore, $r'=s'[\tuple{\vphantom t\smash{b}}/\tuple x\,]$ for some $s'\in X'$ and $\tuple{\vphantom t\smash{b}}\in\mathcal{F}'(s')$. Now we have
\vspace{-0,1cm}\[
	\tuple a = s[\tuple a/\tuple x\,](\tuple x) = r(\tuple x) = r'(\tuple u) = s'(\tuple u) \in X'(\tuple u\,).
\vspace{-0,1cm}\]
Thus $\im(\mathcal{F}')\subseteq\mathcal{P}(X'(\tuple u\,))$, i.e. $\im(\mathcal{F})\subseteq\mathcal{P}(X(\tuple{\vphantom\wedge\smash{t}}\,))$.

\medskip

\noindent
Suppose then that there exists $\mathcal{F}:X\rightarrow\mathcal{P}^*(X(\tuple{\vphantom\wedge\smash{t}}\,))$ such that $\mathcal{M}\true_{X[\mathcal{F}/\tuple x\,]}\varphi$. Let $X'=\{s[s(\tuple{\vphantom\wedge\smash{t}}\,)/u]\mid s\in X\}$ and $\mathcal{F}':X'\rightarrow\mathcal{P}^*(M)$ such that $s\mapsto\mathcal{F}(\,s\!\upharpoonright\!\dom(X))$.

In order to show that $\mathcal{M}\true_{X'[\mathcal{F}'/\tuple x\,]}\tuple x\inc\tuple u$, let $r\in X'[\mathcal{F}'/\tuple x\,]$. Now there are $s\in X'$ and $\tuple a\in\mathcal{F}'(s)$ such that $r=s[\tuple a/\tuple x\,]$. Since $\im(\mathcal{F})\subseteq\mathcal{P}(X(\tuple{\vphantom\wedge\smash{t}}\,))$, by Lemma~\ref{the: Storing operator} also $\im(\mathcal{F}')\subseteq \mathcal{P}(X'(\tuple u\,))$. In particular, $\tuple a\in X'(\tuple u)$, and thus there exists $s'\in X'$ s.t. $s'(\tuple u)=\tuple a$. Let $\tuple{\vphantom t\smash{b}}\in\mathcal{F}'(s')$ and $r':=s'[\tuple{\vphantom t\smash{b}}/\tuple x\,]$. Now we have
\[
	r(\tuple x) = s[\tuple a/\tuple x\,](\tuple x) = \tuple a = s'(\tuple u) = r'(\tuple u).
\]
Thus $\mathcal{M}\true_{X'[\mathcal{F}'/\tuple x\,]}\tuple x\inc\tuple u$. Since $\mathcal{M}\true_{X[\mathcal{F}/\tuple x\,]}\varphi$, by locality $\mathcal{M}\true_{X'[\mathcal{F'}/\tuple x\,]}\varphi$. Hence $\mathcal{M}\true_{X'}\Ee\tuple x\,(\tuple x\subseteq\tuple u\wedge \varphi)$ and thus by Lemma \ref{the: Storing operator} we have $\mathcal{\mathcal{M}}\true_X(\Ee\tuple x\subseteq\tuple{\vphantom\wedge\smash{t}}\,)\,\varphi$.
\bigskip
\\
b) Suppose that $\mathcal{\mathcal{M}}\true_X\,(\Ee\tuple x\mid\tuple{\vphantom\wedge\smash{t}}\,)\,\varphi$. As in a), there exists $\mathcal{F}':X'\rightarrow\mathcal{P}^*(M^k)$ s.t. $\mathcal{M}\true_{X'[\mathcal{F}'/\tuple x\,]}\tuple x\exc\tuple u\wedge\varphi$. We can define the function $\mathcal{F}$ as in a), whence $\mathcal{M}\true_{X[\mathcal{F}/\tuple x\,]}\varphi$. Showing that $\im(\mathcal{F})\subseteq\mathcal{P}(\overline{X(\tuple{\vphantom\wedge\smash{t}}\,)})$ amounts to showing that $\im(\mathcal{F}')\subseteq\mathcal{P}(\overline{X'(\tuple u\,)})$:
For the sake of contradiction, suppose that there are $s\in X'$ and a tuple $\tuple a\in\mathcal{F}'(s)$ s.t. $\tuple a\in X'(\tuple u\,)$. Thus there exists $s'\in X'$ s.t. $s'(\tuple u) = \tuple a$. Let $r:=s[\tuple a/\tuple x\,]$ and $r':=s'[\tuple{\vphantom t\smash{b}}/\tuple x\,]$, where $\tuple{\vphantom t\smash{b}}\in\mathcal{F}(s')$. Now $r,r'\in X'[\mathcal{F}/\tuple x\,]$ and
$r(\tuple x) = s[\tuple a/\tuple x\,](\tuple x) = \tuple a = s'(\tuple u) = r'(\tuple u)$.
This is a contradiction since $\mathcal{M}\true_{X'[\mathcal{F}'/\tuple x]}\tuple x\!\mid\!\tuple u$, and thus $\im(\mathcal{F}')\subseteq\mathcal{P}(\overline{X'(\tuple u\,)})$.
\medskip
\\
Suppose then that there exists $\mathcal{F}:X\rightarrow\mathcal{P}^*(\overline{X(\tuple{\vphantom\wedge\smash{t}}\,)})$ s.t. $\mathcal{M}\true_{X[\mathcal{F}/\tuple x\,]}\varphi$. We can define $X'$ and $\mathcal{F}'$ as in a), whence $\mathcal{M}\true_{X'[\mathcal{F}'/\tuple x\,]}\varphi$ and $\im(\mathcal{F}')\subseteq\mathcal{P}(\overline{X'(\tuple u\,)})$.

In order to show that $\mathcal{M}\true_{X'[\mathcal{F}'/\tuple x\,]}\tuple x\mid\tuple u$, we suppose for the sake of contradiction that there exist $r,r'\in X'[\mathcal{F}'/\tuple x\,]$ such that $r(\tuple x)=r'(\tuple u)$. Now there exist $s,s'\in X'$, $\tuple a\in\mathcal{F}'(s)$ and $\tuple{\vphantom t\smash{b}}\in\mathcal{F}'(s')$, s.t. $r=s[\tuple a/\tuple x\,]$ and $r'=s'[\tuple{\vphantom t\smash{b}}/\tuple x\,]$. Now $\tuple a\in\mathcal{F'}(s)\in\im(\mathcal{F'})$, but also
\vspace{-0,1cm}\[
	\tuple a = s[\tuple a/\tuple x\,](\tuple x) = r(\tuple x) = r'(\tuple u) = s'(\tuple u) \in X'(\tuple u).
\vspace{-0,1cm}\]
This is a contradiction since $\im(\mathcal{F}')\subseteq\mathcal{P}(\overline{X'(\tuple u\,)})$. Hence $\mathcal{M}\true_{X'[\mathcal{F}'/\tuple x\,]}\tuple x\mid\tuple u$ and thus $\mathcal{M}\true_{X'}\Ee\tuple x\,(\tuple x\mid\tuple u\wedge \varphi)$. By Lemma \ref{the: Storing operator} we have $\mathcal{\mathcal{M}}\true_X(\Ee\tuple x\mid\tuple{\vphantom\wedge\smash{t}}\,)\,\varphi$.
\end{proof}

\begin{remark}
When defining these quantifiers, we did not want to put any restrictions on tuples $\tuple x,\tuple{\vphantom\wedge\smash{t}}$ and thus, in particular, we also allow the variables in $\tuple x$ to occur in $\tuple{\vphantom\wedge\smash{t}}$. This is why we need to use the storing operator, since the values of $\tuple{\vphantom\wedge\smash{t}}$ in a team might change after the quantification of $\tuple x$. 

If we would drop the storing operator from Definition~\ref{def: Existential quantification}, then the choice function $\mathcal{F}$ would be required to choose values within ($\subseteq$) or outside ($\;\mid\;$) the set $X[\mathcal{F}/\tuple x\,](\tuple{\vphantom\wedge\smash{t}}\,)$ instead of the set $X(\tuple{\vphantom\wedge\smash{t}}\,)$. Hence the values in the team \emph{after} the quantification would restrict the range of the choice function that is used for the quantification. This would lead to a very unnatural truth condition.
\end{remark}
Since quantifications may change the values of terms in a team, several identical consecutive existential inclusion/exclusion quantifications can change the meaning of a formula, as seen by the following example.  
\begin{example}
Assume that $c\in L$ is a constant symbol and $f\in L$ is a unary function symbol. We write
\vspace{-0,2cm}
\begin{align*}
	\varphi &:= \Es x\,(\Es x\inc fx)(x=c) \\
	\psi &:= \Es x\,(\Es x\inc fx)(\Es x\inc fx)(x=c). \\[-0,8cm]
\end{align*}
The sentences $\varphi$ and $\psi$ are not logically equivalent since we have
\vspace{-0,1cm}
\begin{align*}
	\mathcal{M}\true\varphi\, &\text{ iff } \mathcal{M}\true\Es x\,(fx=c), \text{ but }\\
	\mathcal{M}\true\psi\, &\text{ iff } \mathcal{M}\true\Es x\,(ffx=c).
\end{align*}
\end{example}
The following example presents a property that is not FO-definable, but can be expressed with a sentence containing a single existential inclusion quantifier. A similar example was presented originally for inclusion logic in \cite{Hella13}.
\begin{example}
A directed finite graph $\mathcal{G}=(V,E)$ contains a cycle if and only if the following holds:
\[
	\mathcal{G}\true \Ee x\,(\Ee y\inc x)Exy.
\]
\end{example}


\subsubsection*{Universal inclusion and exclusion quantifiers}


We define semantics for universal inclusion and exclusion quantifiers $(\Ae x\inc y)$ and $(\Ae x\exc y)$. Again we allow $y$ to be any $L$-term $t$ and first consider the semantics for these operators from an intuitive perspective. 
For \emph{universal inclusion quantifier} $(\Ae x\!\subseteq\!t)$ a natural reading would be: ``for all $x$ within the values of $t$''. This restricted universal quantification is done simply by quantifying $x$ over the set $X(t)$ instead of the whole universe $M$:
\begin{align*}
	\mathcal{M}\true_X (\Ae x\subseteq t)\,\varphi\,
	\;\,\text{ iff }\; \mathcal{M}\true_{X[A/x]}\varphi, \text{ where } A = X(t).
\end{align*}
A dualistic reading for \emph{universal exclusion quantifier} $(\Ae x\!\mid\!t)$ is ``for all $x$ outside the values of $t$''. This is achieved by quantifying $x$ over the \emph{complement} of $X(t)$:
\begin{align*}
	\mathcal{M}\true_X (\Ae x\,|\, t)\,\varphi\,
	\;\,\text{ iff }\; \mathcal{M}\true_{X[A/x]}\varphi, \text{ where } A = \overline{X(t)}.
\end{align*}
As with existential inclusion and exclusion quantifiers, we can observe the semantics above from a game-theoretic perspective by restricting the allowed moves of the players. This time, when choosing values for $x$, the \emph{falsifying} player may only choose the values of $t$ in the case of inclusion, and  the values of $t$ are forbidden from him/her in the case of exclusion.

Next we define these operators as abbreviations in INEX, aiming for the truth conditions as described above. Again we give the definitions in a more general form by using tuples instead of just single variables. The definitions here turn out to be much more complicated than the ones for existential quantifiers.

\begin{definition}\label{def: Universal quantification}
Let $\varphi\in\INEXset$, $\tuple{\vphantom\wedge\smash{t}}\in\Tset$ a $k$-tuple, $\tuple x$ a $k$-tuple of variables and $\tuple u$, $\tuple y$, $\tuple z$ \,$k$-tuples of fresh variables. We use the following notations:
\begin{align*}
	(\Ae\tuple x\subseteq\tuple{\vphantom\wedge\smash{t}}\,)\,\varphi := \rec{\tuple{\vphantom\wedge\smash{t}}}{\tuple u}
	\bigl(&\Ae\tuple x\,(\tuple x\subseteq\tuple u\wedge\varphi) \\[-0,2cm]
	&\sqcup\Ae\tuple x\,(\Ee\tuple y\subseteq\tuple u)(\Ee\tuple z\mid\tuple u)
	\bigl((\tuple x=\tuple y\,\wedge\varphi)\vee\tuple x=\tuple z\bigr)\bigr) \\
	(\Ae\tuple x\mid\tuple{\vphantom\wedge\smash{t}}\,)\,\varphi := 
	\rec{\tuple{\vphantom\wedge\smash{t}}}{\tuple u}\bigl(&\Ae\tuple x\,(\tuple x\subseteq\tuple u) \\[-0,2cm]
	&\sqcup\Ae\tuple x\,(\Ee\tuple y\subseteq\tuple u)(\Ee\tuple z\mid\tuple u)
	\bigl(\tuple x=\tuple y\,\vee(\tuple x=\tuple z\,\wedge\varphi)\bigr)\bigr).
\end{align*}
\end{definition}
Also here the arities match: if $\varphi\in\INEXset[k]$ and $\tuple x,\tuple{\vphantom\wedge\smash{t}}$ are $k$-tuples, then $(\Ae\tuple x\inc\tuple{\vphantom\wedge\smash{t}}\,)\,\varphi,(\Ae\tuple x\mid\tuple{\vphantom\wedge\smash{t}}\,)\,\varphi\in\INEX[k]$. But since we need both inclusion and exclusion atoms for both of these definitions, neither of these quantifiers can be defined in this way in just INC or EXC. It is thus natural to ask whether we could give these definitions in a way that only one type of atoms would be used for each definition -- we will get back to this question in the next subsection.

\begin{proposition}\label{the: Universal quantification}
With the same assumptions as in Definition \ref{def: Universal quantification}, we obtain the following truth conditions:
\begin{enumerate}
\item[a)] $\mathcal{M}\true_X (\Ae\tuple x\subseteq\tuple{\vphantom\wedge\smash{t}}\,)\,\varphi$ \;iff\, $\mathcal{M}\true_{X[A/\tuple x\,]}\varphi$, where $A = X(\tuple{\vphantom\wedge\smash{t}}\,)$.
\item[b)] $\mathcal{M}\true_X (\Ae\tuple x\mid\tuple{\vphantom\wedge\smash{t}}\,)\,\varphi$ \;iff\, $\mathcal{M}\true_{X[A/\tuple x\,]}\varphi$, where $A = \overline{X(\tuple{\vphantom\wedge\smash{t}}\,)}$.	
\end{enumerate}
\end{proposition}

The idea of the proofs for these truth conditions is that the trivial case when $X(\tuple{\vphantom\wedge\smash{t}}\,)=M^k$ is dealt on left side of the intuitionistic disjunction. If $X(\tuple{\vphantom\wedge\smash{t}}\,)\neq M^k$, we first universally quantify $\tuple x$ and then split the resulting team into subteams $Y,Y'$ such that $Y=X[X(\tuple{\vphantom\wedge\smash{t}}\,)/\tuple x\,]$ and $Y'=X[\overline{X(\tuple{\vphantom\wedge\smash{t}}\,)}/\tuple x\,]$. Then we just say that the formula $\varphi$ holds on the desired side.

We first prove the following claim which shows how we can force the team $X[M^k/\tuple x\,]$ to be split into the subteams $X[X(\tuple{\vphantom\wedge\smash{t}}\,)/\tuple x\,]$ and  $X[\overline{X(\tuple{\vphantom\wedge\smash{t}}\,)}/\tuple x\,]$.

\begin{claim}\label{uni}
Let $\psi,\theta\in\INEXset$, $\tuple{\vphantom\wedge\smash{t}}\in\Tset$, let $\tuple x$ be a $k$-tuple of variables, and let $\tuple y,\tuple z$ be $k$-tuples of fresh variables. We additionally assume here that $X(\tuple{\vphantom\wedge\smash{t}}\,)\neq M^k$ and that the variables in the tuple $\tuple x$ are not in $\vr(\tuple{\vphantom\wedge\smash{t}}\,)$. Let
\begin{align*}
	\xi \,:=\, \Ae\tuple x\,&(\Ee\tuple y\subseteq\tuple{\vphantom\wedge\smash{t}}\,)(\Ee\tuple z\mid\tuple{\vphantom\wedge\smash{t}}\,)
	\bigl((\tuple x=\tuple y\wedge\psi)\vee(\tuple x=\tuple z\wedge\theta)\bigr).
\end{align*}
Now we have: \qquad
$
	\mathcal{M}\true_X\xi \;\;\text{ iff }\;\, \mathcal{M}\true_{X[X(\tuple{\vphantom\wedge\smash{t}}\,)/\tuple x\,]}\psi
	\text{ and } \mathcal{M}\true_{X[\overline{X(\tuple{\vphantom\wedge\smash{t}}\,)}/\tuple x\,]}\theta.
$
\end{claim}

\begin{proof}
By locality we may assume for this proof that $\dom(X)=\fr(\xi)$.

\medskip

\noindent
Suppose first that we have $\mathcal{M}\true_X\xi$. Thus there exist $\mathcal{F}_1:X_1\rightarrow\mathcal{P}^*(X_1(\tuple{\vphantom\wedge\smash{t}}\,))$ and $\mathcal{F}_2:X_2\rightarrow\mathcal{P}^*(\overline{X_2(\tuple{\vphantom\wedge\smash{t}}\,)})$ such that $\mathcal{M}\true_{X_3}(\tuple x=\tuple y\wedge\psi)\vee (\tuple x=\tuple z\wedge\theta)$, where $X_1:=X[M^k/\tuple x\,]$, $X_2:=X_1[\mathcal{F}_1/\tuple y\,]$ and $X_3:=X_2[\mathcal{F}_2/\tuple z\,]$. 
Furthermore there exist $Y,Y'\subseteq X_3$ such that $Y\cup Y'=X_3$, $\mathcal{M}\true_{Y}\tuple x\!=\!\tuple y\wedge\psi$ and $\mathcal{M}\true_{Y'}\tuple x\!=\!\tuple z\wedge\theta$. Since $\mathcal{M}\true_{Y}\psi$, $\mathcal{M}\true_{Y'}\theta$ and by the assumption $\tuple y,\tuple z\notin\dom(X_1)$, by locality it is sufficient to show that 
\[
	X[X(\tuple{\vphantom\wedge\smash{t}}\,)/\tuple x\,]=Y\!\upharpoonright\!\dom(X_1) \;\text{ and }\;
	X[\overline{X(\tuple{\vphantom\wedge\smash{t}}\,)}/\tuple x\,]=Y'\!\upharpoonright\!\dom(X_1).
\]
For the sake of proving that $X[X(\tuple{\vphantom\wedge\smash{t}}\,)/\tuple x\,]\subseteq Y\!\upharpoonright\!\dom(X_1)$ let $s\in X[X(\tuple{\vphantom\wedge\smash{t}}\,)/\tuple x\,]$. Let then $r:=s[\tuple a/\tuple y,\tuple{\vphantom t\smash{b}}/\tuple z\,]$, where $\tuple a\in~\!\!\!\mathcal{F}_1(s)$ and $\tuple{\vphantom t\smash{b}}\in\mathcal{F}_2(s[\tuple a/\tuple y\,])$, whence $r\in X_3$. If we would have $r\in Y'$, then we would have $r(\tuple x)=r(\tuple z)=\tuple{\vphantom t\smash{b}}\in\overline{X_2(\tuple{\vphantom\wedge\smash{t}}\,)}$. This is impossible since $r(\tuple x)=s(\tuple x)\in X(\tuple{\vphantom\wedge\smash{t}}\,)=X_2(\tuple{\vphantom\wedge\smash{t}}\,)$. Hence it has to be that $r\in Y$ and thus $s\in Y\!\upharpoonright\!\dom(X_1)$. Therefore $X[X(\tuple{\vphantom\wedge\smash{t}}\,)/\tuple x\,]\subseteq Y\!\upharpoonright\!\dom(X_1)$.

Let then $s\in Y\!\upharpoonright\!\dom(X_1)$. Now there exists $r\in Y$ such that $r=s[\tuple a/\tuple y,\tuple{\vphantom t\smash{b}}/\tuple z\,]$ for some tuples $\tuple a\in\mathcal{F}_1(s)$ and $\tuple{\vphantom t\smash{b}}\in\mathcal{F}_2(s[\tuple a/\tuple y\,])$. Now $r(\tuple x)=r(\tuple y)=\tuple a\in X_1(\tuple{\vphantom\wedge\smash{t}}\,)$ since $r\in Y$. Therefore $s(\tuple x)\in X_1(\tuple{\vphantom\wedge\smash{t}}\,)=X(\tuple{\vphantom\wedge\smash{t}}\,)$ and thus $s\in X[X(\tuple{\vphantom\wedge\smash{t}}\,)/\tuple x\,]$. Hence we have shown that $X[X(\tuple{\vphantom\wedge\smash{t}}\,)/\tuple x\,] = Y\!\upharpoonright\!\dom(X_1)$. We can show with similar reasoning that also $X[\overline{X(\tuple{\vphantom\wedge\smash{t}}\,)}/\tuple x\,] = Y'\!\upharpoonright\!\dom(X_1)$.
\medskip
\\
Suppose then that $\mathcal{M}\true_{X[X(\tuple{\vphantom\wedge\smash{t}}\,)/\tuple x\,]}\psi$ and $\mathcal{M}\true_{X[\overline{X(\tuple{\vphantom\wedge\smash{t}}\,)}/\tuple x\,]}\theta$. We may assume that $X$ is nonempty, because otherwise the claim would hold trivially. Since now $X(\tuple{\vphantom\wedge\smash{t}}\,)\neq\emptyset$ and by assumption $X(\tuple{\vphantom\wedge\smash{t}}\,)\neq M^k$, there exist $\tuple a^*\in X(\tuple{\vphantom\wedge\smash{t}}\,)$ and $\tuple{b}^*\in\overline{X(\tuple{\vphantom\wedge\smash{t}}\,)}$. Let $X_1:=X[M^k/\tuple x\,]$ and
\vspace{-0,1cm}
\[
	\begin{cases}
		\mathcal{F}_1:X_1\rightarrow\mathcal{P}^*(M^k) \,\text{ s.t. }
		\begin{cases}
			s\mapsto\{s(\tuple x)\}\, \text{ if } s(\tuple x)\in X_1(\tuple{\vphantom\wedge\smash{t}}\,) \\
			s\mapsto\{\tuple a^*\}\quad \text{ else}
		\end{cases}\\
		\hspace{1.4cm} X_2:=X_1[\mathcal{F}_1/\tuple y\,]\\
		\mathcal{F}_2:X_2\rightarrow\mathcal{P}^*(M^k) \,\text{ s.t. }
		\begin{cases}
			s\mapsto\{s(\tuple x)\}\, \text{ if } s(\tuple x)\in\overline{X_2(\tuple{\vphantom\wedge\smash{t}}\,)} \\
			s\mapsto\{\tuple{b}^*\}\quad \text{ else}
		\end{cases}\\
		\hspace{1.4cm} X_3:=X_2[\mathcal{F}_2/\tuple z\,].
	\end{cases}
\]
Clearly $\im(\mathcal{F}_1)\subseteq\mathcal{P}^*(X_1(\tuple{\vphantom\wedge\smash{t}}\,))$ and $\im(\mathcal{F}_2)\subseteq\mathcal{P}^*(\overline{X_2(\tuple{\vphantom\wedge\smash{t}}\,)})$. We define the teams $Y:=\{s\in X_3\mid s(\tuple x)\in X_3(\tuple{\vphantom\wedge\smash{t}}\,)\}$ and $Y':=\{s\in X_3\mid s(\tuple x)\in\overline{X_3(\tuple{\vphantom\wedge\smash{t}}\,)}\}$. Now clearly $Y\cup Y'=X_3$, $\mathcal{M}\true_{Y}\tuple x=\tuple y$ and $\mathcal{M}\true_{Y'}\tuple x=\tuple z$. Next we show that $X[X(\tuple{\vphantom\wedge\smash{t}}\,)/\tuple x\,]=Y\!\upharpoonright\!\dom(X_1)$ and $X[\overline{X(\tuple{\vphantom\wedge\smash{t}}\,)}/\tuple x\,]=Y'\!\upharpoonright\!\dom(X_1)$.

Let $s\in X[X(\tuple{\vphantom\wedge\smash{t}}\,)/\tuple x\,]$ and let $r:=s[s(\tuple x)/\tuple y,\,\tuple{b}^*/\tuple z\,]$, whence $r\in X_3$. If we would have $r\in Y'$, then it would hold that $r(\tuple x)=r(\tuple z)=\tuple{b}^*\notin X(\tuple{\vphantom\wedge\smash{t}}\,)$. But this is not possible since $r(\tuple x)=s(\tuple x)\in X(\tuple{\vphantom\wedge\smash{t}}\,)$. Hence $r\in Y$, and thus $s\in Y\!\upharpoonright\!\dom(X_1)$.

Let then $s\in Y\!\upharpoonright\!\dom(X_1)$, whence there exists $r\in Y$ s.t. $r=s[\tuple a/\tuple y,\tuple{\vphantom t\smash{b}}/\tuple z\,]$ for some $\tuple a\in\mathcal{F}_1(s)$ and $\tuple{\vphantom t\smash{b}}\in\mathcal{F}_2(s[\tuple a/\tuple y\,])$. Now $r(\tuple x)\in X_3(\tuple{\vphantom\wedge\smash{t}}\,)=X(\tuple{\vphantom\wedge\smash{t}}\,)$ since $r\in Y$. Therefore also $s(\tuple x)\in X(\tuple{\vphantom\wedge\smash{t}}\,)$ and thus $s\in X[X(\tuple{\vphantom\wedge\smash{t}}\,)/\tuple x\,]$.

Hence we have shown that $X[X(\tuple{\vphantom\wedge\smash{t}}\,)/\tuple x\,] = Y\!\upharpoonright\!\dom(X_1)$. We can show with similar reasoning that also $X[\overline{X(\tuple{\vphantom\wedge\smash{t}}\,)}/\tuple x\,] = Y'\!\upharpoonright\!\dom(X_1)$, and thus by locality we have $\mathcal{M}\true_{Y}\psi$ and $\mathcal{M}\true_{Y'}\theta$. Hence $\mathcal{M}\true_{X_3}(\tuple x\!=\!\tuple y\wedge\psi)\vee (\tuple x\!=\!\tuple z\wedge\theta)$, and furthermore we have $\mathcal{M}\true_X\xi$.
\end{proof}

Now we are ready to prove the truth conditions for universal inclusion and exclusion quantifiers (Proposition \ref{the: Universal quantification}). We have already done most of the work by proving Claim \ref{uni}. We only need to consider the use of storing operator and the special case when $X(\tuple{\vphantom\wedge\smash{t}}\,)=M^k$. When using the storing operator, we may drop the extra assumption that $\vr(\tuple x)\cap\vr(\tuple{\vphantom\wedge\smash{t}}\,)=\emptyset$. In the  special case when $X(\tuple{\vphantom\wedge\smash{t}}\,)=M^k$ the universal inclusion quantifier $(\Ae\tuple x\inc\tuple{\vphantom\wedge\smash{t}}\,)$ becomes the normal universal quantifier $\Ae\tuple x$ and the universal exclusion quantifier $(\Ae\tuple x\exc\tuple{\vphantom\wedge\smash{t}}\,)$ becomes trivially true.

\begin{proof}
(Proposition \ref{the: Universal quantification})
In this proof we write $X':=\{s[s(\tuple{\vphantom\wedge\smash{t}}\,)/\tuple u\,]\mid s\in X\}$. 

\medskip

\noindent
a) Suppose first that $\mathcal{\mathcal{M}}\true_X(\Ae\tuple x\subseteq\tuple{\vphantom\wedge\smash{t}}\,)\,\varphi$. By Lemma \ref{the: Storing operator} and the truth condition of intuitionistic disjunction we have $\mathcal{M}\true_{X'}\Ae\tuple x\,(\tuple x\subseteq\tuple u\,\wedge\varphi)$ or
\vspace{-0,1cm}
\begin{align*}
	\mathcal{M}\true_{X'}\Ae\tuple x\,&(\Ee\tuple y\subseteq\tuple u)(\Ee\tuple z\mid\tuple u)
	\bigl((\tuple x=\tuple y\wedge\varphi)\vee\tuple x=\tuple z)\bigr). \tag{$\star$}
\vspace{-0,1cm}
\end{align*}
Suppose first that $\mathcal{M}\true_{X'}\,\Ae\tuple x\,(\tuple x\inc\tuple u\,\wedge\varphi)$. Since $\mathcal{M}\true_{X'[M^k/\tuple x]}\tuple x\inc\tuple u\,\wedge\varphi$, we clearly have $X'(\tuple u)=M^k$ and $\mathcal{M}\true_{X'}\Ae\tuple x\,\varphi$. By Lemma \ref{the: Storing operator}, $X(\tuple{\vphantom\wedge\smash{t}}\,)=M^k$, and by locality $\mathcal{M}\true_{X}\Ae\tuple x\,\varphi$. Since now $X[X(\tuple{\vphantom\wedge\smash{t}}\,)/\tuple x\,]=X[M^k/\tuple x\,]$, we have $\mathcal{M}\true_{X[X(\tuple{\vphantom\wedge\smash{t}}\,)/\tuple x\,]}\varphi$.

Suppose then that ($\star$) holds. Note that since $\tuple z$ can be quantified within the complement of $\tuple u$, it cannot be the case that $X(\tuple u)=M^k$. By choosing $\psi:=\varphi$ and $\theta:=(\tuple x\!=\!\tuple x)$ we can apply Claim \ref{uni} to obtain $\mathcal{M}\true_{X'[X'(\tuple u\,)/\tuple x\,]}\varphi$. Since $X(\tuple{\vphantom\wedge\smash{t}}\,)=X'(\tuple u)$, by locality we have $\mathcal{M}\true_{X[X(\tuple{\vphantom\wedge\smash{t}}\,)/\tuple x\,]}\varphi$.
\medskip
\\
Suppose then that $\mathcal{M}\true_{X[X(\tuple{\vphantom\wedge\smash{t}}\,)/\tuple x\,]}\varphi$. If $X(\tuple{\vphantom\wedge\smash{t}}\,)=M^k$, then it is easy to see that $\mathcal{M}\true_{X'}\Ae\tuple x\,(\tuple x\subseteq\tuple u\wedge\varphi)$ and thus $\mathcal{\mathcal{M}}\true_X(\Ae\tuple x\subseteq\tuple{\vphantom\wedge\smash{t}}\,)\,\varphi$. Thus we may assume that $X(\tuple{\vphantom\wedge\smash{t}}\,)\neq M^k$. By Lemma \ref{the: Storing operator} we have $\mathcal{M}\true_{X'[X'(\tuple u\,)/\tuple x]}\varphi$ and thus by applying Claim~\ref{uni} for $\psi:=\varphi$ and $\theta:=(\tuple x\!=\!\tuple x)$, we obtain ($\star$). Hence $\mathcal{\mathcal{M}}\true_X(\Ae\tuple x\subseteq\tuple{\vphantom\wedge\smash{t}}\,)\,\varphi$.
\medskip
\\
b) Suppose first that $\mathcal{\mathcal{M}}\true_X(\Ae\tuple x\mid\tuple{\vphantom\wedge\smash{t}}\,)\,\varphi$. Now we have $\mathcal{M}\true_{X'}\Ae\tuple x\,(\tuple x\subseteq\tuple u)$ or
\vspace{-0,1cm}
\begin{align*}
	\mathcal{M}\true_{X'}\Ae\tuple x\,
	&(\Ee\tuple y\subseteq\tuple{\vphantom\wedge\smash{t}})(\Ee\tuple z\mid\tuple{\vphantom\wedge\smash{t}})
	\bigl(\tuple x=\tuple y\,\vee(\tuple x=\tuple z\,\wedge\varphi))\bigr). \tag{$\star\star$}
\vspace{-0,1cm}
\end{align*}
Suppose first that $\mathcal{M}\true_{X'}\Ae\tuple x\,(\tuple x\inc\tuple u)$. Therefore we have $X'(\tuple u)=M^k$ and since $X(\tuple{\vphantom\wedge\smash{t}}\,)=X'(\tuple u)$, we obtain $\overline{X(\tuple{\vphantom\wedge\smash{t}}\,)}=\emptyset$.  Now $X[\overline{X(\tuple{\vphantom\wedge\smash{t}}\,)}/\tuple x\,]=\emptyset$  and thus trivially $\mathcal{M}\true_{X[\overline{X(\tuple{\vphantom\wedge\smash{t}}\,)}/\tuple x\,]}\varphi$.
Suppose then that ($\star\star$) holds. By choosing $\psi:=(\tuple x\!=\!\tuple x)$ and $\theta:=\varphi$ we obtain $\mathcal{M}\true_{X'[\overline{X'(\tuple u\,)}/\tuple x\,]}\varphi$ by Claim~\ref{uni}, and thus $\mathcal{M}\true_{X[\overline{X(\tuple{\vphantom\wedge\smash{t}}\,)}/\tuple x\,]}\varphi$.
\medskip
\\
Suppose then that $\mathcal{M}\true_{X[\overline{X(\tuple{\vphantom\wedge\smash{t}}\,)}/\tuple x\,]}\varphi$. If $X(\tuple{\vphantom\wedge\smash{t}}\,)=M^k$, clearly $\mathcal{M}\true_{X'}\Ae\tuple x\,(\tuple x\subseteq\tuple u)$ and thus $\mathcal{\mathcal{M}}\true_X(\Ae\tuple x\exc\tuple{\vphantom\wedge\smash{t}}\,)\,\varphi$. Hence we may assume that $X(\tuple{\vphantom\wedge\smash{t}}\,)\neq M^k$. By the assumption $\mathcal{M}\true_{X'[\overline{X'(\tuple u\,)}/\tuple x\,]}\varphi$ and thus by applying Claim \ref{uni} for $\psi:=(\tuple x\!=\!\tuple x)$ and $\theta:=\varphi$, we obtain ($\star\star$). Hence we have $\mathcal{\mathcal{M}}\true_X(\Ae\tuple x\mid\tuple{\vphantom\wedge\smash{t}}\,)\,\varphi$.
\end{proof}

\begin{remark}
As with existential inclusion and exclusion quantifiers, we allow the variables in $\tuple x$ to be in $\vr(\tuple{\vphantom\wedge\smash{t}}\,)$. In particular, we allow universal quantifiers of the form $(\Ae \tuple x\subseteq\tuple x)$. This strange looking quantifier turns out be a rather useful operator in an another context which is studied by the author in \cite{Ronnholm16}.
\end{remark}

A natural idea for the truth definition for universal inclusion quantification $(\Ae\tuple x\subseteq\tuple{y}\,)$ is ``$\Ae\tuple x\in M^k:(\tuple x\subseteq\tuple{y}\,\Rightarrow\,\varphi)$''. This intuition would give us the following definition: $(\Ae\tuple x\subseteq\tuple{y}\,)\,\varphi \,:=\; \Ae\tuple x\,(\tuple x\mid\tuple{y}\;\vee\,\varphi)$.
However, this simple idea does not work for two reasons. Firstly, there might be too many values chosen for $\tuple x$ on the right side of the disjunction, which can be a problem since INEX is not closed downwards. Secondly, the exclusion atom is evaluated \emph{after} splitting the team and thus some of the original values for $\tuple y$ might be lost. This general problem regarding the ``loss of information'' when evaluating disjunctions will be discussed more in the Subsection~\ref{ssec: Term value preserving disjunction}, where we define term value preserving disjunction.

\subsection{Analyzing the properties of inclusion and exclusion quantifiers}\label{ssec: Properties of quantifiers}

In the previous subsections we showed that inclusion and exclusion quantifiers can be expressed with inclusion and exclusion atoms, and thus we were able to define them as abbreviations in INEX. In this subsection we take a reverse perspective by considering them as basic operations to be added to FO and examining the expressive power of the resulting logics.
The following observation shows that we can define inclusion and exclusion atoms with existential inclusion and exclusion quantifiers $(\Ee\tuple x\inc\tuple{\vphantom\wedge\smash{t}}\,)$ and $(\Ee\tuple x\exc\tuple{\vphantom\wedge\smash{t}}\,)$.
\begin{observation}\label{obs: Defining atoms 1}
Let $\tuple{\vphantom\wedge\smash{t}}_1,\tuple{\vphantom\wedge\smash{t}}_2$ be $k$-tuples of $L$-terms and let $\tuple x$ be a $k$-tuple of fresh variables. Now it holds that:
\begin{align*}
	\mathcal{M}\true_X \tuple{\vphantom\wedge\smash{t}}_1\subseteq\tuple{\vphantom\wedge\smash{t}}_2
	&\;\;\text{ iff  }\; \mathcal{M}\true_X (\Ee \tuple x\subseteq\tuple{\vphantom\wedge\smash{t}}_2)(\tuple x = \tuple{\vphantom\wedge\smash{t}}_1). \\
	\mathcal{M}\true_X \tuple{\vphantom\wedge\smash{t}}_1\mid\tuple{\vphantom\wedge\smash{t}}_2
	&\;\;\text{ iff  }\; \mathcal{M}\true_X (\Ee \tuple x\mid\tuple{\vphantom\wedge\smash{t}}_2)(\tuple x = \tuple{\vphantom\wedge\smash{t}}_1),
\end{align*}
We explain briefly why these equivalences hold. We first notice that for any function $\mathcal{F}:X\rightarrow\mathcal{P}^*(M^k)$ the following holds:
\[
	\mathcal{M}\true_{X[\mathcal{F}/\tuple x\,]}\tuple x=\tuple{\vphantom\wedge\smash{t}}_1
	\;\text{ iff }\; \mathcal{F}(s)=\{s(\tuple{\vphantom\wedge\smash{t}}_1)\} \text{ for each } s\in X. \tag{$\star$}
\]
It is easy to see that if $\mathcal{F}$ is a function which satisfies the (both) sides of $(\star)$, then we have: $\ran(\mathcal{F})\subseteq\mathcal{P}^*(X(\tuple{\vphantom\wedge\smash{t}}_2))$ if and only if $\mathcal{M}\true_X \tuple{\vphantom\wedge\smash{t}}_1\inc\tuple{\vphantom\wedge\smash{t}}_2$. The first equivalence follows from this. The second one is also clear since $\ran(\mathcal{F})\subseteq\mathcal{P}^*(\overline{X(\tuple{\vphantom\wedge\smash{t}}_2)})$ iff $\mathcal{M}\true_X \tuple{\vphantom\wedge\smash{t}}_1\exc\tuple{\vphantom\wedge\smash{t}}_2$, for any $\mathcal{F}$ which satisfies the both sides of $(\star)$.
\end{observation}

Recall that, in Definition~\ref{def: Existential quantification}, we were able to define the quantifier $(\Ee\tuple x\inc\tuple{\vphantom\wedge\smash{t}}\,)$ with inclusion atom and the quantifier $(\Ee\tuple x\exc\tuple{\vphantom\wedge\smash{t}}\,)$ with exclusion atom. Hence, by the previous observation, if we extend FO with quantifiers $(\Ee\tuple x\inc\tuple{\vphantom\wedge\smash{t}}\,)$ or $(\Ee\tuple x\exc\tuple{\vphantom\wedge\smash{t}}\,)$, we obtain equivalent logics with $\INC$ and $\EXC$, respectively. We call these logics \emph{inclusion and exclusion friendly logics} due their similarity with IF-logic. By using the both of these quantifiers, we obtain \emph{inclusion-exclusion friendly logic} that is equivalent with $\INEX$. 

Also note that the arities of these operations match, since the use of  existential inclusion (exclusion) quantifiers for $k$-tuples corresponds to the use of $k$-ary inclusion (exclusion) atoms. Hence the use of existential  inclusion and exclusion quantifiers for \emph{single}
%
%
first order variables corresponds to the use of \emph{unary} inclusion and exclusion atoms, and thus, by extending FO with either/both of them, we obtain logics equivalent to $\INC[1]$, $\EXC[1]$ and $\INEX[1]$.

After the Observation \ref{obs: Defining atoms 1} it is natural to ask whether we can define inclusion and exclusion atoms alternatively by using universal inclusion and exclusion quantifiers $(\Ae\tuple x\inc\tuple{\vphantom\wedge\smash{t}}\,)$ and $(\Ae\tuple x\exc\tuple{\vphantom\wedge\smash{t}}\,)$. This can also be done, however, this time inclusion atom is defined with universal \emph{exclusion} quantifier and exclusion atom is defined with universal \emph{inclusion} quantifier.
\begin{observation}\label{obs: Defining atoms 2}
Let $\tuple{\vphantom\wedge\smash{t}}_1,\tuple{\vphantom\wedge\smash{t}}_2$ be $k$-tuples of $L$-terms and let $\tuple x$ be a $k$-tuple of fresh variables. Now the following equivalences hold:
\begin{align*}
	\mathcal{M}\true_X\tuple{\vphantom\wedge\smash{t}}_1\inc\tuple{\vphantom\wedge\smash{t}}_2 \;&\text{ iff }\;
	\mathcal{M}\true_X (\Ae\tuple x\exc\tuple{\vphantom\wedge\smash{t}}_2)(\tuple x\neq\tuple{\vphantom\wedge\smash{t}}_1) \\
	\mathcal{M}\true_X\tuple{\vphantom\wedge\smash{t}}_1\exc\tuple{\vphantom\wedge\smash{t}}_2 \;&\text{ iff }\;
	\mathcal{M}\true_X (\Ae\tuple x\inc\tuple{\vphantom\wedge\smash{t}}_2)(\tuple x\neq\tuple{\vphantom\wedge\smash{t}}_1),
\end{align*}
We prove the first equivalence by contraposition:
Suppose that $\mathcal{M}\ntrue_X\tuple{\vphantom\wedge\smash{t}}_1\inc\tuple{\vphantom\wedge\smash{t}}_2$, i.e. there is $s\in X$ such that $s(\tuple{\vphantom\wedge\smash{t}}_1)\notin X(\tuple{\vphantom\wedge\smash{t}}_2)$. Let $r:=s[s(\tuple{\vphantom\wedge\smash{t}}_1)/\tuple x\,]$, whence $r\in X[\overline{X(\tuple{\vphantom\wedge\smash{t}}_2)}/\tuple x\,]$. Now $r(\tuple x)=s(\tuple{\vphantom\wedge\smash{t}}_1)=r(\tuple{\vphantom\wedge\smash{t}}_1)$ and thus $\mathcal{M}\ntrue_X (\Ae\tuple x\exc\tuple{\vphantom\wedge\smash{t}}_2)(\tuple x\neq\tuple{\vphantom\wedge\smash{t}}_1)$.

For the other direction suppose that $\mathcal{M}\ntrue_X (\Ae\tuple x\exc\tuple{\vphantom\wedge\smash{t}}_2)(\tuple x\neq\tuple{\vphantom\wedge\smash{t}}_1)$, whence there is $r\in X[\overline{X(\tuple{\vphantom\wedge\smash{t}}_2)}/\tuple x\,]$ such that $r(\tuple x)=r(\tuple{\vphantom\wedge\smash{t}}_1)$. Now there is $s\in X$ and $\tuple a\in\overline{X(\tuple{\vphantom\wedge\smash{t}}_2)}$ such that $r=s[\tuple a/\tuple x\,]$. But since $s(\tuple{\vphantom\wedge\smash{t}}_1)=r(\tuple{\vphantom\wedge\smash{t}}_1)=r(\tuple x)=\tuple a\notin X(\tuple{\vphantom\wedge\smash{t}}_2)$, we have $\mathcal{M}\ntrue_X\tuple{\vphantom\wedge\smash{t}}_1\inc\tuple{\vphantom\wedge\smash{t}}_2$. 
The second equivalence can be proven by a similar reasoning.
\end{observation}

When we combine the equivalences above with the respective equivalences in Observation~\ref{obs: Defining atoms 1}, we obtain the following correspondence.
\begin{align*}
	(\Ee\tuple x\inc\tuple{\vphantom\wedge\smash{t}}_2)(\tuple x=\tuple{\vphantom\wedge\smash{t}}_1) 
	&\;\equiv\; (\Ae\tuple x\exc\tuple{\vphantom\wedge\smash{t}}_2)(\tuple x\neq\tuple{\vphantom\wedge\smash{t}}_1) \\
	(\Ee\tuple x\exc\tuple{\vphantom\wedge\smash{t}}_2)(\tuple x=\tuple{\vphantom\wedge\smash{t}}_1) 
	&\;\equiv\; (\Ae\tuple x\inc\tuple{\vphantom\wedge\smash{t}}_2)(\tuple x\neq\tuple{\vphantom\wedge\smash{t}}_1).
\end{align*}
Here we have an interesting duality between the inclusion and exclusion quantifiers. This leads to a natural question whether existential inclusion quantifier $(\Ee\tuple x\inc\tuple{\vphantom\wedge\smash{t}}\,)$ has the same expressive power as universal exclusion quantifier $(\Ae\tuple x\exc\tuple{\vphantom\wedge\smash{t}}\,)$, and the whether the same holds for the quantifiers $(\Ee\tuple x\exc\tuple{\vphantom\wedge\smash{t}}\,)$ and $(\Ae\tuple x\inc\tuple{\vphantom\wedge\smash{t}}\,)$. We approach this question by first comparing universal inclusion/exclusion quantifiers with INC and EXC.

In Definition~\ref{def: Universal quantification} we defined universal inclusion and exclusion quantifiers in INEX by using both inclusion and exclusion atoms. We examine next whether either of them could be defined by using only one type of these atoms. For the next observation, recall that EXC is closed downwards and INC under unions.

\begin{observation}\label{obs: Properties of quantifiers}
Let $\mathcal{M}=(\mathcal{I},M)$ be an $L$-model s.t. $M=\{0,1,2\}$, and let $X_1=\{s_{01}\}$ and $X_2=\{s_{10}\}$, where $s_{01}(x)=0=s_{10}(y)$ and $s_{01}(y)=1=s_{10}(x)$. 

\medskip

\noindent
(A) We first show that universal inclusion quantifier is not closed under unions.
For this, let $\varphi:=(\Ae z\inc x)(y\neq z)$. We consider the following teams
\begin{align*}
	&Y_1:=X_1[X_1(x)/z]=X_1[\{0\}/z]=\{s_{01}[0/z]\} \\
	&Y_2:=X_2[X_2(x)/z]=X_2[\{1\}/z]=\{s_{10}[1/z]\} \\
	&Y_3:=(X_1\!\cup\!X_2)\bigl[(X_1\!\cup\!X_2)(x)/z\bigr]=(X_1\!\cup\!X_2)[\{0,1\}/z] \\
	&\hspace{5,72cm}=\{s_{01}[0/z],s_{01}[1/z],s_{10}[0/z],s_{10}[1/z]\}.
\end{align*}
Now we have $\mathcal{M}\true_{Y_1}y\neq z$ and $\mathcal{M}\true_{Y_2}y\neq z$, but $\mathcal{M}\ntrue_{Y_3}y\neq z$. Hence $\mathcal{M}\true_{X_1}\varphi$ and $\mathcal{M}\true_{X_2}\varphi$, but $\mathcal{M}\ntrue_{X_1\cup X_2}\varphi$.

\medskip

\noindent
(B) Next, we show that universal exclusion quantifier is not closed under unions. Let $\psi:=(\Ae z\exc x)(y\inc z)$. Note that, by Observation~\ref{obs: Defining atoms 2}, $y\inc z$ can be expressed with universal exclusion quantifier ($\psi\equiv(\Ae z\exc x)(\Ae w\exc z)(w\!\neq\!y)$). Let
\begin{align*}
	&Z_1:=X_1\bigl[\overline{X_1(x)}/z\bigr]=X_1\bigl[\overline{\{0\}}/z\bigr]=X_1[\{1,2\}/z]=\{s_{01}[1/z],s_{01}[2/z]\} \\
	&Z_2:=X_2\bigl[\overline{X_2(x)}/z\bigr]=X_2\bigl[\overline{\{1\}}/z\bigr]=X_2[\{0,2\}/x]=\{s_{10}[0/z],s_{10}[2/z]\} \\
	&Z_3:=(X_1\!\cup\!X_2)\bigl[\overline{(X_1\!\cup\!X_2)(x)}/z\bigr]=(X_1\!\cup\!X_2)\bigl[\overline{\{0,1\}}/z\bigr] \\
	&\hspace{5,75cm}=(X_1\!\cup\!X_2)[\{2\}/z]=\{s_{01}[2/z],s_{10}[2/z]\}.
\end{align*}
Since $Z_1(y)=\{1\}\subseteq\{1,2\}=Z_1(z)$ and $Z_2(y)=\{0\}\subseteq\{0,2\}=Z_2(z)$, we have $\mathcal{M}\true_{Z_1}y\inc z$ and $\mathcal{M}\true_{Z_2}y\inc z$. But because $Z_3(y)=\{0,1\}\not\subseteq\{2\}=Z_3(z)$, we have $\mathcal{M}\ntrue_{Z_3}y\inc z$. Hence $\mathcal{M}\true_{X_1}\psi$ and $\mathcal{M}\true_{X_2}\psi$, but $\mathcal{M}\ntrue_{X_1\cup X_2}\psi$.

\medskip

\noindent
(C) Finally, we show that universal exclusion quantifier is not closed downwards either. Let $\theta:=(\Ae z\exc x)(y\neq z)$ and let $Z_1,Z_3$ be as above. Now $\mathcal{M}\true_{Z_3}y\neq z$, but $\mathcal{M}\ntrue_{Z_1}y\neq z$. Hence $\mathcal{M}\true_{X_1\cup X_2}\theta$, but $\mathcal{M}\ntrue_{X_1}\theta$; even though $X_1\subseteq X_1\!\cup\!X_2$.
\end{observation}

By this observation, universal exclusion quantifier cannot be defined in EXC and \emph{neither} universal inclusion nor exclusion quantifier can be defined in INC. But there is still a possibility that universal inclusion quantifier could be defined in EXC. It turns out that this can indeed be done, but we must give its definition in a form that would not work properly in INEX. To make distinction with the earlier definition, we denote this quantifier $(\Ae\tuple x\ince\tuple{\vphantom\wedge\smash{t}}\,)$, where ``$e$'' stands for ``exclusion'', as this operator is defined for exclusion logic only.

\begin{definition}\label{def: Universal inclusion quantifier for EXC}
Let $\varphi\in\EXCset$, $\tuple{\vphantom\wedge\smash{t}}\in\Tset$ a $k$-tuple, $\tuple x$ a $k$-tuple of variables and $\tuple u,\tuple y$ \,$k$-tuples of fresh variables. We use the following notation:
\[
	(\Ae\tuple x\ince\tuple{\vphantom\wedge\smash{t}}\,)\,\varphi \,:=\; 
	\;\Ae\tuple x\,\varphi\,\sqcup\,
	\rec{\tuple{\vphantom\wedge\smash{t}}}{\tuple u}\Ae\tuple x\,(\Ee\tuple y\mid\!\tuple u)(\tuple y\!=\!\tuple x\,\vee\varphi).
\]
Since intuitionistic disjunction can be defined with unary exclusion atoms we have $(\Ae\tuple x\ince\tuple{\vphantom\wedge\smash{t}}\,)\,\varphi\in\EXCset[k]$ when $\varphi\in\EXCset[k]$ (for any $k\geq 1$).
\end{definition}

\begin{proposition}\label{the: Universal inclusion quantifier for EXC}
With the same assumptions as in Definition \ref{def: Universal inclusion quantifier for EXC}, we obtain the following truth condition: \;
$\mathcal{M}\true_X (\Ae\tuple x\ince\tuple{\vphantom\wedge\smash{t}}\,)\,\varphi \;\text{ iff }\, \mathcal{M}\true_{X[X(\tuple{\vphantom\wedge\smash{t}}\,)/\tuple x\,]}\varphi$.
\end{proposition}

\begin{proof}
By locality we may assume for this proof that the variables in $\tuple y$ are not in $\dom(X)$. We write $V^*:=\,\dom(X)\cup\vr(\tuple u\tuple x)$.

\medskip

\noindent
Suppose that $\mathcal{M}\true_X(\Ae\tuple x\ince\tuple{\vphantom\wedge\smash{t}}\,)\,\varphi$, i.e. $\mathcal{M}\true_X\Ae\tuple x\,\varphi$ or 
\vspace{-0,1cm}\[
	\mathcal{M}\true_X\rec{\tuple{\vphantom\wedge\smash{t}}}{\tuple u}\Ae\tuple x
	\,(\Ee\tuple y\exc\tuple u)(\tuple y=\tuple x\,\vee\varphi). \tag{$\star$}
\vspace{-0,1cm}\]
Suppose first that $\mathcal{M}\true_X\Ae\tuple x\,\varphi$, i.e. $\mathcal{M}\true_{X[M^k/\tuple x\,]}\varphi$. Since $X(\tuple{\vphantom\wedge\smash{t}}\,)\subseteq M^k$, also $X[X(\tuple{\vphantom\wedge\smash{t}}\,)/\tuple x\,]\subseteq X[M^k/\tuple x\,]$. Thus $\mathcal{M}\true_{X[X(\tuple{\vphantom\wedge\smash{t}}\,)/\tuple x\,]}\varphi$ since EXC is closed downwards.

Suppose then that ($\star$) holds. Now $\mathcal{M}\true_{X'}\Ae\tuple x\,(\Ee\tuple y\exc\tuple u)(\tuple y\!=\!\tuple x\,\vee\varphi)$, where $X'=\{s[s(\tuple{\vphantom\wedge\smash{t}}\,)/\tuple u\,]\mid s\in X\}$. Hence we have $\mathcal{M}\true_{X_1}(\Ee\tuple y\exc\tuple u)(\tuple y\!=\!\tuple x\,\vee\varphi)$, where $X_1=X'[M^k/\tuple x\,]$. Thus there exists a function $\mathcal{F}:X_1\rightarrow\mathcal{P}^*(\overline{X_1(\tuple u)})$ such that $\mathcal{M}\true_{X_2}\tuple y\!=\!\tuple x\,\vee\varphi$, where $X_2=X_1[\mathcal{F}/\tuple y\,]$. Now there are $Y,Y'\subseteq X_2$ such that $Y\cup Y'=X_2$, $\mathcal{M}\true_{Y}\tuple y=\tuple x$ and $\mathcal{M}\true_{Y'}\varphi$.

For the sake of showing that $X'[X'(\tuple u)/\tuple x\,]\subseteq Y'\!\upharpoonright V^*$, let $r\in X'[X'(\tuple u)/\tuple x\,]$. Now there is $s\in X'$ and $\tuple a\in X'(\tuple u\,)$ such that $r=s[\tuple a/\tuple x\,]$. Let $\tuple{\vphantom t\smash{b}}\in\mathcal{F}(r)$ and $q:=r[\tuple{\vphantom t\smash{b}}/\tuple y\,]$, whence $q\in X_2$. Since $\mathcal{F}$ only chooses values in $\overline{X_1(\tuple u)}$ and $\tuple a\in X'(\tuple u)=X_1(\tuple u)$, we must have $q(\tuple y)=\tuple{\vphantom t\smash{b}}\neq\tuple a$. Thus
\vspace{-0,2cm}\[
	q(\tuple x) = r(\tuple x) = s[\tuple a/\tuple x\,](\tuple x) = \tuple a \neq q(\tuple y).
\vspace{-0,2cm}\]
But since $\mathcal{M}\true_{Y}\tuple y=\tuple x$, we must have $q\notin Y$ and therefore $q\in Y'$. Furthermore $r=q\upharpoonright V^*\,\in\,Y'\!\upharpoonright V^*$, and thus $X'[X'(\tuple u)/\tuple x\,]\subseteq Y'\!\upharpoonright V^*$.

Because $\mathcal{M}\true_{Y'}\varphi$, by locality we have $\mathcal{M}\true_{Y'\upharpoonright V^*}\varphi$. Since exclusion logic is closed downwards, we have $\mathcal{M}\true_{X'[X'(\tuple u)/\tuple x\,]}\varphi$. But since $X'(\tuple u)=X(\tuple{\vphantom\wedge\smash{t}}\,)$, it is now easy to see that by locality $\mathcal{M}\true_{X[X(\tuple{\vphantom\wedge\smash{t}}\,)/\tuple x\,]}\varphi$.

\medskip

\noindent
Suppose then that $\mathcal{M}\true_{X[X(\tuple{\vphantom\wedge\smash{t}}\,)/\tuple x\,]}\varphi$. If $X(\tuple{\vphantom\wedge\smash{t}}\,)=M^k$, we have $\mathcal{M}\true_{X[M^k/\tuple x\,]}\varphi$, i.e. $\mathcal{M}\true_{X}\Ae\tuple x\,\varphi$, and therefore $\mathcal{M}\true_X(\Ae\tuple x\ince\tuple{\vphantom\wedge\smash{t}}\,)\,\varphi$.
Hence we may assume that $X(\tuple{\vphantom\wedge\smash{t}}\,)\neq M^k$, whence there exists $\tuple c\notin X(\tuple{\vphantom\wedge\smash{t}}\,)$. Let $X':=\{s[s(\tuple{\vphantom\wedge\smash{t}}\,)/\tuple u\,]\mid s\in X\}$ and $X_1=X'[M^k/\tuple x\,]$. Since $X'(\tuple u)=X(\tuple{\vphantom\wedge\smash{t}}\,)$, we have $\tuple c\notin X'(\tuple u\,)$. Let 
\vspace{-0,2cm}\[
	\mathcal{F}: X_1\rightarrow\mathcal{P}^*(M^k) \,\text{ s.t. }
	\begin{cases}
		s\mapsto \{s(\tuple x)\} \;\, \text{ if } s(\tuple x)\notin X'(\tuple u) \\
		s\mapsto \{\tuple c\,\} \qquad\text{else}. 
	\end{cases}
\vspace{-0,2cm}\]
Let $X_2:=X_1[\mathcal{F}/\tuple y\,]$. Since $X'(\tuple u)=X_1(\tuple u)$, we see that $\im(\mathcal{F})\subseteq\mathcal{P}^*(\overline{X_1(\tuple u)})$. Let $Y:=\{s\in X_2\mid s(\tuple x)\notin X'(\tuple u)\}$ and $Y':=\{s\in X_2\mid s(\tuple x)\in X'(\tuple u)\}$. Now clearly $Y\cup Y'=X_2$ and by the definition of $\mathcal{F}$ we have $\mathcal{M}\true_{Y}\tuple y=\tuple x$.

For the sake of showing that $Y'\!\upharpoonright V^*\subseteq X'[X'(\tuple u)/\tuple x\,]$, let $r^*\in Y'\!\upharpoonright V^*$. Now there exists $r\in Y'$ such that $r^*= r\upharpoonright V^*$. By the definition of $Y'$, we have $r(\tuple x)\in X'(\tuple u)$. Since $r\in X_2=X_1[\mathcal{F}/\tuple y\,]$, there exist $s\in X_1$ and $\tuple{\vphantom t\smash{b}}\in\mathcal{F}(s)$ such that $r=s[\tuple{\vphantom t\smash{b}}/\tuple x\,]$. Because $s\in X_1=X'[M^k/\tuple x\,]$ and $s(\tuple x)=r(\tuple x)\in X'(\tuple u)$, we have $s\in X'[X'(\tuple u)/\tuple x\,]$. But now it must also be that $r^*=s$, and thus we have shown that $Y'\!\upharpoonright V^*\subseteq X'[X'(\tuple u)/\tuple x\,]$. 

Since $X'(\tuple u)=X(\tuple{\vphantom\wedge\smash{t}}\,)$ and by the assumption $\mathcal{M}\true_{X[X(\tuple{\vphantom\wedge\smash{t}}\,)/\tuple x\,]}\varphi$, it easy to see by locality that $\mathcal{M}\true_{X'[X'(\tuple u)/\tuple x\,]}\varphi$. Because exclusion logic is closed downwards, $\mathcal{M}\true_{Y'\upharpoonright V^*}\varphi$, and thus by locality $\mathcal{M}\true_{Y'}\varphi$. Therefore $\mathcal{M}\true_{X_2}\tuple y\!=\!\tuple x\,\vee\varphi$ and furthermore $(\star)$ holds. Hence we have $\mathcal{M}\true_X(\Ae\tuple x\ince\tuple{\vphantom\wedge\smash{t}}\,)\,\varphi$.
\end{proof}

In the proof above we had to use the assumption of downwards closure, and thus this proof is not valid for $\INEXset$-formulas. Furthermore, the claim of Proposition~\ref{the: Universal inclusion quantifier for EXC} is not necessarily true when $\varphi\in\INEXset$ since, for example, if $\varphi:=\Ae x\,(x\inc y)$ and $X(z)\neq M$, then $\mathcal{M}\true_X(\Ae y\ince z)\,\varphi$, but $\mathcal{M}\ntrue_X(\Ae y\inc z)\,\varphi$. 

By the observation above, we see that definability of these quantifiers, as well as many other operators for team semantics, is ``case sensitive''. That is, if a certain operator $O$ is definable in a logic $\mathcal{L}$ and $\mathcal{L}'$ is an extension of $\mathcal{L}$, then the operator $O$ may have to be defined differently in $\mathcal{L'}$. Note that atoms in team semantics are more regular in this sense, since if a certain atom $A$ is definable in a logic $\mathcal{L}$, then $A$ can be defined in all of the extensions of $\mathcal{L}$ identically as it is defined in $\mathcal{L}$.

Since we were able to define universal inclusion quantifier $(\Ae\tuple x\inc\tuple{\vphantom\wedge\smash{t}}\,)$ in EXC, it would have been natural to predict that universal exclusion quantifier $(\Ae\tuple x\exc\tuple{\vphantom\wedge\smash{t}}\,)$ is dually definable in INC. However, this is impossible since this operator is not closed under unions as shown in Observation~\ref{obs: Properties of quantifiers}. Here we have an interesting piece of asymmetry between the inclusion and exclusion operators.

In this subsection we were able to show that \emph{existential} inclusion and exclusion quantifiers are very closely related to inclusion and exclusion atoms. However, perhaps a bit surprisingly, with \emph{universal} inclusion and exclusion quantifiers, this relationship becomes more complicated.
One interesting question, that is still open, is the exact expressive power of universal exclusion quantifier. For now, we only know that when $\tuple x$ and $\tuple{\vphantom\wedge\smash{t}}$ are $k$-ary, then $(\Ae\tuple x\exc\tuple{\vphantom\wedge\smash{t}}\,)$ is (strictly) stronger than $k$-ary inclusion atom. However, it is possible that this difference would disappear on the level of sentences -- that is, FO extended with $(\Ae\tuple x\exc\tuple{\vphantom\wedge\smash{t}}\,)$ (where $\tuple x,\tuple{\vphantom\wedge\smash{t}}$ are $k$-ary) would become equivalent with INC[$k$] when we only consider sentences. We leave this question open for further research.


\subsection{Term value preserving disjunction}\label{ssec: Term value preserving disjunction}


When evaluating disjunctions, the team is split and usually some information is lost about the values of terms in the original team. Often this is desirable, since we want to shrink or distribute the values of certain variables by giving conditions on the disjuncts.

However, sometimes we want that the values of certain terms (or tuples of terms) are preserved on both sides after the evaluation of the disjunction. This is desirable especially when we are using variables to carry information about sets (or tuples of variables to carry information about relations). This method will be crucial in the proof of Theorem \ref{the: Expressing ESO[k]-formula with INEX[k]} later in this paper.

For this purpose we introduce \emph{term value preserving disjunction}. It can be defined by using constancy atoms, intuitionistic disjunctions and inclusion atoms of the same arity as the lengths of the tuples whose values we want to preserve. Thus, with this operator, the values of single terms can be preserved in INEX[$1$] and the values of $k$-tuples of terms can be preserved in INEX[$k$].

\begin{definition}\label{def: Term value preserving disjunction}
Let $\tuple{\vphantom\wedge\smash{t}}_1,\dots,\tuple{\vphantom\wedge\smash{t}}_n$ be $k$-tuples of $L$-terms, $\varphi,\psi\in\INEXset$ and $c_l,c_r,y$ fresh variables. We define
\begin{align*}
	\varphi\!\!\underset{\scriptscriptstyle\tuple{\vphantom\wedge\smash{t}}_1,\dots,\tuple{\vphantom\wedge\smash{t}}_n}{\vee}\!\!\psi \,:=\,& 
	(\varphi\sqcup\psi)\,\sqcup\,\Ee c_l\Ee c_r\,\bigl(\dep(c_l)\wedge\dep(c_r)
	\wedge c_l\neq c_r \\[-0,1cm]
	&\hspace{20mm}\wedge\Ee y\,\bigl(((y=c_l\wedge\varphi)\vee(y=c_r\wedge\psi)) 
	\wedge\bigwedge_{i\leq n}(\theta_i\wedge\theta_i')\bigr)\bigr), \\
	\theta_i &:= \Ee\tuple z_1\Ee\tuple z_2
	\bigl(((y=c_l\wedge\tuple z_1=\tuple{t_i}\wedge\tuple z_2=\tuple c_1\,) \\[-0,1cm]
	&\hspace{2cm}\vee(y=c_r\wedge\tuple z_1=\tuple c_1\wedge\tuple z_2=\tuple{\vphantom\wedge\smash{t}}_i))
	\wedge\tuple{\vphantom\wedge\smash{t}}_i\subseteq\tuple z_1\wedge\tuple{\vphantom\wedge\smash{t}}_i\subseteq\tuple z_2\bigr) \\
	\theta_i' &:= \Ee\tuple z_1\Ee\tuple z_2
	\bigl(((y=c_l\wedge\tuple z_1=\tuple{\vphantom\wedge\smash{t}}_i\wedge\tuple z_2=\tuple c_2\,) \\[-0,1cm]
	&\hspace{2cm}\vee(y=c_r\wedge\tuple z_1=\tuple c_2\wedge\tuple z_2=\tuple{\vphantom\wedge\smash{t}}_i))
	\wedge\tuple{\vphantom\wedge\smash{t}}_i\subseteq\tuple z_1
	\wedge\tuple{\vphantom\wedge\smash{t}}_i\subseteq\tuple z_2\bigr),
\end{align*}
where $\tuple z_1,\tuple z_2,\tuple c_1,\tuple c_2$ are $k$-tuples of variables such that the tuples $\tuple z_1,\tuple z_2$ consist of fresh variables, and $\tuple c_1,\tuple c_2$ are defined as $\tuple c_1 := c_l\dots c_l$ and $\tuple c_2:= c_r\dots c_r$.
\end{definition}
The next proposition gives the truth condition for this operator. Note that this truth condition is the same as for the normal disjunction with an extra condition that the values for the tuples $\tuple{\vphantom\wedge\smash{t}}_1,\dots,\tuple{\vphantom\wedge\smash{t}}_n$ must be preserved on both sides after splitting the team (supposing that the split is nontrivial).
\begin{proposition}\label{the: Term value preserving disjunction}
With the same assumptions as in Definition \ref{def: Term value preserving disjunction}, we obtain the following truth condition:
\begin{align*}
	\mathcal{M}\true_X \varphi\!\!\underset{\scriptscriptstyle\tuple{\vphantom\wedge\smash{t}}_1,\dots,\tuple{\vphantom\wedge\smash{t}}_n}{\vee}\!\!\psi  
	&\text{ iff there are } Y,Y'\!\subseteq\!X \text{ s.t. } Y\cup Y'\!=\!X, \, \mathcal{M}\true_{Y}\varphi, \,
	\mathcal{M}\true_{Y'}\psi \\[-0,1cm]
	&\;\text{ and if }\, Y,Y'\neq\emptyset, \text{ then }
	Y(\tuple{\vphantom\wedge\smash{t}}_i)\!=\!X(\tuple{\vphantom\wedge\smash{t}}_i)\!
	=\!Y'(\tuple{\vphantom\wedge\smash{t}}_i) \text{ for all }\,i\leq n.
\end{align*}
\end{proposition}
Before presenting a proof for this proposition, we explain its idea here briefly:  We first check if the splitting can be done so that one of the sides is the empty team. In this case we don't set any requirements since all $\INEXset$-formulas are true in the empty team and on the other side values are trivially preserved since it has to be the whole team $X$.

Otherwise we fix two constants $c_l,c_r$ which correspond to the left hand and right hand sides of the disjunction.  Then we attach a ``label'' $y$ to each assignment in the team. This label can be either $c_l$, $c_r$ or both depending on if the assignment in question will be placed on the left, on the right or both. Since these labels are attached before doing the actual splitting, we can check beforehand that the information will be preserved.

The truth of formula $\theta_i$ guarantees that values of term $t_i$ will be preserved on both sides for all values, expect possibly for the value of $\tuple c_1$ which is a constant. The formula $\theta_i'$ does the same, but it cannot make sure that the value for the constant $\tuple c_2$ is preserved. But the truth of both $\theta_i$ and $\theta_i'$ guarantees that the values for $\tuple{\vphantom\wedge\smash{t}}_i$ are indeed preserved on both sides.

\begin{proof}
(Proposition \ref{the: Term value preserving disjunction})
In this proof we use the abbreviation $\varphi\,\veebar\,\psi\,:=\,\varphi\!\!\underset{\scriptscriptstyle\tuple{\vphantom\wedge\smash{t}}_1,\dots,\tuple{\vphantom\wedge\smash{t}}_n}{\vee}\!\!\psi$.


\noindent
If $X$ would be an empty team, the claim would hold trivially, and thus we may assume that $X\neq\emptyset$. By locality we may also assume that $c_l,c_r,y\notin\dom(X)$.

\smallskip

\noindent
Suppose first that $\mathcal{M}\true_X\varphi\veebar\psi$. Now either $\mathcal{M}\true_X\varphi\sqcup\psi$ or
\begin{align*}
	\mathcal{M}\true_X\Ee c_l\Ee c_r&\bigl(\dep(c_l)\wedge\dep(c_r)\wedge c_l\neq c_r \\[-0,05cm]
	&\wedge\Ee y\,\bigl(((y=c_l\wedge\varphi)\vee(y=c_r\wedge\psi))
	\wedge\bigwedge_{i\leq n}(\theta_i\wedge\theta_i')\bigr)\bigr). \tag{$\star$} \\[-0,8cm]
\end{align*}
Suppose first that $\mathcal{M}\true_X\varphi\sqcup\psi$, i.e. $\mathcal{M}\true_X\varphi$ or $\mathcal{M}\true_X\psi$. If $\mathcal{M}\true_X\varphi$, then we can choose $Y:=X$ and $Y':=\emptyset$, when the claim holds trivially. Analogously if $\mathcal{M}\true_X\psi$, we can choose $Y:=\emptyset$ and $Y':=X$. 
Suppose then that ($\star$) holds. Now there exist $F_1:X\rightarrow\mathcal{P}^*(M)$ and $F_2:X[F_1/c_l]\rightarrow\mathcal{P}^*(M)$ such that
\begin{align*}
	\mathcal{M}\true_{X_1}&\dep(c_l)\wedge\dep(c_r)\wedge c_l\neq c_r \\[-0,05cm]
	&\quad\wedge\Ee y\,\bigl(((y=c_l\wedge\varphi)
	\vee(y=c_r\wedge\psi))\wedge\bigwedge_{i\leq n}(\theta_i\wedge\theta_i')\bigr), \\[-0,8cm]
\end{align*}
where $X_1:=X[F_1/c_l,F_2/c_r]$.
Since $\mathcal{M}\true_{X_1}\!\dep(c_l)$, $\mathcal{M}\true_{X_1}\!\dep(c_r)$, $\mathcal{M}\true_{X_1}\!c_l\!\neq\! c_r$ and $X\neq\emptyset$, there exist $a,b\in M$ such that $X_1(c_l)=\{a\}$, $X_1(c_r)=\{b\}$ and $a\neq b$. There also exists a function $F_3:X_1\rightarrow\mathcal{P}^*(M)$ such that 
\[
	\mathcal{M}\true_{X_2}((y=c_l\wedge\varphi)\vee(y=c_r\wedge\psi))
	\wedge\bigwedge_{i\leq n}(\theta_i\wedge\theta_i'), \;\text{ where } X_2:=X_1[F_3/y].
\]
Now there exist $Z_1,Z_1'\subseteq X_2$, such that $Z_1\cup Z_1'=X_2$, $\mathcal{M}\true_{Z_1} y=c_l\wedge\varphi$ and $\mathcal{M}\true_{Z_1'} y=c_r\wedge\psi$. Since $X_2(c_l)=\{a\}$, $X_2(c_r)=\{b\}$ and $a\neq b$, it is easy to see that the following holds for each $s\in X_2$:
\vspace{-0,1cm}\[
	s\in Z_1 \;\text{ iff } s(y)=a \quad\text{ and }\quad s\in Z_1' \;\text{ iff } s(y) = b.
\]
Let $Y:=Z_1\upharpoonright\dom(X)$ and $Y':=Z_1'\upharpoonright\dom(X)$.
Since $\mathcal{M}\true_{Z_1}\varphi$ and $\mathcal{M}\true_{Z_1'}\psi$, we have $\mathcal{M}\true_Y\varphi$ and $\mathcal{M}\true_{Y'}\psi$ by locality. Because $Z_1\cup Z_1'=X_2$, we must also have $Y\cup Y'=X $ (recall that we assumed that $c_l,c_r,y\notin\dom(X)$).

We still need to show that the values of $\tuple{\vphantom\wedge\smash{t}}_i$ ($i\leq n$) are preserved when $X$ is split into $Y$ and $Y'$. For the sake of showing this, let $i\leq n$, whence $\mathcal{M}\true_{X_2}\theta_i\wedge\theta_i'$. In particular  $\mathcal{M}\true_{X_2}\theta_i$ and thus there are $\mathcal{F}_1:X_2\rightarrow\mathcal{P}^*(M^k)$ and  $\mathcal{F}_2:X_2[\mathcal{F}_1/\tuple z_1]\rightarrow\mathcal{P}^*(M^k)$ such that
\begin{align*}
	\mathcal{M}\true_{X_3}&\,\bigl((y=c_l\wedge\tuple z_1=\tuple{\vphantom\wedge\smash{t}}_i\wedge\tuple z_2=\tuple c_1) \\[-0,1cm]
	&\vee(y=c_r\wedge\tuple z_1=\tuple c_1\wedge\tuple z_2=\tuple{\vphantom\wedge\smash{t}}_i)\bigr)
	\wedge\tuple{\vphantom\wedge\smash{t}}_i\subseteq\tuple z_1\wedge\tuple{\vphantom\wedge\smash{t}}_i\subseteq\tuple z_2,
\end{align*}
where $X_3 = X_2[\mathcal{F}_1/\tuple z_1,\mathcal{F}_2/\tuple z_2]$.
Now there are $Z_2,Z_2'\subseteq X_3$ s.t. $Z_2\!\cup\!Z_2'=X_3$ and
\[
	\begin{cases}
 		\mathcal{M}\true_{Z_2}y=c_l\wedge\tuple z_1=\tuple{\vphantom\wedge\smash{t}}_i\wedge\tuple z_2=\tuple c_1 \\
		\mathcal{M}\true_{Z_2'}y=c_r\wedge\tuple z_1=\tuple c_1\wedge\tuple z_2=\tuple{\vphantom\wedge\smash{t}}_i.
	\end{cases}
\]
Let $\tuple a:=(a,\dots,a)$ and $\tuple b:=(b,\dots,b)$.
For the sake of showing that $X(\tuple{\vphantom\wedge\smash{t}}_i)\subseteq Y(\tuple{\vphantom\wedge\smash{t}}_i)\cup\{\tuple a\}$, let $\tuple c\in X(\tuple{\vphantom\wedge\smash{t}}_i)$. Now there is $s\in X$ such that $s(\tuple{\vphantom\wedge\smash{t}}_i)=\tuple c$, whence there is $r\in X_3$ such that $r(\tuple{\vphantom\wedge\smash{t}}_i)=s(\tuple{\vphantom\wedge\smash{t}}_i)$. Since $\mathcal{M}\true_{X_3}\tuple{\vphantom\wedge\smash{t}}_i\subseteq\tuple z_1$, there exists $r'\in X_3$ such that $r'(\tuple z_1)=r(\tuple{\vphantom\wedge\smash{t}}_i)$. Now we have $\tuple c = s(\tuple{\vphantom\wedge\smash{t}}_i) = r(\tuple{\vphantom\wedge\smash{t}}_i) = r'(\tuple z_1)$.

Suppose first $r'\in Z_2$. Then $r'(\tuple z_1)=r'(\tuple{\vphantom\wedge\smash{t}}_i)$ and $r'(y)=r'(c_l)=a$. Hence there is $s'\in Y$ s.t. $s'(\tuple{\vphantom\wedge\smash{t}}_i)=r'(\tuple{\vphantom\wedge\smash{t}}_i)$. Now $\tuple c = r'(\tuple z_1) = r'(\tuple{\vphantom\wedge\smash{t}}_i) = s'(\tuple{\vphantom\wedge\smash{t}}_i) \in Y(\tuple{\vphantom\wedge\smash{t}}_i)$.
If $r'\notin Z_2$, then $r'\in Z_2'$, whence we have $\tuple c=r'(\tuple z_1)=r'(\tuple c_1)=r'(c_l\dots c_l)=\tuple a$. Hence in either case $\tuple c\in Y(\tuple{\vphantom\wedge\smash{t}}_i)\cup\{\tuple a\}$ and thus $X(\tuple{\vphantom\wedge\smash{t}}_i)\subseteq Y(\tuple{\vphantom\wedge\smash{t}}_i)\cup\{\tuple a\}$.

By using the fact that $\mathcal{M}\true_{X_2}\theta_i'$, we can analogously deduce the inclusion $X(\tuple{\vphantom\wedge\smash{t}}_i)\subseteq Y(\tuple{\vphantom\wedge\smash{t}}_i)\cup\{\tuple{\vphantom t\smash{b}}\,\}$. Since $\tuple a\neq\tuple{\vphantom t\smash{b}}$, it thus has to be that $X(\tuple{\vphantom\wedge\smash{t}}_i)\subseteq Y(\tuple{\vphantom\wedge\smash{t}}_i)$. Clearly $Y(\tuple{\vphantom\wedge\smash{t}}_i)\subseteq X(\tuple{\vphantom\wedge\smash{t}}_i)$, and therefore we have $Y(\tuple{\vphantom\wedge\smash{t}}_i)=X(\tuple{\vphantom\wedge\smash{t}}_i)$. By using a symmetric argumentation we can also show that $Y'(\tuple{\vphantom\wedge\smash{t}}_i)=X(\tuple{\vphantom\wedge\smash{t}}_i)$.

\medskip

\noindent
Suppose then that there exist $Y,Y'\subseteq X$ such that $Y\cup Y'=X$, $\mathcal{M}\true_Y\varphi$ and $\mathcal{M}\true_{Y'}\psi$, and if $Y,Y'\neq\emptyset$, then we have $Y(\tuple{\vphantom\wedge\smash{t}}_i)\!=\!Y'(\tuple{\vphantom\wedge\smash{t}}_i)\!=\!X(\tuple{\vphantom\wedge\smash{t}}_i)$ for each~$i\leq n$.

If $Y\!=\!\emptyset$, then $Y'\!=\!X$ and thus $\mathcal{M}\true_X\psi$. Therefore $\mathcal{M}\true_X\varphi\sqcup\psi$ and thus $\mathcal{M}\true\varphi\veebar\psi$. And if $Y'=\emptyset$, we obtain $\mathcal{M}\true\varphi\veebar\psi$ by a similar argumentation. Hence we may assume $Y,Y'\neq~\emptyset$, whence $Y(\tuple{\vphantom\wedge\smash{t}}_i)=Y'(\tuple{\vphantom\wedge\smash{t}}_i)=X(\tuple{\vphantom\wedge\smash{t}}_i)$ for each $i\leq n$.

We first examine the special case when $\abs{M}=1$. Because $X\neq\emptyset$, the team $X$ has to be a singleton $\{s\}$ for some $s$. Since $Y,Y'\neq\emptyset$, we have $Y=X$ and $Y'=X$. Therefore $\mathcal{M}\true_X\varphi\sqcup\psi$ and thus we have $\mathcal{M}\true_X\varphi\veebar\psi$.
Hence we may assume that $\abs{M}\geq 2$, whence there are $a,b\in M$ such that $a\neq b$. 

Let $F_1:X\rightarrow\mathcal{P}^*(M)$ s.t. $s\mapsto\{a\}$ and let $F_2:X[F_1/c_l]\rightarrow\mathcal{P}^*(M)$ s.t. $s\mapsto\{b\}$. We write $X_1:=X[F_1/c_l,F_2/c_r]$. By the definitions of $F_1$ and~$F_2$, we clearly have $\mathcal{M}\true_{X_1}\dep(c_l)$, $\mathcal{M}\true_{X_1}\dep(c_r)$ and $\mathcal{M}\true_{X_1}c_l\neq c_r$. Let
\[
	F_3:X_1\rightarrow\mathcal{P}^*(M)\, \text{ s.t. }
	\begin{cases}
		s\mapsto\{a\} \hspace{0,48cm}\text{ if }\; s\upharpoonright\dom(X)\in Y\setminus Y' \\
		s\mapsto\{b\} \hspace{0,51cm}\text{ if }\; s\upharpoonright\dom(X)\in Y'\setminus Y \\
		s\mapsto\{a,b\} \;\text{ if }\; s\upharpoonright\dom(X)\in Y\cap Y'.
	\end{cases}
\]
We define the following teams $X_2:=X_1[F_3/y]$, $Z_1:=\{s\in X_3\mid s(y)=a\}$ and $Z_1':=\{s\in X_3\mid s(y)=b\}$. Now it clearly holds that $Z_1\cup Z_1'=X_2$, $\mathcal{M}\true_{Z_1}y=c_l$ and $\mathcal{M}\true_{Z_1'}y=c_r$. By locality and the definition of $F_3$, we have $\mathcal{M}\true_{Z_1}\varphi$ and $\mathcal{M}\true_{Z_1'}\psi$. Therefore $\mathcal{M}\true_{X_2}(y=c_l\wedge\varphi)\vee(y=c_r\wedge\psi)$.

Let $i\leq n$. We define $\tuple a:=(a,\dots,a)$ and
\[
	\begin{cases}
		\mathcal{F}_1:X_2\rightarrow\mathcal{P}^*(M^k)\, \text{ s.t. }
		\begin{cases}
			s\mapsto\{s(\tuple{\vphantom\wedge\smash{t}}_i)\}\, \text{ if } s(y)=a \\
			s\mapsto\{\tuple a\}\, \quad\;\,\text{ if } s(y)=b
		\end{cases} \\
		\mathcal{F}_2:X_2[\mathcal{F}_1/\tuple z_1]\rightarrow\mathcal{P}^*(M^k)\, \text{ s.t. }
		\begin{cases}
			s\mapsto\{\tuple a\}\, \quad\;\,\text{ if } s(y)=a \\
			s\mapsto\{s(\tuple{\vphantom\wedge\smash{t}}_i)\}\, \text{ if } s(y)=b.
		\end{cases}
	\end{cases}
\]
Let $X_3:=X_2[\mathcal{F}_1/\tuple z_1,\mathcal{F}_2/\tuple z_2]$, $Z_2:=\{s\in X_3\mid s(y)=a\}$ and $Z_2':=\{s\in X_3\mid s(y)=b\}$. Now $Z_2\cup Z_2'=X_3$ and by the definitions of $\mathcal {F}_1$ and $\mathcal{F}_2$ we have
\[
	\mathcal{M}\true_{Z_2}y=c_l\wedge\tuple z_1=\tuple{\vphantom\wedge\smash{t}}_i\wedge\tuple z_2=\tuple c_1
	\quad\text{and}\quad
	\mathcal{M}\true_{Z_2'}y=c_r\wedge\tuple z_1=\tuple c_1\wedge\tuple z_2=\tuple{\vphantom\wedge\smash{t}}_i.
\]
For the sake of showing that $\mathcal{M}\true_{X_3}\tuple{\vphantom\wedge\smash{t}}_i\inc\tuple z_1$, let $r\in X_3$. Now there is $s\in X$, s.t. $r(\tuple{\vphantom\wedge\smash{t}}_i) = s(\tuple{\vphantom\wedge\smash{t}}_i)$. Since $s(\tuple{\vphantom\wedge\smash{t}}_i)\in X(\tuple{\vphantom\wedge\smash{t}}_i)=Y(\tuple{\vphantom\wedge\smash{t}}_i)$, there is $s'\in Y$, such that $s'(\tuple{\vphantom\wedge\smash{t}}_i)=s(\tuple{\vphantom\wedge\smash{t}}_i)$. Let $r':=s'[a/c_l,b/c_r,a/y,s'(\tuple{\vphantom\wedge\smash{t}}_i)/\tuple z_1,\tuple a/\tuple z_2]$. Now $r'\in X_3$ and $r(\tuple{\vphantom\wedge\smash{t}}_i) = s(\tuple{\vphantom\wedge\smash{t}}_i) = s'(\tuple{\vphantom\wedge\smash{t}}_i) = r'(\tuple z_1)$.
Hence $\mathcal{M}\true_{X_3}\tuple{\vphantom\wedge\smash{t}}_i\inc\tuple z_1$. Analogously we can show that $\mathcal{M}\true_{X_3}\tuple{\vphantom\wedge\smash{t}}_i\inc\tuple z_2$ and thus $\mathcal{M}\true_{X_2}\theta_i$. By a similar argumentation $\mathcal{M}\true_{X_2}\theta_i'$ and thus $\mathcal{M}\true_{X_2}\bigwedge_{i\leq n}(\theta_i\wedge\theta_i')$. Hence ($\star$) holds, and therefore $\mathcal{M}\true_X\varphi\veebar\psi$.
\end{proof}

\begin{remark}
The tuples $\tuple{\vphantom\wedge\smash{t}}_1,\dots,\tuple{\vphantom\wedge\smash{t}}_n$ of terms, in term value preserving disjunction for $k$-tuples, could also be of different lengths (at most $k$) since we can repeat the last term  in a tuple several times in order to make it a $k$-tuple. 
\end{remark}

Term value preserving disjunction has several natural variants. The version we defined requires that the values of given tuples of terms are preserved both on the left and right side of the disjunction. We could weaken this condition by requiring these values to be preserved only on the left, only on the right or only on either of the sides without specifying which. Or we could modify this condition by requiring different tuples of terms to be preserved on the left and different tuples to be preserved on the right.

Now we allow the splitting to be done in such a way that either of the sides becomes empty, which is natural for our needs since INEX has empty team property. But strictly speaking, the values of the given terms are not necessarily preserved in this case, since there are no values in the empty team. If we require values to be preserved in then as well, we can additionally require that splitting must be done in a way that neither of the sides becomes empty. If we only require this condition -- ignoring the values of any terms -- we obtain a disjunction that can be seen as a dual operator for intuitionistic disjunction\footnote{Intutionistic disjunction states that the splitting must be done in a way that either of the sides becomes empty -- a dual condition is that neither of the sides can be left empty.}.

We will not go into details here, but all of the variants described above can be defined in INEX. We just need to do some simple modifications on the formula that defines term value preserving disjunction in Definition~\ref{def: Term value preserving disjunction}. In this paper we use term value preserving disjunction only as a useful tool in INEX, but it would be interesting to study the properties and the expressive power of this operator (or some of its variants) independently. 
We could also add it to some related logics and see how it affects their expressive power.


\subsection{Relativization method for team semantics}\label{ssec: Relativization}


In this subsection we introduce an application which uses several of the new operators that we have defined in this section. Suppose that $\varphi$ is an $\INEXset$-sentence and $y\notin\vr(\varphi)$ is a variable in the domain of a team $X$. If we replace all quantifiers $\Ee x,\Ae x$ in $\varphi$ with the corresponding inclusion quantifiers $(\Ee x\inc y)$, $(\Ae x\inc y)$, the evaluation of the resulting formula is identical to evaluation of~$\varphi$, except that the quantifiers in $\varphi$ may only choose values within the values of~$y$. If we further replace disjunctions in $\varphi$ with the ones that preserve the value of $y$, then the quantifications may only choose values within the set $X(y)$ (the \emph{initial} values of $y$ in $X$). Since the resulting formula only ``sees'' the part of model that is restricted to the set $X(y)$, we call this method \emph{relativization}.

\begin{definition}
Let $\varphi$ be an $\INEXset$-sentence and let $y\notin\vr(\varphi)$ be a variable. \emph{The relativization of $\varphi$ on $y$}, denoted by $\varphi\!\upharpoonright\!y$, is defined  recursively:
\begin{align*}
	\psi\!\upharpoonright\!y &= \psi \quad \text{if $\psi$ is a literal or inclusion/exclusion atom} \\
	(\psi\wedge\theta)\!\upharpoonright\!y &= \psi\!\upharpoonright\!y\,\wedge\,\theta\!\upharpoonright\!y \\
	(\psi\vee\theta)\!\upharpoonright\!y &= \psi\!\upharpoonright\!y\,\veebar\,\theta\!\upharpoonright\!y,
	\quad \text{ where } \veebar\,:=\,\underset{y}{\vee} \\
	(\Ee x\,\psi)\!\upharpoonright\!y &= (\Ee x\!\subseteq\!y)(\psi\!\upharpoonright\!y) \\
	(\Ae x\,\psi)\!\upharpoonright\!y &= (\Ae x\!\subseteq\!y)(\psi\!\upharpoonright\!y).
\end{align*}
Note that since $\varphi$ is a sentence, we have $\fr(\varphi\!\upharpoonright\!y)=\{y\}$. 
Any formula $\varphi$ (of any logic with team semantics) could be relativized on any variable $y$ as above, but here we only examine a special case when it is applied to $\INEXset$-sentences. 
\end{definition}

Let $X$ be a team and $y\in\dom(X)\setminus\vr(\varphi)$. If $\varphi$ defines some property of the domain of a model, then the formula $\varphi\!\upharpoonright\!y$ defines the same property of the set values for $y$ in the team $X$. This is proven in Proposition~\ref{the: Relativization} below. This proposition could be proven also for many other logics $\mathcal{L}$ with team semantics. If the following assumptions hold for $\mathcal{L}$, the proof can be done identically as it is done here: $\mathcal{L}$ is an extension of FO with new atomic formulas, it is local, has empty team property, and inclusion quantifiers (for single variables) and term value preserving disjunction (for single terms) can expressed in $\mathcal{L}$. Note that in order to express these operators, it would suffice that we could use \emph{unary} inclusion and exclusion atoms in $\mathcal{L}$.

If $\mathcal{M}=(\mathcal{I},M)$ is an $L$-model and $A\subseteq M$, the notation $\mathcal{M}\upharpoonright A$ denotes the submodel of $\mathcal{M}$ that is \emph{relativized} on $A$. That is, the universe of $\mathcal{M}\upharpoonright A$ is the set $A$ and the symbols in $L$ are interpreted as: $R^{\mathcal{M}\upharpoonright A}=R^\mathcal{M}\upharpoonright A^n$ for $n$-ary relation symbols $R\in L$, $f^{\mathcal{M}\upharpoonright A}=f^\mathcal{M}\upharpoonright A^n$ for $n$-ary function symbols $f\in L$ and $c^{\mathcal{M}\upharpoonright A}=c^\mathcal{M}$ for constant symbols $c\in L$. Note that if $L$ contains function or constant symbols, then $\mathcal{M}\upharpoonright A$ can be an $L$-model only if $f^\mathcal{M}\upharpoonright A^n: A^n\rightarrow A$ for all each $n$-ary $f\in L$ and $c^\mathcal{M}\in A$ for each $c\in L$. But if $L$ is relational, then $\mathcal{M}\upharpoonright A$ is an $L$-model for any $A\subseteq M$.
\begin{proposition}\label{the: Relativization}
Let $\varphi$ be an $\INEXset$-sentence and $y$ be a variable such that $y\notin\vr(\varphi)$. Now
\[
	\mathcal{M}\true_X \varphi\!\upharpoonright\!y \quad\text{iff}\quad \mathcal{M}\upharpoonright X(y) \true\varphi
\]
for all $L$-models $\mathcal{M}$ and teams $X$ such that $\mathcal{M}\upharpoonright X(y)$ is an $L$-model.
\end{proposition}

\begin{proof}
We first show that
\[
	\text{If }\mathcal{M}\true_X\mu\!\upharpoonright\!y, 
	\;\text{ then }\; \mathcal{M}\upharpoonright X(y)\true_X\mu, \tag{R1}
\]
for all $\mu\in\subf(\varphi)$ and teams $X$ for which the following condition holds:
\begin{equation*}
	X(z)\subseteq X(y) \text{ for all } z\in\dom(X). \tag{$\star$}
\end{equation*}
Note that if the condition ($\star$) would not hold, then $X$ would not be a team for the model $\mathcal{M}\upharpoonright X(y)$.
We prove the claim (R1) by induction on $\mu$:
\begin{itemize}[leftmargin=*]
\item If $\mu$ is a literal or inclusion/exclusion atom, then the claim holds trivially since $\mu\!\upharpoonright\!y = \mu$ and $X(z)\subseteq X(y)$ for all $z\in\vr(\mu)$.

\item The case $\mu = \psi\wedge\theta$ is straightforward to prove.


\item Let $\mu = \psi\vee\theta$.
Suppose first that $\mathcal{M}\true_X (\psi\vee\theta)\!\upharpoonright\!y$, i.e. $\mathcal{M}\true_X \psi\!\upharpoonright\!y\,\veebar\,\theta\!\upharpoonright\!y$. Thus there exist $Y_1,Y_2\subseteq X$ s.t. $Y_1\cup Y_2 = X$, $\mathcal{M}\true_{Y_1} \psi\!\upharpoonright\!y$ and $\mathcal{M}\true_{Y_2}\theta\!\upharpoonright\!y$, and if $Y_1,Y_2\neq\emptyset$, then $Y_1(y)=Y_2(y)=X(y)$. If $Y_1=\emptyset$, then the condition ($\star$) holds trivially for $Y_1$. Also if $Y_2=\emptyset$, then ($\star$) holds for $Y_1$ as $Y_1=X$. Suppose then that $Y_1,Y_2\neq\emptyset$, whence $Y_1(y)=X(y)$. Since $Y_1\subseteq X$, we have $Y_1(z)\subseteq X(z)\subseteq X(y)=Y_1(y)$ for all $z\in\dom(X)=\dom(Y_1)$. Thus ($\star$) holds for $Y_1$ in all cases. By an analogous argumentation ($\star$) holds for $Y_2$.

Suppose first that $Y_2=\emptyset$. Now $Y_1=X$ and thus $\mathcal{M}\true_X\psi\!\upharpoonright\!y$. By the inductive hypothesis  $\mathcal{M}\upharpoonright X(y)\true_X\psi$ and thus $\mathcal{M}\upharpoonright X(y)\true_X \psi\vee\theta$. The case when $Y_1=\emptyset$ is analogous. Suppose then that $Y_1,Y_2\neq\emptyset$, whence $Y_1(y)\!=\!Y_2(y)\!=\!X(y)$. By the inductive hypothesis $\mathcal{M}\upharpoonright Y_1(y)\true_{Y_1}\psi$. Since $Y_1(y)=X(y)$, we have $\mathcal{M}\upharpoonright X(y)\true_{Y_1}\psi$. By similar argumentation we have $\mathcal{M}\upharpoonright X(y)\true_{Y_2}\theta$ and thus $\mathcal{M}\upharpoonright X(y)\true_{X}\psi\vee\theta$.

\item Let $\mu = \Ee x\,\psi$.
Suppose $\mathcal{M}\true_X (\Ee x\,\psi)\!\upharpoonright\!y$, i.e. $\mathcal{M}\true_X (\Ee x\!\subseteq\!y)(\psi\!\upharpoonright\!y)$. Thus there exists $F:X\rightarrow\mathcal{P}^*(X(y))$ such that $\mathcal{M}\true_{X'}\psi\!\upharpoonright\!y$, where $X'=X[F/x]$. Since $X(z)\subseteq X(y)=X'(y)$ for all $z\in\dom(X)$ and $X'(x)\subseteq X(y)=X'(y)$, the condition $(\star)$ holds for the team $X'$. Thus, by the inductive hypothesis, $\mathcal{M}\upharpoonright X'(y)\true_{X'}\psi$. Since $X'(y)=X(y)$, we have $\mathcal{M}\upharpoonright X(y)\true_{X[F/x]}\psi$ and furthermore $\mathcal{M}\upharpoonright X(y)\true_X \Ee x\,\psi$.

\item Let $\mu = \Ae x\,\psi$.
Suppose $\mathcal{M}\true_X (\Ae x\,\psi)\!\upharpoonright\!y$, i.e. $\mathcal{M}\true_X (\Ae x\!\subseteq\!y)(\psi\!\upharpoonright\!y)$. Thus we have $\mathcal{M}\true_{X'}\psi\!\upharpoonright\!y$, where $X'=X[X(y)/x]$. Since $X(z)\subseteq X(y)=X'(y)$ for all $z\in\dom(X)$ and $X'(x)=X(y)=X'(y)$, the condition $(\star)$ holds for $X'$. Thus, by the inductive hypothesis, $\mathcal{M}\upharpoonright X'(y)\true_{X'}\psi$. Since $X'(y)=X(y)$, we have $\mathcal{M}\upharpoonright X(y)\true_{X[X(y)/x]}\psi$ and furthermore $\mathcal{M}\upharpoonright X(y)\true_X \Ae x\,\psi$.
\end{itemize}

\medskip

\noindent
We then show that if $A\subseteq M$ such that $\mathcal{M}\upharpoonright A$ is an $L$-model, then the following holds:
\[
	\text{If }\; \mathcal{M}\upharpoonright A\,\true_X\mu,
	\;\text{ then }\; \mathcal{M}\true_{X[A/y]}\mu\!\upharpoonright\!y, \tag{R2}
\]
for all $\mu\in\subf(\varphi)$ and teams $X$ for which $\dom(X)=\fr(\mu)$.
We prove this claim by induction on $\mu$:
\begin{itemize}[leftmargin=*]
\item Suppose that $\mu$ is a literal or inclusion/exclusion atom. Since $X$ is a team for the model $\mathcal{M}\!\upharpoonright\!A$, we must have $X(z)\subseteq A$ for all $z\in\dom(X)=\vr(\mu)$. We also have $\mu\!\upharpoonright\!y = \mu$ and $y\notin\fr(\mu)$, and thus the claim holds trivially.

\item The case $\mu = \psi\wedge\theta$ is straightforward to prove. 


\item Let $\mu = \psi\vee\theta$.
Suppose that $\mathcal{M}\upharpoonright A\true_{X}\psi\vee\theta$, i.e. there are $Y_1,Y_2\subseteq X$ such that $Y_1\cup Y_2 = X$, $\mathcal{M}\upharpoonright A\true_{Y_1}\psi$ and $\mathcal{M}\upharpoonright A\true_{Y_2}\theta$. Hence, by the inductive hypothesis and locality, $\mathcal{M}\true_{Y_1'}\psi\!\upharpoonright\!y$ and $\mathcal{M}\true_{Y_2'}\theta\!\upharpoonright\!y$, where $Y_1'=Y_1[A/y]$ and $Y_2'=Y_2[A/y]$. Now clearly $Y_1'\cup Y_2' = X[A/y]$ and if $Y_1',Y_2'\neq\emptyset$, then
\begin{align*}
	Y_1'(y)=Y_2'(y)=A=(X[A/y])(y).
\end{align*}
Thus we have $\mathcal{M}\true_{X[A/y]} \psi\!\upharpoonright\!y\,\veebar\,\theta\!\upharpoonright\!y$, i.e. $\mathcal{M}\true_{X[A/y]} (\psi\vee\theta)\!\upharpoonright\!y$.

\item Let $\mu = \Ee x\,\psi$.
Suppose $\mathcal{M}\upharpoonright A\true_X \Ee x\,\psi$, i.e. there is $F:X\rightarrow\mathcal{P}^*(A)$ such that $\mathcal{M}\upharpoonright A\true_{X'}\psi$, where $X'=X[F/x]$. Thus, by the inductive hypothesis and locality, $\mathcal{M}\true_{X'[A/y]}\psi\!\upharpoonright\!y$. Let
\begin{align*}
	&F':X[A/y]\rightarrow\mathcal{P}^*(A), \; s\mapsto F(s\!\upharpoonright\fr(\mu))
	\;\text{ and }\; X'' := (X[A/y])[F'/x].
\end{align*}
Note that $F'$ is well-defined since $\dom(X)=\fr(\mu)$ by the assumption. By the definition of $F'$, we have $X''=X'[A/y]$ and thus $\mathcal{M}\true_{X''}\psi\!\upharpoonright\!y$. We also have $\im(F')=\im(F)\subseteq\mathcal{P}^*(A)=\mathcal{P}^*((X[A/y])(y))$, and therefore $\mathcal{M}\true_{X[A/y]} (\Ee x\!\subseteq\!y)(\psi\!\upharpoonright\!y)$, i.e. $\mathcal{M}\true_ {X[A/y]} (\Ee x\,\psi)\!\upharpoonright\!y$.

\item Let $\mu = \Ae x\,\psi$.
Suppose that $\mathcal{M}\!\upharpoonright\!A\true_X \Ae x\,\psi$, i.e. $\mathcal{M}\!\upharpoonright\!A\true_{X'}\psi$, where $X'=X[A/x]$. By the inductive hypothesis and locality, $\mathcal{M}\true_{X'[A/y]}\psi\!\upharpoonright\!y$. Let $X''=(X[A/y])[A/x]$. Now $X''=X'[A/y]$, and thus we have $\mathcal{M}\!\upharpoonright\!A\true_{X''}\psi\!\upharpoonright\!y$. Since $(X[A/y])(y)=A$, it holds that $\mathcal{M}\true_{X[A/y]} (\Ae x\!\subseteq\!y)(\psi\!\upharpoonright\!y)$, i.e. $\mathcal{M}\true_{X[A/y]} (\Ae x\,\psi)\!\upharpoonright\!y$.
\end{itemize}

\medskip

\noindent
We are now ready to prove the claim of this proposition:
\[
	\mathcal{M}\true_X \varphi\!\upharpoonright\!y \quad\text{iff}\quad \mathcal{M}\upharpoonright X(y) \true\varphi.
\]

Suppose first that $\mathcal{M}\true_X \varphi\!\upharpoonright\!y$. By locality we have $\mathcal{M}\true_{X'}\varphi\!\upharpoonright\!y$, where $X'=X\upharpoonright\fr(\varphi\!\upharpoonright\!y)$. Since $\varphi$ is a sentence, $\dom(X')=\fr(\varphi\!\upharpoonright\!y)=\{y\}$, and thus the condition ($\star$) holds trivially for the team $X'$.  Hence by (R1) we have $\mathcal{M}\upharpoonright X'(y) \true_{X'}\varphi$. Since $X'(y)=X(y)$, we have $\mathcal{M}\upharpoonright X(y) \true_{X'}\varphi$ and thus by locality $\mathcal{M}\upharpoonright X(y) \true\varphi$.

Suppose then that $\mathcal{M}\upharpoonright X(y) \true\varphi$. Now by (R2), we have $\mathcal{M}\true_{\{\emptyset\}[X(y)/y]} \varphi\!\upharpoonright\!y$. Since $X\upharpoonright\{y\}=\{\emptyset\}[X(y)/y]$, we have $\mathcal{M}\true_{X\upharpoonright\{y\}} \varphi\!\upharpoonright\!y$. Since $\fr(\varphi\!\upharpoonright\!y)=\{y\}$, by locality $\mathcal{M}\true_X \varphi\!\upharpoonright\!y$.
\end{proof}

The relativization method gives us a simple way to express properties of certain sets of values in a team. We can apply the same technique for many other logics with team semantics if we extend them with unary inclusion and exclusion atoms. For example, there is a dependence logic sentence $\varphi$ which expresses that a model has even cardinality (\cite{Vaananen07}). Now the formula $\varphi\!\upharpoonright\!y$ expresses that the variable $y$ has even number of different values in a team. We will give more examples on this method in Section~\ref{sec: Examples}.


\section{The expressive power of INEX[$k$]}\label{sec: Translations}

In this section we will analyze the expressive power of INEX[$k$]. We first present translations from INC[$k$] and EXC[$k$] to ESO[$k$] and then combine them to form a translation from INEX[$k$] to ESO[$k$]. For the  other direction we show that any ESO[$k$]-formula, with at most $k$-ary free relation variables, can be expressed in INEX[$k$].


\subsection{Translation from INEX[$k$] to ESO[$k$]}\label{ssec: Expressing INEX}


For the language $\ESOset$ we also need a set of \emph{relation variables} which are symbols not in the vocabulary $L$. These relation variables can appear in atomic formulas similarly as relation symbols in $L$ and they can also be existentially quantified. We require all of these \emph{second order quantifiers} to appear in front of the $\ESOset$-formula, before its \emph{first order part}.

In the language $\ESOset[k]$ we only allow existential quantification of at most $k$-ary relation variables, but free relation variables in a formula may have any arity. Hence $\ESOset[0]$-fomulas are second order quantifier free, but may contain free relation variables.  If an $\ESOset$-formula $\Phi$ has free relation variables $R_1,\dots,R_n$, we can emphasize this by writing $\Phi$ as $\Phi(R_1\dots R_n)$. In this paper we will not consider ESO-formulas with free first order variables and thus their first order part can be seen as FO-sentence.
\footnote{To compare ESO-formula with free first order variables with INEX-formulas in a natural way, we would have to define team semantics also for ESO. But there are several possible ways to interpret second order quantifications in such semantics for ESO, and this topic is out of the scope of the current paper. }
After evaluating all second order quantifications, the truth of $\Phi$ in depends only on the first order part of $\Phi$. We may then apply team semantics for the first order part of $\Phi$ in any suitable model, whence flatness and locality properties hold as well.


Let $\mathcal{L}$ be any logic with team semantics and let $\varphi(\tuple y\,)$ be an $\mathcal{L}$-formula. The truth of $\varphi$ depends on a model $\mathcal{M}$ and a team $X$. If $\mathcal{L}$ is local, it is sufficient to consider the team $X\!\upharpoonright\!\vr(\tuple y)$ that is determined by the relation $X(\tuple y\,)$ (notice that the ordering of the variables in $\vec y$ here needs to be \emph{fixed}). Therefore it is natural to compare $\varphi$ with an ESO-formula $\Phi(R)$ and check whether the relations in $M$ that satisfy $\Phi$ correspond to the relations $X(\tuple y\,)$, where $X$ satisfies $\varphi$. Thus we say that $\varphi$ and $\Phi$ are equivalent if we have
\[
	\mathcal{M}\true_X\varphi\, \,\text{ iff }\, \mathcal{M}[X(\tuple y\,)/R]\true\Phi.
\]
The $\mathcal{L}$-formula $\varphi(\tuple y)$ defines a class of models and teams that satisfy it. If ESO-formula $\Phi(R)$ is equivalent with $\varphi$, it defines exactly the same models and teams by defining the relations that correspond to those teams.

\subsubsection*{Translation from EXC[$k$] to ESO[$k$]}


In the next theorem we formulate a translation from EXC[$k$] to ESO[$k$]. The idea of the proof is that we quantify a separate relation variable $P$ for each occurrence of an exclusion atom $\tuple{\vphantom\wedge\smash{t}}_1\mid\tuple{\vphantom\wedge\smash{t}}_2$. The values quantified for $P$ are the limit for the values that $\tuple{\vphantom\wedge\smash{t}}_1$ can get and $\tuple{\vphantom\wedge\smash{t}}_2$ cannot get, when $\tuple{\vphantom\wedge\smash{t}}_1\mid\tuple{\vphantom\wedge\smash{t}}_2$ is evaluated.

\begin{theorem}\label{the: Expressing EXC[k]-formula with ESO[k]}
Let $\varphi(\tuple y\,)\in\EXCset[k]$. Now there exists an $\ESOset[k]$-formula $\Phi(R)$, for which 
\[
	\mathcal{M}\true_X\varphi\; \text{ iff }\, \mathcal{M}[X(\tuple y\,)/R]\true\Phi.
\]
\end{theorem}

\begin{proof}
Without loss of generality we may assume that each exclusion atom in $\varphi$ is $k$-ary. We index these atoms by $(\tuple{\vphantom\wedge\smash{t}}_1\,|\,\tuple{\vphantom\wedge\smash{t}}_2)_1,\dots,(\tuple{\vphantom\wedge\smash{t}}_1\,|\,\tuple{\vphantom\wedge\smash{t}}_2)_n$. This is done so that each occurrence of an exclusion atom has a unique index. Let $P_1,\dots,P_n$ be $k$-ary relation variables.
Let $\psi\in\subf(\varphi)$. We define the formula $\psi'$ recursively:
\begin{align*}
	\psi' &= \psi\; \text{ if $\psi$ is a literal} \\
	((\tuple{\vphantom\wedge\smash{t}}_1\,|\,\tuple{\vphantom\wedge\smash{t}}_2)_i)' 
	&= P_i\tuple{\vphantom\wedge\smash{t}}_1\wedge\neg P_i\tuple{\vphantom\wedge\smash{t}}_2
		\;\text{ for each } i\leq n \\[0,1cm]
	(\psi\wedge\theta)' &= \psi'\!\wedge\theta', \;
	(\psi\vee\theta)' = \psi'\!\vee\theta' \\
	(\Es x\,\psi)' &= \Es x\,\psi', \;\;
	(\Ae x\,\psi)' = \Ae x\,\psi'.
\end{align*}
We can now define the formula $\Phi$ in the following way\footnote{If $\varphi$ is an $\EXCset$-sentence we define simply $\Phi := \Ee P_1\dots\Ee P_n\varphi'$.}:
\[
	\Phi := \Ee P_1\dots\Ee P_n\Ae\tuple{y}\,
	\bigl(\neg R\tuple{y}\vee(R\tuple{y}\wedge\varphi')\bigr).
\]
Clearly $\Phi$ is an $\ESOset[k]$-formula and $R$ is the only free relation variable in $\Phi$. We first need to prove the following claim.
\begin{claim}\label{star}
Let $\mu\in\subf(\varphi)$. Now the following holds for all suitable teams $X$:
\begin{align*}
	\mathcal{M}\true_X\mu\; &\text{ iff\, there exist } A_1,\dots,A_n\subseteq M^k \text{ such that } \mathcal{M}'\true_X\mu',
\end{align*}
where $\mathcal{M}':=\mathcal{M}[\tuple A/\tuple P]$ \,\emph{($=\mathcal{M}[A_1/P_1,\dots,A_n/P_n]$)}. 
\end{claim}
\noindent
We prove this claim by structural induction on $\mu$:
\begin{itemize}[leftmargin=*]
\item If $\mu$ is a literal we can set $A_i:=\emptyset$ for each $i\leq n$. Now the claim holds trivially since $\mu'=\mu$ and $P_i$ does not occur in $\mu$ for any $i\leq n$.

\item Let $\varphi = (\tuple{\vphantom\wedge\smash{t}}_1\mid\tuple{\vphantom\wedge\smash{t}}_2)_j$ for some $j\leq n$.
Suppose first that $\mathcal{M}\true_X\tuple{\vphantom\wedge\smash{t}}_1\mid\tuple{\vphantom\wedge\smash{t}}_2$. Let
\[
	\mathcal{M}' := \mathcal{M}[\tuple A/\tuple P], \text{ where }
	A_i :=
	\begin{cases}
		X(\tuple{\vphantom\wedge\smash{t}}_1)\, \text{ if } i=j \\
		\emptyset\qquad\, \text{ else}.  
	\end{cases}
	\;  (i\leq n)
\]
Because $X(\tuple{\vphantom\wedge\smash{t}}_1) = A_j = P_j^\mathcal{M'}$, we clearly have $\mathcal{M'}\true_{X}P_j\tuple{\vphantom\wedge\smash{t}}_1$.

For the sake of contradiction, suppose that there exists $s\in X$ for which $s(\tuple{\vphantom\wedge\smash{t}}_2)\in P_j^{\mathcal{M}'}$. Since $P_j^{\mathcal{M}'}=X(\tuple{\vphantom\wedge\smash{t}}_1)$, there exists $s'\in X$ s.t $s'(\tuple{\vphantom\wedge\smash{t}}_1) = s(\tuple{\vphantom\wedge\smash{t}}_2)$. But this is a contradiction since by the assumption $\mathcal{M}\true_X\,\tuple{\vphantom\wedge\smash{t}}_1\!\mid\!\tuple{\vphantom\wedge\smash{t}}_2$. Therefore $\mathcal{M'}\true_X\neg P_j\tuple{\vphantom\wedge\smash{t}}_2$ and thus  $\mathcal{M'}\true_X P_j\tuple{\vphantom\wedge\smash{t}}_1\wedge\neg P_j\tuple{\vphantom\wedge\smash{t}}_2$, i.e. $\mathcal{M'}\true_X((\tuple{\vphantom\wedge\smash{t}}_1\mid\tuple{\vphantom\wedge\smash{t}}_2)_j)'$.

\medskip

Suppose then that there exist $A_1,\dots,A_n\subseteq M^k$ s.t. $\mathcal{M'}\true_X((\tuple{\vphantom\wedge\smash{t}}_1\mid\tuple{\vphantom\wedge\smash{t}}_2)_j)'$. Hence $\mathcal{M'}\true_X P_j \tuple{\vphantom\wedge\smash{t}}_1$ and $\mathcal{M'}\true_X\neg P_j \tuple{\vphantom\wedge\smash{t}}_2$.
For the sake of contradiction, suppose that there are $s,s'\in X$ s.t. $s(\tuple{\vphantom\wedge\smash{t}}_1) = s'(\tuple{\vphantom\wedge\smash{t}}_2)$. Because $\mathcal{M'}\true_X P_j\tuple{\vphantom\wedge\smash{t}}_1$, we have $s(\tuple{\vphantom\wedge\smash{t}}_1)\in P_j^{\mathcal{M}'}$. But because $\mathcal{M'}\true_X\neg P_j\tuple{\vphantom\wedge\smash{t}}_2$, it has to be that $s(\tuple{\vphantom\wedge\smash{t}}_1)=s'(\tuple{\vphantom\wedge\smash{t}}_2)\notin P_j^{\mathcal{M}'}$. This is a contradiction, and thus  $\mathcal{M}\true_X\tuple{\vphantom\wedge\smash{t}}_1\mid\tuple{\vphantom\wedge\smash{t}}_2$.

\item Let $\mu = \psi\vee\theta$ \,(The case $\mu=\psi\wedge\theta$ can be proven similarly).
Suppose first that $\mathcal{M}\true_X\psi\vee\theta$. Thus there are $Y,Y'\subseteq X$ s.t. $Y\cup Y'=X$, $\mathcal{M}\true_Y\psi$ and $\mathcal{M}\true_{Y'}\theta$. By the inductive hypothesis there are $B_1,\dots,B_n\subseteq M^k$ and $B_1',\dots,B_n'\subseteq M^k$ s.t. $\mathcal{M}[\tuple B/\tuple P]\true_Y\psi'$ and $\mathcal{M}[\tuple B'/\tuple P]\true_{Y'}\theta'$. Let
\begin{align*}
	\mathcal{M}' &:= \mathcal{M}[\tuple A/\tuple P], \text{ where }
	A_i :=
	\begin{cases}
		B_i\; \text{ if $P_i$ occurs in $\psi'$} \\
		B_i'\; \text{ if $P_i$ does not occur in $\psi'$}.  
	\end{cases}
\end{align*}
Since none of $P_i$ can occur in both $\psi'$ and $\theta'$, we have $\mathcal{M}'\true_{Y}\psi'$ and $\mathcal{M}'\true_{Y'}\theta'$. Hence $\mathcal{M}'\true_{X}\psi'\vee\theta'$, i.e. $\mathcal{M}'\true_{X}(\psi\vee\theta)'$.

\smallskip
Suppose then that there are $A_1,\dots,A_n\subseteq M^k$ s.t. $\mathcal{M}'\true_{X}(\psi\vee\theta)'$. Thus $\mathcal{M}'\true_{X}\psi'\vee\theta'$, i.e. there are $Y,Y'\subseteq X$ s.t. $Y\cup Y'=X$, $\mathcal{M}'\true_Y\psi'$ and $\mathcal{M}'\true_{Y'}\theta'$. By the inductive hypothesis $\mathcal{M}\true_Y\psi$ and $\mathcal{M}\true_{Y'}\theta$, i.e. $\mathcal{M}\true_X\psi\vee\theta$.

\item The cases $\mu = \Es x\,\psi$ and $\mu = \Ae x\,\psi$ are straightforward to prove.
\end{itemize}
\noindent
Let $\mathcal{M}'=\mathcal{M}[\tuple A/\tuple P]$ for some $A_1,\dots,A_n\subseteq M^k$. Since $\fr(\varphi')=\vr(\tuple y)$, by locality it is easy to see that the following holds for all suitable teams $X$:
\[
	\mathcal{M}'\true_X\varphi'\;\text{ iff }\, 
	\mathcal{M}'[X(\tuple y\,)/R]\true\Ae\tuple{y}\,\bigl(\neg R\tuple{y}\vee(R\tuple{y}\wedge\varphi')\bigr).
\]
By combining this with the result of Claim \ref{star}, we obtain:
\begin{align*}
	\mathcal{M}\true_X\varphi \;\text{ iff }\;
	&\text{there are } A_1,\dots,A_n\subseteq M^k \\
	&\text{ s.t. }\mathcal{M}[\tuple A/\tuple P,X(\tuple y\,)/R]\true
	\Ae\tuple{y}\,\bigl(\neg R\tuple{y}\vee(R\tuple{y}\wedge\varphi')\bigr).
\end{align*}
Equivalently: \; $\mathcal{M}\true_X\varphi$ iff $\mathcal{M}[X(\tuple y\,)/R]\true\Phi$.
\end{proof}

\subsubsection*{Translation from INC[$k$] to ESO[$k$]}


In the next theorem we present a translation from INC[$k$] to ESO[$k$]. Again the idea is that we quantify a separate predicate symbol $P$ for each inclusion atom $\tuple{\vphantom\wedge\smash{t}}_1\inc\tuple{\vphantom\wedge\smash{t}}_2$, and the values of $\tuple{\vphantom\wedge\smash{t}}_1$ must be included in the values chosen for $P$. But we must also show that each value of~$P$ is a value that tuple $\tuple{\vphantom\wedge\smash{t}}_2$ gets in the team when $\tuple{\vphantom\wedge\smash{t}}_1\inc\tuple{\vphantom\wedge\smash{t}}_2$ is evaluated. For this we need special formulas, $\varphi_i'(\tuple u)$, which ``find'' the assignment that gets same values for $\tuple u$ and $\tuple{\vphantom\wedge\smash{t}}_2$ -- for any value of $\tuple u$ that is in the values chosen for $P$. 

\begin{theorem}\label{the: Expressing INC[k]-formula with ESO[k]}
Let $\varphi(\tuple y\,)$, where $\tuple y = y_1\dots y_m$, be an $\INCset[k]$-formula. Then there exists an $\ESOset[k]$-formula $\Phi(R)$, for which we have
\[
	\mathcal{M}\true_X\varphi\; \text{ iff }\, \mathcal{M}[X(\tuple y\,)/R]\true\Phi.
\]
\end{theorem}

\begin{proof}
Without loss of generality we may assume that each inclusion atom in $\varphi$ is $k$-ary. We index these atoms by $(\tuple{\vphantom\wedge\smash{t}}_1\subseteq\tuple{\vphantom\wedge\smash{t}}_2)_1,\dots,(\tuple{\vphantom\wedge\smash{t}}_1\subseteq\tuple{\vphantom\wedge\smash{t}}_2)_n$. Let $\tuple u$ be a $k$-tuple of fresh variables and $P_1,\dots,P_n$ be $k$-ary relation variables. 

\medskip

\noindent
Let $\psi\in\subf(\varphi)$. We define the formula $\psi'$ recursively:
\begin{align*}
	(\psi)' &= \psi\; \text{ if $\psi$ is a literal} \\
	((\tuple{\vphantom\wedge\smash{t}}_1\subseteq\tuple{\vphantom\wedge\smash{t}}_2)_i)' 
	&= P_i\tuple{\vphantom\wedge\smash{t}}_1 \;\text{ for each } i\leq n \hspace{2,85cm} \\[0,1cm]
	(\psi\wedge\theta)' &= \psi'\!\wedge\theta', \;
	(\psi\vee\theta)' = \psi'\!\vee\theta' \\
	(\Es x\,\psi)' &= \Es x\,\psi', \;\;
	(\Ae x\,\psi)' = \Ae x\,\psi'.
\end{align*}
Formulas $\psi_i'$ are defined recursively for all $i\leq n$:
\begin{align*}
	(\psi)_i' &= \psi\; \text{ if $\psi$ is a literal} \\
	((\tuple{\vphantom\wedge\smash{t}}_1\subseteq\tuple{\vphantom\wedge\smash{t}}_2)_j)_i' &= P_j\tuple{\vphantom\wedge\smash{t}}_1\; \text{ if } j\neq i \\
	((\tuple{\vphantom\wedge\smash{t}}_1\subseteq\tuple{\vphantom\wedge\smash{t}}_2)_i)_i' 
		&= (\tuple u = \tuple{\vphantom\wedge\smash{t}}_2)\wedge P_i\tuple{\vphantom\wedge\smash{t}}_1 \\
	(\psi\wedge\theta)_i' &= \psi_i'\wedge\theta_i' \\
	(\psi\vee\theta)_i' &= 
	\begin{cases}
		\psi_i' \qquad\;\, \text{if $(\tuple{\vphantom\wedge\smash{t}}_1\subseteq\tuple{\vphantom\wedge\smash{t}}_2)_i$ occurs in $\psi$} \\
		\theta_i' \qquad\;\text{ if $(\tuple{\vphantom\wedge\smash{t}}_1\subseteq\tuple{\vphantom\wedge\smash{t}}_2)_i$ occurs in $\theta$} \\
		\psi_i'\vee\theta_i'\, \text{ else} \\
	\end{cases} \\
	(\Es x\,\psi)_i' &= \Es x\,\psi_i' \\
	(\Ae x\,\psi)_i' & = \Es x\,\psi_i'\wedge\Ae x\,\psi'.
\end{align*}
Note that the cases of disjunction above are exclusive, since for each $i\leq n$ the inclusion atom $(\tuple{\vphantom\wedge\smash{t}}_1\subseteq\tuple{\vphantom\wedge\smash{t}}_2)_i$ can occur in at most one of the disjuncts.

\medskip

\noindent
We can now define the formula $\Phi$ in the following way\footnote{If $\varphi$ is an $\INCset$-sentence we can define $\Phi := \Ee P_1\dots\Ee P_n\bigl(\varphi'\wedge\bigwedge_{i\leq n}\Ae\tuple u\,(\neg P_i\tuple u\,\vee\varphi_i'(\tuple u))\bigr).$}:
\[
	\Phi := \Ee P_1\dots\Ee P_n\bigl(\Ae\tuple y\,\bigl(\neg R\tuple y\,\vee(R\tuple y\wedge\varphi')\bigr)
	\wedge\bigwedge_{i\leq n}\Ae\tuple u\,\bigl(\neg P_i\tuple u\,\vee
	\Ee\tuple y\,(R\tuple y\wedge\varphi_i'(\tuple u))\bigr)\bigr).
\]
Clearly $\Phi$ is an $\ESOset[k]$-formula and $R$ is the only free relation variable in $\Phi$. To complete the proof, we need to prove the following claim which demonstrates the relevance of the formulas $\varphi_i'$:

\begin{claim}\label{AT3}
The following holds for all $\mu\in\subf(\varphi)$ and all suitable teams $X$:
\begin{align*}
	\mathcal{M}\true_X\mu\, \;\text{ iff }\;\,&\text{there exist }
	A_1,\dots,A_n\subseteq M^k \text{ s.t. } \mathcal{M}[\tuple A/\tuple P]\true_X\mu', \\
	&\text{and for any } i\leq n \text{ and tuple of elements } \tuple a\in A_i \\[-0,1cm]
	&\text{there exists } s\in X \text{ s.t. } \mathcal{M}[\tuple A/\tuple P]\true_{\{s[\tuple a/\tuple u\,]\}}\mu_i'.
\end{align*}
\end{claim}
\noindent
We present a proof for this claim in the appendix. 

Note that since $\mu'$ and $\mu_i'$ do not contain the relation variable $R$, we can replace the model $\mathcal{M}[\tuple A/\tuple P]$ with the model $\mathcal{M}[\tuple A/\tuple P,X(\tuple y\,)/R]$ in Claim \ref{AT3}. 
The existence of $s\in\!X$ for each tuple $\tuple a\in A_i$ such that $s[\tuple a/\tuple u\,]$ satisfies $\varphi_i'$, guarantees that the values in~$A_i$ are included in the values $\tuple{\vphantom\wedge\smash{t}}_2$ in the team when the inclusion atom $(\tuple{\vphantom\wedge\smash{t}}_1\subseteq\tuple{\vphantom\wedge\smash{t}}_2)_i$ is evaluated. With this result we can prove the claim of this theorem; that is \, $\mathcal{M}\true_X\varphi\; \text{ iff }\, \mathcal{M}[X(\tuple y\,)/R]\true\Phi$.

\smallskip

Suppose first that $\mathcal{M}\true_X\varphi$. Now by Claim \ref{AT3} there are $A_1,\dots,A_n\subseteq M^k$ s.t. $\mathcal{M}'\true_X\varphi'$, and for all $i\leq n$ and $\tuple a\in A_i$ there exists $s\in X$ such that $\mathcal{M}'\true_{\{s[\tuple a/\tuple u\,]\}}\varphi_i'$, where $\mathcal{M}'=\mathcal{M}[\tuple A/\tuple P,X(\tuple y\,)/R]$. 
Let $j\leq n$ and let
\[
		Y:=\{r\in\{\emptyset\}[M^k/\tuple u\,]\mid r(\tuple u)\notin A_j\} \;\text{ and }\;
		Y':=\{r\in\{\emptyset\}[M^k/\tuple u\,]\mid r(\tuple u)\in A_j\}.
\]
Now $Y\cup Y'=\{\emptyset\}[M^k/\tuple u\,]$ and because $A_j=P_j^{\mathcal{M}'}$, clearly $\mathcal{M'}\true_Y \neg P_j\tuple u$. By the definition of $Y'$, we have $r(\tuple u)\in A_j$ for each $r\in Y'$. Hence, by applying the result of Claim \ref{AT3} for the  values $r(\tuple u)$, the following holds: for each $r\in Y'$ there exists $s_r\in X$ s.t. $\mathcal{M}'\true_{\{s_r[r(\tuple u)/\tuple u\,]\}}\varphi_j'$. 

Let $\mathcal{F}:Y'\rightarrow\mathcal{P}^*(M^m) \,\text{ s.t. } r\mapsto \{s_r(\tuple y)\}$.
Now $r[s_r(\tuple y)/\tuple y\,] = s_r[r(\tuple u)/\tuple u\,]$ for each $r\in Y'$ and thus $\mathcal{M}'\true_{\{r[s_r(\tuple y)/\tuple y\,]\}}\varphi_j'$ for each $r\in Y'$.
Hence by locality and flatness $\mathcal{M}'\true_{Y'[\mathcal{F}/\tuple y\,]}\varphi_j'$. Because $s_r(\tuple y)\in X(\tuple y\,)=R^\mathcal{M'}$ for each $r\in Y'$, by flatness we also have $\mathcal{M}'\true_{Y'[\mathcal{F}/\tuple y\,]}R\tuple y$ and thus $\mathcal{M}'\true_{Y'}\Ee\tuple y\,(R\tuple y\wedge\varphi_j')$. Therefore $\mathcal{M}'\true_{\{\emptyset\}[M^k/\tuple u\,]}\neg P_j\tuple u\vee\Ee\tuple y\,(R\tuple y\wedge\varphi_j')$ and thus $\mathcal{M}'\true\bigwedge_{i\leq n}\Ae\tuple u\,(\neg P_i\tuple u\,\vee\Ee\tuple y\,(R\tuple y\wedge\varphi_i'))$. 

Because $\mathcal{M}'\true_X\varphi'$ and $X(\tuple y\,)=R^\mathcal{M'}$, by locality $\mathcal{M'}\true\Ae\tuple y\,(\neg R\tuple y\,\vee(R\tuple y\wedge\varphi'))$. Therefore we can conclude that $\mathcal{M}[X(\tuple y\,)/R]\true\Phi$.

\medskip

\noindent
Suppose then that $\mathcal{M}[X(\tuple y\,)/R]\true\Phi$. Thus there exist $A_1,\dots,A_n\subseteq M^k$ such that the first order part of $\Phi$ holds in $\mathcal{M'}:=\mathcal{M}[\tuple A/\tuple P,X(\tuple y\,)/R]$. In particular, $\mathcal{M'}\true\Ae\tuple y\,(\neg R\tuple y\vee(R\tuple y\wedge\varphi'))$ and thus by locality $\mathcal{M}'\true_X\varphi'$.

For the sake of proving the right side of the equivalence of Claim~\ref{AT3}, let $j\leq n$ and $\tuple a\in A_j$. Now $\mathcal{M}'\true\Ae\tuple u\,(\neg P_j\tuple u\vee\Ee\tuple y\,(R\tuple y\wedge\varphi_j'))$, and thus there are $Y,Y'\subseteq \{\emptyset\}[M^k/\tuple u\,]$ such that $Y\cup Y'=\{\emptyset\}[M^k/\tuple u\,]$, $\mathcal{M'}\true_Y \neg P_j\tuple u$ and $\mathcal{M}'\true_{Y'}\Ee\tuple y\,(R\tuple y\wedge\varphi_j')$. Hence there is a function $\mathcal{F}:Y'\rightarrow\mathcal{P}^*(M^m)$ such that $\mathcal{M'}\true_{Y'[\mathcal{F}/\tuple y\,]}R\tuple y\wedge\varphi_j'$.

Let $r:=\emptyset[\tuple a/\tuple u\,]$, whence $r\in\{\emptyset\}[M^k/\tuple u\,]$. Since $r(\tuple u)=\tuple a\in A_j = P_j^\mathcal{M'}$ and $\mathcal{M'}\true_Y \neg P_j\tuple u$, we have $r\notin Y$ and thus $r\in Y'$. Let $\tuple{\vphantom t\smash{b}}\in\mathcal{F}(r)$ and let $s:=r[\tuple{\vphantom t\smash{b}}/\tuple y\,]$. By flatness, $\mathcal{M'}\true_{\{s\}}R\tuple y$\, and thus $s(\tuple y)\in R^\mathcal{M'}\!=X(\tuple y\,)$. Hence there exists $s'\in X$ such that $s'(\tuple y)=s(\tuple y)$. Since $\mathcal{M'}\true_{\{s\}}\varphi_j'$, by locality also $\mathcal{M'}\true_{\{s'[s(\tuple u)/\tuple u\,]\}}\varphi_j'$. Because $s(\tuple u)=r(\tuple u)=\tuple a$, by Claim \ref{AT3} we have $\mathcal{M}\true_X\varphi$.
\end{proof}

\subsubsection*{Forming a translation from INEX[$k$] to ESO[$k$]}


The next theorem shows that there is also a translation from INEX[$k$] to ESO[$k$]. This translation can be formulated by first eliminating exclusion atoms as in Theorem \ref{the: Expressing EXC[k]-formula with ESO[k]} and then inclusion atoms as in Theorem \ref{the: Expressing INC[k]-formula with ESO[k]}.

\begin{theorem}\label{the: Expressing INEX[k]-formula with ESO[k]}
Let $\varphi(\tuple y\,)\in\INEXset[k]$. Now there is an $\ESOset[k]$-formula $\Phi(R)$, for which we have
\[
	\mathcal{M}\true_X\varphi\; \text{ iff }\; \mathcal{M}[X(\tuple y\,)/R]\true\Phi.
\]
\end{theorem}

\begin{proof}
Without loss of generality we may assume that each exclusion and inclusion atom in the formula $\varphi$ is $k$-ary. We~index the exclusion atoms by $(\tuple{\vphantom\wedge\smash{t}}_1\,|\,\tuple{\vphantom\wedge\smash{t}}_2)_1,\dots,(\tuple{\vphantom\wedge\smash{t}}_1\,|\,\tuple{\vphantom\wedge\smash{t}}_2)_n$. Let $P_1,\dots P_n$ be $k$-ary relation variables.
Let $\psi\in\subf(\varphi)$. We define the formula $\psi'$ recursively as follows.
\begin{align*}
	\psi' &= \psi\; \text{ if $\psi$ is a literal} \\
	((\tuple{\vphantom\wedge\smash{t}}_1\,|\,\tuple{\vphantom\wedge\smash{t}}_2)_i)' 
	&= P_i\tuple{\vphantom\wedge\smash{t}}_1\wedge\neg P_i\tuple{\vphantom\wedge\smash{t}}_2
		\;\text{ for each } i\leq n \\
	(\tuple{\vphantom\wedge\smash{t}}_1\subseteq\tuple{\vphantom\wedge\smash{t}}_2)' 
	&= \tuple{\vphantom\wedge\smash{t}}_1\subseteq\tuple{\vphantom\wedge\smash{t}}_2 \\[0,1cm]
	(\psi\wedge\theta)' &= \psi'\!\wedge\theta', \;
	(\psi\vee\theta)' = \psi'\!\vee\theta' \\
	(\Es x\,\psi)' &= \Es x\,\psi', \;\;
	(\Ae x\,\psi)' = \Ae x\,\psi'.
\end{align*}
We can prove the equivalence of Claim \ref{star} for any $\mu\in\subf(\varphi)$ by structural induction on $\mu$: Since inclusion atoms are left as they are, their step in the induction is trivial. Other steps can be proven identically as in the proof of Claim \ref{star} within the proof of Theorem \ref{the: Expressing EXC[k]-formula with ESO[k]}. Thus we have
\[
	\mathcal{M}\true_X\varphi\, \text{ iff\, there exist } A_1,\dots,A_n\subseteq M^k
	\text{ s.t. }\mathcal{M}[\tuple A/\tuple P]\true_X\varphi'.
\]
Since $\varphi'$ contains only inclusion atoms and $\fr(\varphi)=\fr(\varphi')=\vr(\tuple y)$, we can apply Theorem \ref{the: Expressing INC[k]-formula with ESO[k]} for $\varphi'$ to get an $\ESOset$-formula $\Psi(R)$ for which we have
\[
	\mathcal{M}\true_X\varphi'\; \text{ iff }\, \mathcal{M}[X(\tuple y)/R]\true\Psi.
\]
We can now define $\Phi:=\Ee P_1\dots\Ee P_n\,\Psi$, whence $\Phi$ is an $\ESOset$-formula with the free relation variable $R$. Then we have
\begin{align*}
	\mathcal{M}\true_X\varphi \;
	&\text{ iff } \text{ there exist } A_1,\dots A_n\subseteq M^k
	\,\text{ s.t. }\mathcal{M}[\tuple A/\tuple P]\true_X\varphi' \\
	&\text{ iff } \text{ there exist } A_1,\dots A_n\subseteq M^k
	\,\text{ s.t. }\mathcal{M}[\tuple A/\tuple P,X(\tuple y)/R]\true\Psi \\
	&\text{ iff }\; \mathcal{M}[X(\tuple y)/R]\true\Phi.
\end{align*}
\text{}\vspace{-1.4cm}

\end{proof}
\noindent
The result of Theorem \ref{the: Expressing INEX[k]-formula with ESO[k]} can be formulated equivalently as follows: \\
All INEX[$k$]-definable properties of teams are ESO[$k$]-definable.


\subsection{Translation from ESO[$k$] to INEX[$k$]}\label{ssec: Expressing ESO}


When translating from ESO to INEX, our technique is to simulate second order quantification by replacing the quantifications of $k$-ary relation variables $P$ simply with quantifications of $k$-tuples of \emph{first order variables} $\tuple w$. The idea is then to choose such values for $\tuple w$ that in the resulting team $X$, the relation $X(\tuple w)$ is the same as the relation that is quantified for the value of $P$. 
However, we cannot simulate the quantification of the empty set this way, since the first order variables must be given at least one value. But this problem can be avoided, since any ESO-formula $\Phi$ can be written in an equivalent form $\Phi'$ which is satisfied if and only if it is satisfied with nonempty interpretations for the quantified relation variables.
This is shown in the following easy lemma.

\begin{lemma}\label{the: Non-emptiness lemma}
Let $\Phi:=\Ee P_1\dots\Ee P_n\,\gamma$ be an $\ESOset[k]$-formula, where $\gamma$ is the first order part of $\Phi$. Then there exists $\delta\in\ESOset[0]$ with the same free relation variables as $\gamma$ such that the following holds:
\[
	\mathcal{M}\true \Phi\;\text{ iff \,there exist nonempty }
	A_1,\dots,A_n\subseteq M^k \text{ s.t. } \mathcal{M}[\tuple A/\tuple P]\true\delta.
\]
\end{lemma}

\begin{proof}
We prove the claim by induction on $n$: If $n=0$, then $\Phi=\gamma$ and we can trivially choose $\delta:=\gamma$. Suppose then that the claim holds for $n-1$, i.e. there exists $\xi\in\ESOset[0]$ with the same free relation variables as $\psi$ such that
\begin{align*}
	\mathcal{M}\true\Ee P_1\dots\Ee P_{n-1}\gamma \;\text{ iff \,there exist nonempty } A_1,\dots,A_{n-1}\subseteq M^k \\
	\text{ s.t. } \mathcal{M}[A_1/P_1,\dots,A_{n-1}/P_{n-1}]\true\xi,
\end{align*}
for all models $\mathcal{M}$ that have some interpretation for all the free relation variables in the formula $\psi$.
Let $\psi\in\subf(\xi)$. We define $\psi'$ recursively as 
\begin{align*}
	\psi' &= \psi\; \text{ if $\psi$ is a literal and does not contain $P_n$} \\
	(P_n\,\tuple{\vphantom\wedge\smash{t}}\,)' 
	&= (\tuple{\vphantom\wedge\smash{t}}\neq\tuple{\vphantom\wedge\smash{t}}\,), \\
	(\neg P_n\,\tuple{\vphantom\wedge\smash{t}}\,)' 
	&= (\tuple{\vphantom\wedge\smash{t}}=\tuple{\vphantom\wedge\smash{t}}\,) \\[0,1cm]
	(\psi\wedge\theta)' &= \psi'\!\wedge\theta', \;
	(\psi\vee\theta)'= \psi'\vee\theta', \\
	(\Ee x\,\psi)' &= \Ee x\,\psi', \;
	(\Ae x\,\psi)' = \Ae x\,\psi'.
\end{align*}
Now clearly the formula $\xi'$ is satisfied in a model $\mathcal{M}$ if and only if $\xi$ is satisfied in the model $\mathcal{M}[\,\emptyset/P_n]$. Thus we can define $\delta:=\xi\vee\xi'$, whence it is easy to see that the claim holds by the inductive hypothesis.
\end{proof}

Now we are ready to formulate our translation from ESO[$k$] to INEX[$k$]. For this translation we must require the given teams to be nonempty and assume that the free relation variables in $\ESOset[k]$-formulas are at most $k$-ary.
We also assume the first order part of $\ESOset$-formula to be in negation normal form.
But here we can allow the $\ESOset[k]$-formula $\Phi$ to have any number of free relation variables instead of just one. Suppose that $\Phi$ defines some properties $\text{p}_1,\dots,\text{p}_m$ for relations $R_1,\dots,R_m$ respectively. Then it is natural to say that $\varphi(\tuple y_1\dots \tuple y_m)\in\INEXset$ is equivalent with $\Phi$ if the relations $X(\tuple y_1),\dots,X(\tuple y_m)$ have the properties $\text{p}_1,\dots,\text{p}_m$ in all teams $X$ in which $\varphi$ true. 

\begin{theorem}\label{the: Expressing ESO[k]-formula with INEX[k]}
Let $\Phi(R_1\dots R_m)\in\ESOset[k]$, where the free relation variables $R_i$ are at most $k$-ary. Let $\tuple y,\dots,\tuple y_m$ be $k$-tuples of fresh variables. Then there exists an $\INEXset[k]$-formula $\varphi(\tuple y_1\dots \tuple y_m)$, such that
\[
	\mathcal{M}\true_X\varphi\; \text{ iff }\, \mathcal{M}[X(\tuple y_1)/R_1,\dots,X(\tuple y_m)/R_m]\true\Phi,
\]
for all suitable $L$-models $\mathcal{M}$ and nonempty teams $X$.
\end{theorem}

\begin{proof}
Since $\Phi\in\ESOset[k]$, it is of the form $\Phi=\Ee P_1\dots\Ee P_n\,\gamma$, where $P_1,\dots,P_n$ are relation variables and $\gamma$ is the first order part of $\Phi$. Without loss of generality, we may assume that $P_1,\dots,P_n,R_1,\dots,R_m$ are all distinct and $k$-ary. Let $\delta$ be the formula given by Lemma \ref{the: Non-emptiness lemma} for the formula $\gamma$. Now we have
\begin{align*}
	\mathcal{M}\true \Phi\;\,\text{ iff \;there exist nonempty }
	A_1,\dots,A_n\subseteq M^k \text{ s.t. } \mathcal{M}[\tuple A/\tuple P]\true \delta, \tag{$\triangle$}
\end{align*}
for all models $\mathcal{M}$ that have interpretations for the relation variables $R_1,\dots,R_m$.
Let $\tuple w_1,\dots,\tuple w_n$ be $k$-tuples of fresh variables. The formula $\psi'$ is defined recursively for each $\psi\in\subf(\delta)$:
\begin{align*}
	\psi' &= \psi\; \text{ if $\psi$ is a literal and neither $P_i$ nor $R_j$ } \\
	&\hspace{3cm}\text{ occurs in } \psi \text{ for any } i \text{ or } j. \\[0,1cm]
	(P_i\,\tuple{\vphantom\wedge\smash{t}}\,)' 
	&= \tuple{\vphantom\wedge\smash{t}}\subseteq\tuple w_i, \hspace{9pt}
	(\neg P_i\,\tuple{\vphantom\wedge\smash{t}}\,)' 
	= \tuple{\vphantom\wedge\smash{t}}\,\mid\tuple w_i \hspace{0,45cm}\text{ for all } i\leq n \\
	(R_i\,\tuple{\vphantom\wedge\smash{t}}\,)' 
	&= \tuple{\vphantom\wedge\smash{t}}\subseteq\tuple y_i, \quad
	(\neg R_i\,\tuple{\vphantom\wedge\smash{t}}\,)' 
	= \tuple{\vphantom\wedge\smash{t}}\,\mid\tuple y_i \hspace{0,5cm}\text{ for all } i\leq m \\[0,1cm]
	(\psi\wedge\theta)' &= \psi'\!\wedge\theta' \\
	(\psi\vee\theta)' &= \psi'\!\veebar\theta', \;\;\text{ where }\;
	\veebar \,:=\,  \underset{\scriptscriptstyle\tuple w_1,\dots,\tuple w_n,\tuple y_1,\dots,\tuple y_m}{\vee} \\
	(\Es x\,\psi)' &= \Es x\,\psi', \quad
	(\Ae x\,\psi)' = \Ae x\,\psi'.
\end{align*}
Now we can define the formula $\varphi$ simply as:
\begin{align*}
	\varphi := \Ee\tuple w_1 \dots\Ee\tuple w_n\delta'.
\end{align*}
Clearly $\varphi$ is an $\INEXset[k]$-formula and $\fr(\varphi)=\vr(\tuple y_1\dots\tuple y_m)$\footnote{Also, note that if $\Phi$ is an $\ESOset$-sentence, then $\varphi$ is an $\INEXset$-sentence.}. Before proving the claim of this theorem need to prove Claims \ref{R1} and \ref{R2}.

\begin{claim}\label{R1}
Let $\mu\in\subf(\delta)$ and let $X$ be a team such that the variables $\tuple w_1,\dots,\tuple w_n$, $\tuple y_1,\dots,\tuple y_m$ are in $\dom(X)$. Let
\begin{align*}
	\mathcal{M}' := \mathcal{M}[&X(\tuple w_1)/P_1,\dots,X(\tuple w_n)/P_n,\,X(\tuple y_1)/R_1,\dots,X(\tuple y_m)/R_m].
\end{align*}
Now we have: \qquad $\text{If } \mathcal{M}\true_X\mu', \text{ then } \mathcal{M}'\true_X\mu.$
\end{claim}

\smallskip

\noindent
We prove this claim by structural induction on $\mu$:
\begin{itemize}[leftmargin=*]
\item If $\mu$ is a literal such that neither $P_i$ nor $R_j$ occurs in $\mu$ for any $i\leq n$ or $j\leq m$, the claim holds trivially since $\mu' = \mu$.

\item Let $\mu=P_j\tuple{\vphantom\wedge\smash{t}}$\, for some $j$ \;(the case $\mu=R_j\tuple{\vphantom\wedge\smash{t}}$ is analogous).
Suppose that $\mathcal{M}\true_{X}(P_j\tuple{\vphantom\wedge\smash{t}}\,)'$, i.e. $\mathcal{M}\true_{X}\tuple{\vphantom\wedge\smash{t}}\subseteq\tuple w_j$, and let $s\in X$. Because $\mathcal{M}\true_{X}\tuple{\vphantom\wedge\smash{t}}\subseteq\tuple w_j$, there exists $s'\in X$ such that $s'(\tuple w_j)=s(\tuple{\vphantom\wedge\smash{t}}\,)$. Now we have $s(\tuple{\vphantom\wedge\smash{t}}\,) \in X(\tuple w_j) = P_j^\mathcal{M'}$, and thus $\mathcal{M}'\true_X P_j \tuple{\vphantom\wedge\smash{t}}$.

\item Let $\mu=\neg P_j\tuple{\vphantom\wedge\smash{t}}$\, for some $j$ \;(the case $\mu=\neg R_j\tuple{\vphantom\wedge\smash{t}}$ is analogous).
Suppose that $\mathcal{M}\true_{X}(\neg P_j\tuple{\vphantom\wedge\smash{t}}\,)'$, i.e. $\mathcal{M}\true_{X}\tuple{\vphantom\wedge\smash{t}}\mid\tuple w_j$ and let $s\in X$. Since $\mathcal{M}\true_{X}\tuple{\vphantom\wedge\smash{t}}\mid\!\tuple w_j$, we have $s(\tuple{\vphantom\wedge\smash{t}}\,)\neq s'(\tuple w_j)$ for each $s'\in X$. Therefore $s(\tuple{\vphantom\wedge\smash{t}}\,)\notin X(\tuple w_j) =  P_j^\mathcal{M'}$, and thus $\mathcal{M}'\true_X \neg P_j\tuple{\vphantom\wedge\smash{t}}$.

\item The case $\mu=\psi\wedge\theta$ is straightforward to prove.

\item Let $\mu=\psi\vee\theta$.
Suppose that $\mathcal{M}\true_X(\psi\vee\theta)'$, i.e. $\mathcal{M}\true_{X}\psi'\veebar\,\theta'$. Thus there are $Y_1,Y_2\subseteq X$ s.t. $Y_1\cup Y_2=X$, $\mathcal{M}\true_{Y_1}\psi'$ and $\mathcal{M}\true_{Y_2}\theta'$, and if $Y_1,Y_2\neq\emptyset$, then the tuples $\tuple w_i$ and $\tuple y_j$ have the same set of values in $Y_1$ and $Y_2$ as they have in $X$ (for each $i\leq n$ and $j\leq m$).

If $Y_1=\emptyset$, then $Y_2=X$ and thus $\mathcal{M}\true_X\theta'$. By the inductive hypothesis $\mathcal{M}'\true_{X}\theta$ and thus $\mathcal{M}'\true_{X}\psi\vee\theta$. Analogously if $Y_2=\emptyset$, then $\mathcal{M}'\true_{X}\psi\vee\theta$.
Suppose then that $Y_1,Y_2\neq\emptyset$. Now by the inductive hypothesis we have
\[
	\begin{cases}
		\mathcal{M}[Y_1(\tuple w_i)_{i\leq n}/\tuple P,\,Y_1(\tuple y_j)_{j\leq m}/\tuple R\,]\true_{Y_1}\psi \\
		\mathcal{M}[Y_2(\tuple w_i)_{i\leq n}/\tuple P,\,Y_2(\tuple y_j)_{j\leq m}/\tuple R\,]\true_{Y_2}\theta.
	\end{cases}
\]
Because $\tuple w_i$ and $\tuple y_j$ have the same set of values in $Y_1$ and $Y_2$ as in $X$ (for any $i\leq n$, $j\leq m$), we have $\mathcal{M}'\true_{Y_1}\psi$ and $\mathcal{M}'\true_{Y_2}\theta$. Therefore $\mathcal{M}'\true_{X}\psi\vee\theta$.

\item The cases $\mu=\Es x\,\psi$ and $\mu=\Ae x\,\psi$ are straightforward to prove. 
\end{itemize}

\smallskip

\begin{claim}\label{R2}
Let $\mu\in\subf(\delta)$ and assume that $A_1,\dots,A_n,$ $B_1,\dots,B_m\subseteq M^k$ are nonempty sets. Let $X\neq\emptyset$ be a team such that $\vr(\tuple y_1\dots\tuple y_m)\subseteq\dom(X)$ and for each $i\leq m$ and $r\in X\!\upharpoonright\!\fr(\mu)$ the following assumption holds:
\[
	X_r(\tuple y_i) = B_i, \text{ where } X_r:= \{s\in X\,\mid \;s\upharpoonright\fr(\mu)=r\}. \tag{$\star$}
\]
This condition can be written equivalently as: For each $r\!\in\!X\!\upharpoonright\!\fr(\mu)$, $i\leq m$ and $\tuple{\vphantom t\smash{b}}\!\in\!B_i$ there exists $s\in X$ such that $s\!\upharpoonright\!(\fr(\mu)\!\cup\!\vr(\tuple y_i))=r[\tuple{\vphantom t\smash{b}}/\tuple y_i]$. 
That is, each assignment in $X\!\upharpoonright\!\fr(\varphi)$ can be extended to $X$ with all of the values in $B_i$.

\medskip

\noindent
Now the following implication holds:
\[
	\text{If } \mathcal{M}'\true_{X\upharpoonright\fr(\mu)}\mu, \text{ then } \mathcal{M}\true_{X'}\mu',
\]
where $\mathcal{M}' := \mathcal{M}[\tuple A/\tuple P,\tuple B/\tuple R\,]$ and $X' := X[A_1/\tuple w_1,\dots,A_n/\tuple w_n]$.
\end{claim}

\smallskip

\noindent
We prove this claim by structural induction on $\mu$:
\begin{itemize}[leftmargin=*]
\item If $\mu$ is a literal such that neither $P_i$ nor $R_j$ occurs in $\mu$, then the claim holds by locality since $\mu' = \mu$.

\item Let $\mu=P_j\tuple{\vphantom\wedge\smash{t}}$ for some $j\leq n$.
Suppose that $\mathcal{M}'\true_{X\upharpoonright\fr(\mu)} P_j\tuple{\vphantom\wedge\smash{t}}$. Let $s\in X'$ and let $r\in~\!\!X\!\upharpoonright\!\fr(\mu)$ be an assignment for which $r=s\!\upharpoonright\!\fr(\mu)$. Since $\mathcal{M}'\true_{X\upharpoonright\fr(\mu)} P_j\tuple{\vphantom\wedge\smash{t}}$, we have $r(\tuple{\vphantom\wedge\smash{t}}\,)\in P_j^\mathcal{M'}\!\!=\!A_j\!=\!X'(\tuple w_j)$. Thus there exists $s'\in X'$ s.t. $s'(\tuple w_j)=r(\tuple{\vphantom\wedge\smash{t}}\,)$. Now $s(\tuple{\vphantom\wedge\smash{t}}\,) = r(\tuple{\vphantom\wedge\smash{t}}\,) = s'(\tuple w_j)$. Therefore $\mathcal{M}\true_{X'}\tuple{\vphantom\wedge\smash{t}}\subseteq\tuple w_j$, i.e. $\mathcal{M}\true_{X'}(P_j\tuple{\vphantom\wedge\smash{t}}\,)'$.

\item Let $\mu=\neg P_j\tuple{\vphantom\wedge\smash{t}}$ for some $j\leq n$.
Suppose that $\mathcal{M}'\true_{X\upharpoonright\fr(\mu)} \neg P_j\tuple{\vphantom\wedge\smash{t}}$. Let $s,s'\in X'$ and let $r\in X\!\upharpoonright\!\fr(\mu)$ be an assignment s.t. $r=s\!\upharpoonright\!\fr(\mu)$. Because $\mathcal{M}'\true_{X\upharpoonright\fr(\mu)}\neg P_j\tuple{\vphantom\wedge\smash{t}}$, we have $r(\tuple{\vphantom\wedge\smash{t}}\,)\notin P_j^\mathcal{M'}\!\!=\!A_j\!=\!X'(\tuple w_j)$. Hence it has to be that $r(\tuple{\vphantom\wedge\smash{t}}\,)\neq s'(\tuple w_j)$, and thus $s(\tuple{\vphantom\wedge\smash{t}}\,) = r(\tuple{\vphantom\wedge\smash{t}}\,) \neq s'(\tuple w_j)$. Therefore $\mathcal{M}\true_{X'}\tuple{\vphantom\wedge\smash{t}}\mid\tuple w_j$, i.e. $\mathcal{M}\true_{X'}(\neg P_j\tuple{\vphantom\wedge\smash{t}}\,)'$.

\item Let $\mu=R_j\tuple{\vphantom\wedge\smash{t}}$ or $\mu=\neg R_j\tuple{\vphantom\wedge\smash{t}}$ for some $j\leq m$.
Note that because the condition ($\star$) holds for $X$ (with respect to $\fr(\mu)$), we have $X(\tuple y_j)=B_j$. Hence $R_j^\mathcal{M'}\!\!=\!B_j\!=\!X(\tuple y_j)\!=\!X'(\tuple y_j)$, and thus the cases $\mu=R_j\tuple{\vphantom\wedge\smash{t}}$ and $\mu=\neg R_j\tuple{\vphantom\wedge\smash{t}}$ can be proved analogously as we proved the two previous cases.

\item The case $\mu=\psi\wedge\theta$ is straightforward to prove.

\item Let $\mu=\psi\vee\theta$.
Suppose that $\mathcal{M}'\true_{X\upharpoonright\fr(\mu)}\psi\vee\theta$, i.e. there are $Y_1^*,Y_2^*\subseteq X\!\upharpoonright\!\fr(\mu)$ such that $Y_1^*\cup Y_2^*=X\!\upharpoonright\!\fr(\mu)$, $\mathcal{M}'\true_{Y_1^*}\psi$ and $\mathcal{M}'\true_{Y_2^*}\theta$. Let
\[
	Y_1 := \{s\in X\mid s\upharpoonright\fr(\mu)\in Y_1^*\} \quad\text{and}\quad
	Y_2 := \{s\in X\mid s\upharpoonright\fr(\mu)\in Y_2^*\}.
\]
Now $Y_1\!\upharpoonright\!\fr(\mu) = Y_1^*$, $Y_2\!\upharpoonright\!\fr(\mu) = Y_2^*$ and $Y_1\cup Y_2=X$. 
Let
\[
	Y_1' := Y_1[A_1/\tuple w_1,\dots,A_n/\tuple w_n] \quad\text{and}\quad
	Y_2' := Y_2[A_1/\tuple w_1,\dots,A_n/\tuple w_n].
\]
Now $X'\!=\!Y_1'\cup Y_2'$. If $Y_1'\!=\!\emptyset$, then $Y_2'\!=\!X'$ and thus clearly $\mathcal{M}\true_{X'}\psi'\veebar\theta'$, i.e $\mathcal{M}\true_{X'}(\psi\vee\theta)'$. Analogously if $Y_2'\!=\!\emptyset$, then $\mathcal{M}\true_{X'}(\psi\vee\theta)'$.
Suppose then that $Y_1',Y_2'\neq\emptyset$. Since the condition ($\star$) holds for $X$ with respect to $\fr(\mu)$, by the definition of $Y_1$ it is easy to see that ($\star$) holds also for $Y_1$ with respect to $\fr(\mu)$. Since $\fr(\psi)\subseteq\fr(\mu)$, $(\star)$ holds for $Y_1$ also with respect to $\fr(\psi)$. Analogously $(\star)$ holds for $Y_2$ with respect to $\fr(\theta)$.
Therefore, by the inductive hypothesis, $\mathcal{M}\true_{Y_1'}\psi'$ and $\mathcal{M}\true_{Y_2'}\theta'$. 
We also have $Y_1'(\tuple w_i)\!=\!Y_2'(\tuple w_i)\!=\!A_i\!=\!X'(w_i)$ for each $i\leq n$. Furthermore, by the condition ($\star$),  $Y_1'(\tuple y_i)\!=\!Y_2'(\tuple y_i)\!=\!B_i\!=\!X'(y_i)$ for each $i\leq m$. Therefore $\mathcal{M}\true_{X'}\psi'\veebar\,\theta'$, i.e $\mathcal{M}\true_{X'}(\psi\vee\theta)'$.

\item Let $\mu=\Es x\,\psi$ \;(the case $\mu=\Ae x\,\psi$ can be proven similarly).
Suppose that $\mathcal{M}'\true_{X\upharpoonright\fr(\mu)}\Es x\,\psi$. Hence there is $F:X\!\upharpoonright\!\fr(\mu)\rightarrow\mathcal{P}^*(M)$ such that $\mathcal{M}'\true_{(X\upharpoonright\fr(\mu))[F/x]}\psi$. Let
\begin{align*}
	G:&\;X\rightarrow\mathcal{P}^*(M), \;\; s\mapsto F(s\!\upharpoonright\!\fr(\mu)) \\
	G':&\;X'\rightarrow\mathcal{P}^*(M), \;\; s\mapsto F(s\!\upharpoonright\!\fr(\mu)).
\end{align*}
Now $X[G/x]\!\upharpoonright\!\fr(\psi)=(X\!\upharpoonright\!\fr(\mu))[F/x]$ and therefore $\mathcal{M}'\true_{X[G/x]\upharpoonright\fr(\psi)}\psi$. Since ($\star$) holds for $X$ with respect to $\fr(\mu)$, by the definition of $G$ ($\star$) holds for $X[G/x]$ with respect to $\fr(\psi)$. Let $X'' := (X[G/x])[A_1/\tuple w_1,\dots,A_n/\tuple w_n]$, whence by the inductive hypothesis we have $\mathcal{M}\true_{X''}\psi'$. By the definition of $G'$, we have $X'' = X'[G'/x]$, and thus $\mathcal{M}\true_{X'[G'/x]}\psi'$. Hence $\mathcal{M}\true_{X'}\Es x\,\psi'$, i.e. $\mathcal{M}\true_{X'}(\Es x\,\psi)'$. 
\end{itemize}

\bigskip

\noindent
We are now are finally ready prove the claim of this theorem:
\[
	\mathcal{M}\true_X\varphi\; \text{ iff }\, \mathcal{M}[X(\tuple y_1)/R_1,\dots,X(\tuple y_m)/R_m]\true\Phi.
\]


\noindent
Suppose first that $\mathcal{M}\true_X\varphi$, i.e. $\mathcal{M}\true_X\Es\tuple w_1\dots\Es\tuple w_n\delta'$. Thus there exist
\begin{align*}
	&\mathcal{F}_1:X\rightarrow\mathcal{P}^*(M^k) \\[-0,1cm]
	&\mathcal{F}_2:X[\mathcal{F}_1/\tuple w_1]\rightarrow\mathcal{P}^*(M^k) \\[-0,3cm]
	&\;\vdots \\[-0,15cm]
	&\mathcal{F}_n:X[\mathcal{F}_1/\tuple w_1,\dots,
	\mathcal{F}_{n-1}/\tuple w_{n-1}]\rightarrow\mathcal{P}^*(M^k) \\
	&\qquad\text{s.t. } \mathcal{M}\true_{X'}\delta', \text{ where } 
	X':=X[\mathcal{F}_1/\tuple w_1,\dots,\mathcal{F}_n/\tuple w_n].
\end{align*}
Let $\mathcal{M}' := \mathcal{M}[X'(\tuple w_1)/P_1,\dots,X'(\tuple w_n)/P_n,\,X'(\tuple y_1)/R_1,\dots,X'(\tuple y_m)/R_m]$.
Now by Claim \ref{R1}, we have $\mathcal{M}'\true_{X'}\delta$. Because $X'\!\upharpoonright\!\fr(\delta)=\{\emptyset\}$, by locality $\mathcal{M}'\true\delta$. Since $X'(\tuple y_i)\!=\!X(\tuple y_i)$ for each $i\leq m$, we have $\mathcal{M}[X(\tuple y_1)/R_1,\dots,X(\tuple y_m)/R_m]\true\Phi$.

\medskip

\noindent
Suppose then that $\mathcal{M}[X(\tuple y_1)/R_1,\dots,X(\tuple y_m)/R_m]\true\Phi$. Thus, by the equation ($\triangle$), there are nonempty sets $A_1,\dots,A_n\subseteq M^k$ such that 
\[
	\mathcal{M}'\true\delta, \text{ where } 
	\mathcal{M}' := \mathcal{M}[A_1/P_1,\dots,A_n/P_n,X(\tuple y_1)/R_1,\dots,X(\tuple y_m)/R_m].
\]
Since, by the assumptions, $X\neq\emptyset$ and $\vr(\tuple y_1\dots\tuple y_m)\subseteq\dom(X)$, we have $X(\tuple y_i)\neq\emptyset$ for each $i\leq m$. We define the function $\mathcal{F}_i$ for each $i\leq n$ by
\[
	\mathcal{F}_i:X[\mathcal{F}_1/\tuple w_1,
	\dots,\mathcal{F}_{i-1}/\tuple w_{i-1}]\rightarrow\mathcal{P}^*(M^k), \;\; s\mapsto A_i.
\]
Let $X' := X[\mathcal{F}_1/\tuple w_1,\dots,\mathcal{F}_n/\tuple w_n]$, whence $X' = X[A_1/\tuple w_1,\dots,A_n/\tuple w_n]$. Since $X\upharpoonright\fr(\delta)=\{\emptyset\}$, the condition ($\star$) in Claim \ref{R2}  holds for the team $X$ with respect to $\fr(\delta)$. We also have $\mathcal{M}'\true_{X\upharpoonright\fr(\delta)}\delta$ and thus by Claim~\ref{R2} we obtain $\mathcal{M}\true_{X'}\delta'$. Hence $\mathcal{M}\true_{X}\Es\tuple w_1\dots\Es\tuple w_n\delta'$, i.e. $\mathcal{M}\true_X\varphi$. 
\end{proof}

\begin{remark}
By Theorem \ref{the: Expressing ESO[k]-formula with INEX[k]}, for each $\ESOset[k]$-formula $\Phi(R)$, for which $R$ is at most $k$-ary, there exists an $\INEXset[k]$-formula $\varphi(\tuple y\,)$ such that for all $X\neq\emptyset$:
\[
	\mathcal{M}\true_X\varphi\; \text{ iff }\, \mathcal{M}[X(\tuple y\,)/R]\true\Phi.
\]
Without the requirement of non-empty teams and the arity restriction on $R$, this would be the converse of Theorem \ref{the: Expressing INEX[k]-formula with ESO[k]}. But due empty team property of INEX, the left side of the equivalence is always true for the empty team and any formula of INEX. Thus, when defining classes of relations with INEX, we can only define such classes that include the empty relation. 
The arity restriction is also necessary since it can be shown that for any $k$ there are ESO[$k$]-definable properties of ($k$$+$$1$)-ary relations $X(\tuple y\,)$ that cannot be defined in INEX[$k$]. A~proof for this claim will be presented in a future work by the author.
\end{remark}

Since Theorem \ref{the: Expressing INEX[k]-formula with ESO[k]} and Theorem \ref{the: Expressing ESO[k]-formula with INEX[k]} can also be proven for $\INEXset[k]$- and $\ESOset[k]$-sentences, we obtain the following corollary.

\begin{corollary}\label{the: INEX[$k$] captures ESO[$k$] on the level of sentences}
On the level of sentences \emph{INEX[$k$]} captures the expressive power of \emph{ESO[$k$]}. In particular, $\INEX[1]$ captures $\EMSO$.
\end{corollary}

As a direct corollary we also obtain a strict arity hierarchy for INEX, since the arity hierarchy for ESO (with arbitrary vocabulary) is strict, as shown by Ajtai \cite{Ajtai83} in 1983. As mentioned in the introduction, $k$-ary dependence and independence logics capture the fragment of ESO where at most ($k$$-$$1$)-ary functions can be quantified. This fragment differs from ESO[$k$] at least when $k$ is one or two -- and presumably for any~$k$. Hence it appears that INEX[$k$] does not correspond to $l$-ary independence logic for any $k$ and $l$, even though without arity bounds these two logics are equivalent.

\subsection*{On the duality of inclusion and exclusion atoms}


For the last topic in this section, we will discuss the relationship of inclusion and exclusion atoms. We will also consider natural candidates for the semantics of negated inclusion and exclusion atoms.
In our translation in Theorem~\ref{the: Expressing ESO[k]-formula with INEX[k]} we used inclusion and exclusion atoms in a dualistic way by replacing atomic formulas $P\,\tuple{\vphantom\wedge\smash{t}}$ with inclusion atoms and negated atomic formulas $\neg P\,\tuple{\vphantom\wedge\smash{t}}$ with exclusion atoms. This correspondence becomes more obvious when we reformulate the truth conditions for $P\,\tuple{\vphantom\wedge\smash{t}}$ and $\neg P\,\tuple{\vphantom\wedge\smash{t}}$ (compare with Definition~\ref{def: Team semantics}) as follows:
\[
	\hspace{0,45cm}\mathcal{M}\true_X P\,\tuple{\vphantom\wedge\smash{t}} 
	\;\text{ iff }\; X(\tuple{\vphantom\wedge\smash{t}}\,) \subseteq P^{\mathcal{M}}
	\qquad\text{and}\qquad
	\mathcal{M}\true_X \neg P\,\tuple{\vphantom\wedge\smash{t}} 
	\;\text{ iff }\; X(\tuple{\vphantom\wedge\smash{t}}\,) \subseteq \overline{P^{\mathcal{M}}}.
\]
The truth conditions for inclusion and exclusion atoms can be written in a form that is very similar to the equivalences above:
\[
	\mathcal{M}\true_X \tuple{\vphantom\wedge\smash{t}}_1\inc\tuple{\vphantom\wedge\smash{t}}_2 
	\;\text{ iff }\; X(\tuple{\vphantom\wedge\smash{t}}_1) \subseteq X(\tuple{\vphantom\wedge\smash{t}}_2)
	\qquad\text{and}\qquad
	\mathcal{M}\true_X \tuple{\vphantom\wedge\smash{t}}_1\exc\tuple{\vphantom\wedge\smash{t}}_2 
	\;\text{ iff }\; X(\tuple{\vphantom\wedge\smash{t}}_1) \subseteq \overline{X(\tuple{\vphantom\wedge\smash{t}}_2)}.
\]
%
%
As we argued earlier (Observation~\ref{obs: contradictory negation}), the semantics of a contradictory negation ($\mathcal{M}\true_X\neg\varphi$ iff $\mathcal{M}\ntrue_X\varphi$) is not a very natural choice of semantics for the negated atoms. Instead, it would be more natural to have such a semantics that is similar to the semantics of literals.  
From this viewpoint, a natural candidate for a semantics of a negated inclusion atom would be the following: 
\[
	\mathcal{M}\true_X\neg(\tuple{\vphantom\wedge\smash{t}}_1\inc\tuple{\vphantom\wedge\smash{t}}_2) 
	\;\text{ iff }\; X(\tuple{\vphantom\wedge\smash{t}}_1) \subseteq \overline{X(\tuple{\vphantom\wedge\smash{t}}_2)}. 
	\tag{$\neg\subseteq$}
\]
Then we would have $\neg(\tuple{\vphantom\wedge\smash{t}}_1\inc\tuple{\vphantom\wedge\smash{t}}_2) \equiv \tuple{\vphantom\wedge\smash{t}}_1\exc\tuple{\vphantom\wedge\smash{t}}_2$. Therefore, if we allow the use of negated atoms in $\INC[k]$ with our semantics, the resulting logic is equivalent with $\INEX[k]$. Note that since the exclusion relation is symmetric, our choice of semantics leads to the following equivalence:
\[
	\neg(\tuple{\vphantom\wedge\smash{t}}_1\inc\tuple{\vphantom\wedge\smash{t}}_2) 
	\;\equiv\; \tuple{\vphantom\wedge\smash{t}}_1\exc\tuple{\vphantom\wedge\smash{t}}_2 
	\;\equiv\; \tuple{\vphantom\wedge\smash{t}}_2\exc\tuple{\vphantom\wedge\smash{t}}_1 
	\;\equiv\; \neg(\tuple{\vphantom\wedge\smash{t}}_2\inc\tuple{\vphantom\wedge\smash{t}}_1).
\]
Hence, by this definition, $\neg(\tuple{\vphantom\wedge\smash{t}}_1\inc\tuple{\vphantom\wedge\smash{t}}_2) \equiv \neg(\tuple{\vphantom\wedge\smash{t}}_2\inc\tuple{\vphantom\wedge\smash{t}}_1)$ even though $\tuple{\vphantom\wedge\smash{t}}_1\inc\tuple{\vphantom\wedge\smash{t}}_2 \not\equiv \tuple{\vphantom\wedge\smash{t}}_2\inc\tuple{\vphantom\wedge\smash{t}}_1$. 
%
%
This kind of property of a negated atom might be a bit exotic, but not unthinkable, since our negation is not a contradictory negation.

Let us then consider semantics for the negated exclusion atom $\neg(\tuple{\vphantom\wedge\smash{t}}_1\exc\tuple{\vphantom\wedge\smash{t}}_2)$. Semantics of the inclusion atom $\tuple{\vphantom\wedge\smash{t}}_1\inc\tuple{\vphantom\wedge\smash{t}}_2$ is not a possible choice here, since by the symmetry of the exclusion relation, we must have $\neg(\tuple{\vphantom\wedge\smash{t}}_1\exc\tuple{\vphantom\wedge\smash{t}}_2) \equiv \neg(\tuple{\vphantom\wedge\smash{t}}_2\exc\tuple{\vphantom\wedge\smash{t}}_1)$. The truth condition $\mathcal{M}\true_X \tuple{\vphantom\wedge\smash{t}}_1\exc\tuple{\vphantom\wedge\smash{t}}_2$ iff $X(\tuple{\vphantom\wedge\smash{t}}_1) \cap X(\tuple{\vphantom\wedge\smash{t}}_2) = \emptyset$, naturally gives us the following candidate for a semantics.\footnote{Another possible candidate would be $\mathcal{M}\true_X\neg(\tuple{\vphantom\wedge\smash{t}}_1\exc\tuple{\vphantom\wedge\smash{t}}_2)$ iff $X(\tuple{\vphantom\wedge\smash{t}}_1) \subseteq X(\tuple{\vphantom\wedge\smash{t}}_2)$ or $X(\tuple{\vphantom\wedge\smash{t}}_1) \supseteq X(\tuple{\vphantom\wedge\smash{t}}_2)$, whence $\neg(\tuple{\vphantom\wedge\smash{t}}_1\exc\tuple{\vphantom\wedge\smash{t}}_2) \equiv \tuple{\vphantom\wedge\smash{t}}_1\inc\tuple{\vphantom\wedge\smash{t}}_2 \sqcup \tuple{\vphantom\wedge\smash{t}}_2\inc\tuple{\vphantom\wedge\smash{t}}_1$. We will not consider this choice here further.}
\[
	\mathcal{M}\true_X\neg(\tuple{\vphantom\wedge\smash{t}}_1\exc\tuple{\vphantom\wedge\smash{t}}_2) 
	\;\text{ iff }\; X(\tuple{\vphantom\wedge\smash{t}}_1) = X(\tuple{\vphantom\wedge\smash{t}}_2) 
	\tag{$\neg\mid\;$} 
\]
Now we have $\neg(\tuple{\vphantom\wedge\smash{t}}_1\exc\tuple{\vphantom\wedge\smash{t}}_2) \equiv \neg(\tuple{\vphantom\wedge\smash{t}}_2\exc\tuple{\vphantom\wedge\smash{t}}_1)$, as required, and $\neg(\tuple{\vphantom\wedge\smash{t}}_1\exc\tuple{\vphantom\wedge\smash{t}}_2) \equiv \tuple{\vphantom\wedge\smash{t}}_1\inc\tuple{\vphantom\wedge\smash{t}}_2 \wedge \tuple{\vphantom\wedge\smash{t}}_2\inc\tuple{\vphantom\wedge\smash{t}}_1$. 
This choice of semantics is actually equivalent with the semantics of \emph{equiextension atom} $\tuple{\vphantom\wedge\smash{t}}_1\bowtie\tuple{\vphantom\wedge\smash{t}}_2$ that was introduced by Galliani in \cite{Galliani12b}. This atom has been shown equivalent with the inclusion atom $\tuple{\vphantom\wedge\smash{t}}_1\inc\tuple{\vphantom\wedge\smash{t}}_2$ of the same arity (\cite{Galliani12b}). 
Hence if we allow the use of negated atoms in $\EXC[k]$ with our semantics, the resulting logic turns out be equivalent with $\INEX[k]$.


With our choices for semantics of negated inclusion and exclusion atoms, ($\neg\subseteq$) and ($\neg\mid\;$), we have
$\neg(\tuple{\vphantom\wedge\smash{t}}_1\inc\tuple{\vphantom\wedge\smash{t}}_2) \equiv \tuple{\vphantom\wedge\smash{t}}_1\mid\tuple{\vphantom\wedge\smash{t}}_2$ and $\neg(\tuple{\vphantom\wedge\smash{t}}_1\exc\tuple{\vphantom\wedge\smash{t}}_2) \equiv \tuple{\vphantom\wedge\smash{t}}_1\bowtie\tuple{\vphantom\wedge\smash{t}}_2$.
Now the exclusion atom is equivalent with the negated inclusion atom, but not vice versa. However, the negated exclusion atom is equally expressive as the inclusion atom of the corresponding arity. Hence even though inclusion and exclusion atoms are not exactly negations of each other, they nevertheless have a dualistic relationship. We could extend FO with either of these atoms and allow the use of its negation to obtain a logic equivalent to INEX.



In team semantics we must require all formulas to be in negation normal form. This can be seen as one of the weaknesses of this framework since the free use of negation is natural for a logic. Dependence logic and other related logics have also been criticized for not having sensible semantics for negated atoms\footnote{Originally, in \cite{Vaananen07}, negations were allowed to appear in front of dependence atoms. But the semantics for negated dependence atom was defined such that it was true only in the empty team. This way both empty team property and downwards closure were preserved.}.
However, there has not been much research on these issues. 

In order to solve these problems, Kuusisto \cite{Kuusisto15} has presented an alternative framework called \emph{double team semantics}. In this approach there are always two teams -- a ``verifying team'' and a ``falsifying team''. This allows to use negations freely, whence it just swaps the roles of these two teams. 
In \cite{Kuusisto15} Kuusisto has also presented a natural game-theoretic variant for this semantics, where negation essentially does role swapping of the verifying and the falsifying player.
This approach has received relatively little attention, but we believe that it should be studied further in order to understand the role of negation in team semantics more deeply.


\section{Examples of INEX-definable properties}\label{sec: Examples}


In this section we present several examples on the expressibility of inclusion-exclusion logic. Within these examples we also utilize several of the new operators we introduced in Section~\ref{sec: New operators}. Although all of the properties expressed here are known to be expressible in INEX by the results of the previous section, we believe that these examples are valuable for demonstrating the nature of inclusion-exclusion logic and team semantics in general.


By Corollary \ref{the: INEX[$k$] captures ESO[$k$] on the level of sentences} we know that, in particular, all EMSO-definable  properties of models can be expressed by using only unary inclusion and exclusion atoms. In the next example we show how two classical EMSO-definable properties of graphs can be defined in INEX[$1$].

\begin{example}\label{ex: Properties of graphs}
Let $\mathcal{G}=(V,E)$ be an undirected graph. Then we have
\begin{enumerate}[leftmargin=0.8cm]
\item[(a)] $\mathcal{G}$ is disconnected if and only if
\[
	\mathcal{G}\true\Es x_1\Es x_2\,\bigl(x_1\,|\, x_2\wedge\Ae z\,(z\inc x_1\vee z\inc x_2) 
	\wedge(\Ae y_1\inc x_1)(\Ae y_2\inc x_2)\neg Ey_1y_2\bigr).
\]
\item[(b)] $\mathcal{G}$ is $k$-colorable if and only if
\begin{align*}
	\mathcal{G}\true \gamma_{\leq k}\vee
	\Es x_1&\dots\Es x_k\,\Bigl(\bigwedge_{i\neq j}x_i\,|\, x_j\wedge\Ae z\,\bigl(\bigvee_{i\leq k} z\,\inc x_i\bigr) \\[-0,1cm]
	&\quad\qquad\qquad\wedge\bigwedge_{i\leq k}(\Ae y_1\inc x_i)(\Ae y_2\inc x_i)\neg Ey_1y_2\Bigr), \\[-0,8cm]
\end{align*}
where $\gamma_{\leq k}:= \Ee x_1\dots\Ee x_k\Ae y\,(\bigvee_{i\leq k}y\!=\!x_i)$.
\end{enumerate}

\smallskip

\noindent
We explain briefly why these equivalences hold: In (a) we first quantify two nonempty sets for the values of $x_1$ and $x_2$. We use exclusion atom to guarantee that these sets are disjoint. The formula $\Ae z\,(z\subseteq x_1\vee z\subseteq x_2)$ checks that the union of these sets covers the whole set of vertices (recall Example \ref{ex: Full relation}). Finally we use universal inclusion quantifiers to confirm that for any pair of elements chosen within these sets, there is no edge between them.

In (b) we first check if $\abs{V}\leq k$, in which case the graph would be trivially $k$-colorable. If that is not the case, we can quantify $k$ nonempty disjoint sets which represent the coloring of the graph. Confirming that these sets are disjoint and cover all the vertices can be done similarly as in (a). Finally we confirm that the coloring is correct by choosing any pair of vertices within an  unicolored set and checking that there is no edge between them.
\end{example}
The properties in Example~\ref{ex: Properties of graphs} could also be expressed in EMSO and then we could directly use our translation in Theorem~\ref{the: Expressing ESO[k]-formula with INEX[k]} to express these properties in INEX[$1$]. This method would give us sentences that are only slightly longer than the ones we have given above. However, the sentences above are not only shorter but also demonstrate the usefulness of universal inclusion quantifier.

Even though the arity fragments of INEX correspond to the arity fragments of ESO on the level of sentences, one should remember that these two logics have a different nature. Despite having the same expressive power, they provide us with alternative tools. Hence the study of inclusion-exclusion logic might even have potential for giving new insight on the arity fragments of ESO.

In the next example we demonstrate how we can use our translation in Theorem \ref{the: Expressing ESO[k]-formula with INEX[k]} to apply techniques of ESO directly to inclusion-exclusion logic. In ESO[$k$$+$$1$] we can quantify a $k$-ary function by quantifying a ($k$$+$$1$)-ary relation and giving requirements that it is a function. We can do this analogously in inclusion-exclusion logic; see the following example.
\begin{example}
Let $\varphi$ be an $\text{INEX}_{L\cup\{F\}}$-sentence where $F$ is a ($k$$+$$1$)-ary relation symbol. Let $\tuple x$ be a ($k$$+$$1$)-tuple of fresh variables. The formula $\Ee F\,\varphi$, where $F$ is quantified as a $k$-ary function, is equivalent with the $\INEXset$-sentence:
\begin{align*}
	\xi :=& \Ee\tuple x\,\bigl(\psi_1(\tuple x)\wedge\psi_2(\tuple x)\wedge\varphi'\bigr), \text{ where } \\
	&\begin{cases}
		\psi_1(\tuple x) := \Ae\tuple y\,\Ee z\,(\tuple yz\subseteq \tuple x)	\\
		\psi_2(\tuple x) := \Ae\tuple y\,\Ae z_1\Ae z_2\,\bigl((\tuple yz_1\,|\,\tuple x
		\,\underset{\scriptscriptstyle \tuple x}{\vee}
		\, \tuple yz_2\,|\, \tuple x)\,\underset{\scriptscriptstyle \tuple x}{\vee}\,z_1=z_2\bigr)
	\end{cases}
\end{align*}
and $\varphi'$ is a formula obtained from $\varphi$ by replacing all subformulas of the form $F \tuple{\vphantom\wedge\smash{t}}$ with inclusion atoms $\tuple{\vphantom\wedge\smash{t}}\subseteq\tuple x$, formulas $\neg F \tuple{\vphantom\wedge\smash{t}}$ with exclusion atoms $\tuple{\vphantom\wedge\smash{t}}\mid\tuple x$ and all disjunctions with the disjunctions that preserve the values of the tuple $\tuple x$. 

Note that $\psi_1$ and $\psi_2$ above are derived by changing the corresponding $\text{ESO}_{L\cup\{F\}}$-sentences $\Ae\tuple y\,\Ee z\,F\tuple yz$ and $\Ae\tuple y\,\Ae z_1\Ae z_2\,((F\tuple yz_1\wedge F\tuple yz_2)\rightarrow z_1\!=\!z_2)$ to negation normal form and then directly using our translation (Theorem \ref{the: Expressing ESO[k]-formula with INEX[k]}) from ESO to INEX.

If we also want the quantified function to injective or surjective, we can add either of the following formulas inside the brackets of the formula $\xi$ above:
\begin{align*}
	\psi_{\text{inj}}(\tuple x) &:= \Ae\tuple y_1\Ae\tuple y_2\Ae z\,((\tuple y_1z\,|\,\tuple x
	\,\underset{\scriptscriptstyle \tuple x}{\vee}\, \tuple y_2z\,|\, \tuple x)\,
	\underset{\scriptscriptstyle \tuple x}{\vee}\,\tuple y_1=\tuple y_2) \\
	\psi_{\text {surj}}(\tuple x) &:= \Ae z\Ee\tuple y\,(\tuple yz\inc \tuple x).
\end{align*}
In a similar way we can require any ESO-definable condition for the function that is quantified in the team as the values of the tuple $\tuple x$.
\end{example}

By using the method of the previous example, we can define infinity of a model in INEX[$2$] by simply saying that we can existentially quantify the values of variables $x_1$ and $x_2$ in such a way that in the resulting team $X$ the relation $X(x_1x_2)$ is a function that is injective but not surjective.

\begin{example}
Let $\delta_{\text{inf}}\in\INEXset[2]$ such that
\begin{align*}
	\delta_{\text{inf}} := \Ee x_1x_2\,\bigl(\psi_1(x_1x_2)&\wedge\psi_2(x_1x_2)\wedge\psi_{\text{inj}}(x_1x_2) 
	\wedge\Ee z\Ae y\,(yz\exc x_1x_2)\bigr),
\end{align*}
where the formulas $\psi_1$, $\psi_2$ and $\psi_{\text{inj}}$ are as in the previous example. 
Now a model $\mathcal{M}$ is infinite if and only if $\mathcal{M}\true\delta_{\text{inf}}$. Note that this property cannot be expressed by using only unary atoms since it is not EMSO-definable. 
\end{example}

The expressive power of INEX[$2$] is rather strong also on the level of formulas since, by Theorem~\ref{the: Expressing ESO[k]-formula with INEX[k]}, all ESO[$2$]-definable properties of $2$-ary relations in teams are definable in INEX[$2$]. In particular, we can say in INEX[$2$] that a certain variable $y$ gets infinitely many values within a team. This can be done simply by relativizing (recall Subsection~\ref{ssec: Relativization}) the sentence $\delta_{\text{inf}}$, of the previous example, on the variable $y$.
See the following example.

\begin{example}\label{ex: Infinity of a team}
Let $\delta_{\text{inf}}$ be as above and let $X$ be a nonempty team such that the variables in $\delta_{\text{inf}}$ are not in $\dom(X)$. Now for every $y\in\dom(X)$ we have
\[
	\mathcal{M}\true_X \delta_{\text{inf}}\!\upharpoonright\! y \; \text{ iff } \; X(y) \text{ is infinite}.
\]
If the team $X$ has a finite domain, we can now say that $X$ consists of infinitely many assignments. That is, if $\dom(X)=\{y_1,\dots,y_n\}$, then we have
\[
	\mathcal{M}\true_X \bigsqcup_{i\leq n}(\delta_{\text{inf}}\!\upharpoonright\! y_i) \; \text{ iff } \; X \text{ is infinite}.
\]
Note that the intuitionistic disjunction above can be used in $\INEX[2]$.
\end{example}

Because of locality property, there cannot be any single $\INEXset$-formula that would define the infinity of a team with arbitrary domain. For the same reason we cannot define this property in INEX for teams with infinite domain (even if it is fixed). To see this, consider a team $X$ for which $\dom(X)=\{x_i\mid i\in\mathbb{N}\}$ and $\{s(x_0x_1x_2\dots) \mid s\in X\} = \{0,1\}^\mathbb{N}$.
Now $X$ is infinite but $X\!\upharpoonright\!V$ is finite for every finite $V\subseteq\dom(X)$. But if $\varphi\in\INEXset$ such that $\fr(\varphi)\subseteq\dom(X)$, by locality $\varphi$ is true in $X$ if and only if it is true in $X\upharpoonright\fr(\varphi)$. 
Note that these restrictions hold for any logics with locality property -- such as dependence and independence logics. Therefore if our logic is local, we can only define infinity of teams that have a fixed finite domain.

Since infinity of a team is not downwards closed property, it cannot be expressed in dependence logic.\footnote{Infinity of a \emph{model}, however, can be expressed in dependence logic with a simple sentence such as $\Ee x\Ae y\Ee z\,(\dep(y,z)\wedge x\!\neq\!z))$ (\cite{Vaananen07}).}
By the results of Galliani \cite{Galliani12b} we know that it is expressible in independence logic. However, to our best knowledge, nobody has presented an explicit formula that would define this property in independence logic (or any other logic with team semantics). If we would express this property in independence logic by directly using the translations given by Galliani \cite{Galliani12b}, the corresponding formula would be very complicated. 


In Example~\ref{ex: Infinity of a team} we defined infinity of a team with a rather simple formula that was constructed in an intuitive way by using methods introduced in this paper. This was just one particular example, but we hope that this demonstrates how our work for this framework can be useful for deriving concrete formulas defining desired properties of models or teams. 


\section{Conclusion}\label{sec: Conclusion}

In this paper we have studied the expressive power of inclusion and exclusion atoms. These two simple types of atoms make a natural pair by a having a dualistic relationship. Our main topic of interest was how does the arity of these atoms affect their expressive power. We showed that INEX has a strict arity hierarchy, and when restricted to INEX[$k$], there is a natural connection to ESO[$k$] on the level of both sentences and formulas. 

When translating from ESO to INEX, atomic $k$-ary relations translate naturally into $k$-ary inclusion atoms and analogously negated atomic $k$-ary relations translate into $k$-ary exclusion atoms. This simple correspondence is somewhat surprising considering how different these two logics seem by first glance. It is also interesting how in team semantics we can use quantified $k$-tuples of variables to simulate quantified $k$-ary relation variables. That is, we can ``embed'' second order quantification within the standard first order quantification. 

Even though INEX is equivalent with independence logic in general, it turned out that the relationship is not so clear when restricting the arities of atoms. Despite being closely related, these logics are of different nature. It~appears that inclusion and exclusion atoms are naturally connected with relations while dependence and independence atoms are with functions. 

The translations we used between INEX[$k$] and ESO are very different from the ones used between $k$-ary dependence and independence logics and ESO on the level of sentences (\cite{Durand12,Kontinen13a}). The methods in our proofs also differ from Galliani's translations (\cite{Galliani12b}) between INEX-formulas and ESO-formulas (without any arity restriction). These earlier translations are not compositional in the sense that they work only for ESO-formulas in a special normal form. Our translations are all compositional and, particularly the ones in Theorems~\ref{the: Expressing EXC[k]-formula with ESO[k]} and \ref{the: Expressing ESO[k]-formula with INEX[k]}, very natural and do not increase the size of the formulas significantly.

In the translation from ESO[$k$] to INEX[$k$], term value preserving disjunction played an important role. This is a useful operator for any logic with team semantics, since the splitting of the team when evaluating disjunctions tends to lose information. This operator has several natural variants that would be interesting to be studied further either independently or in by adding them to some other logic with team semantics -- such as dependence logic.

We also introduced natural semantics for inclusion and exclusion quantifiers and defined them in INEX. With these quantifiers we can restrict the range of quantification to certain sets of values within a team. One practical application of this was the relativization of formulas. \emph{Existential} inclusion and exclusion quantifiers turned out to be equivalent with inclusion and exclusion atoms, and this naturally lead to the definition of inclusion and exclusion friendly logics. However, as discussed in Subsection~\ref{ssec: Properties of quantifiers}, properties of \emph{universal} inclusion and exclusion quantifiers are not so clear, and there still some open questions.

By our results on formulas, we know that all ESO[$k$]-definable properties of at most $k$-ary relations (in teams) can be defined in INEX[$k$], and that ESO[$k$] is the upper bound for the expressive power of $\INEX[k]$. But these limits are not strict and when the arity of relations gets higher than the arity of atoms, things get quite interesting.  In a future work we will pursue this topic further by showing, e.g, that for $2$-ary relations there are some very simple $\ESO[0]$-definable properties which are not INEX[$1$]-definable, but there are also some quite complex INEX[$1$]-definable properties which are not $\ESO[0]$-definable.


\bibliographystyle{abbrv} 
\bibliography{RR-citations}





\appendix

\section*{Appendix: Proof for claim \ref{AT3}}\label{sec: Appendix}

In this appendix we will present the proof for Claim \ref{AT3} that was used in the translation from INC[$k$] to ESO[$k$]. With the assumptions of Theorem \ref{the: Expressing INC[k]-formula with ESO[k]}, the following equivalence holds for all $\mu\in\subf(\varphi)$ and all suitable teams $X$:
\begin{align*}
	\mathcal{M}\true_X\mu\, \;\text{ iff }\;\,&\text{there exist }
	A_1,\dots,A_n\subseteq M^k \text{ s.t. } \mathcal{M}[\tuple A/\tuple P]\true_X\mu', \\
	&\text{and for any } i\leq n \text{ and tuple of elements } \tuple a\in A_i \\[-0,1cm]
	&\text{there exists } s\in X \text{ s.t. } \mathcal{M}[\tuple A/\tuple P]\true_{\{s[\tuple a/\tuple u\,]\}}\mu_i'.
\end{align*}
We first need to present two more claims. The first one is about the trivial cases when the inclusion atom $(\tuple{\vphantom\wedge\smash{t}}_1\inc\tuple{\vphantom\wedge\smash{t}}_2)_i$ does not occur in a formula $\mu\in\subf(\varphi)$.

\begin{customclaim}{I}\label{AT1}
Let $\mu\in\subf(\varphi)$ and assume that $ i\leq n$ is an index such that the atom $(\tuple{\vphantom\wedge\smash{t}}_1\subseteq\tuple{\vphantom\wedge\smash{t}}_2)_i$ does not occur in $\mu$. Then the following equivalences hold:
\begin{align*}
	\mathcal{M}\true_X\mu'\;\;\; \text{ iff }\;\;\; \mathcal{M}\true_X\mu_i'\;\;\;
	\text{ iff }\;\;\; \mathcal{M}\true_{\{s[\tuple a/\tuple u\,]\}}\mu_i' 
	\,\text{ for all } \tuple a\in M^k \text{ and } s\in X.
\end{align*}
\end{customclaim}
\begin{proof}
The first equivalence can be proved by a simple induction on $\mu$. Since $(\tuple{\vphantom\wedge\smash{t}}_1\inc\tuple{\vphantom\wedge\smash{t}}_2)_i$ does not occur in $\mu$, the definitions of $\mu'$ and $\mu_i'$ differ only when $\mu=\Ae x\,\psi$. But then $\mu'=\Ae x\,\psi'$ and $\mu_i'=\Ee x\,\psi_i'\wedge\Ae x\,\psi'$ are equivalent, by assuming that $\psi'\equiv\psi_i'$ by the inductive hypothesis.
For the second equivalence, we note that since $(\tuple{\vphantom\wedge\smash{t}}_1\inc\tuple{\vphantom\wedge\smash{t}}_2)_i$ does not occur in~$\mu$, none of the variables in $\tuple u$ occurs in~$\mu_i'$. Hence the second equivalence holds by locality and flatness.
\end{proof}

The second claim we need shows that we can always extend a team that satisfies the formula $\mu'$ with any teams that satisfy $\mu_i'$ for some $i\leq n$.

\begin{customclaim}{II}\label{AT2}
Let $\mu\in\subf(\varphi)$, $i\leq n$ and let $X_1,X_2$ be teams for which it holds that $\dom(X_2)=\dom(X_1)\!\cup\!\vr(\tuple u)$. Then the following implication holds:
\begin{equation*}
	\text{If } \mathcal{M}\true_{X_1}\mu' \text{ and } \mathcal{M}\true_{X_2}\mu_i', 
	\;\text{ then } \mathcal{M}\true_{X_1\cup X_2^*}\mu',
	\text{ where } X_2^*:=X_2\upharpoonright\dom(X_1).
\end{equation*}
\end{customclaim}

\begin{proof}
We prove this claim by structural induction on $\mu$:

\begin{itemize}[leftmargin=*]
\item Let $\mu$ be a literal and suppose that $\mathcal{M}\true_{X_1}\mu'$ and $\mathcal{M}\true_{X_2}\mu_i'$. Now by locality $\mathcal{M}\true_{X_2^*}\mu_i'$ and furthermore by flatness $\mathcal{M}\true_{X_1\cup X_2^*}\mu'$.

\item Let $\mu = (\tuple{\vphantom\wedge\smash{t}}_1\subseteq\tuple{\vphantom\wedge\smash{t}}_2)_j$ for some $j\leq n$. 
Suppose that $\mathcal{M}\true_{X_1}((\tuple{\vphantom\wedge\smash{t}}_1\inc\tuple{\vphantom\wedge\smash{t}}_2)_j)'$ and $\mathcal{M}\true_{X_2}((\tuple{\vphantom\wedge\smash{t}}_1\inc\tuple{\vphantom\wedge\smash{t}}_2)_j)_i'$. By the first assumption, $\mathcal{M}\true_{X_1}P_j\tuple{\vphantom\wedge\smash{t}}_1$. If $j\neq i$, then $\mathcal{M}\true_{X_2}P_j\tuple{\vphantom\wedge\smash{t}}_1$, and if $j=i$, then $\mathcal{M}\true_{X_2}(\tuple u=\tuple{\vphantom\wedge\smash{t}}_2\,)\wedge P_j\tuple{\vphantom\wedge\smash{t}}_1$. Thus in either case we have $\mathcal{M}\true_{X_2}P_j\tuple{\vphantom\wedge\smash{t}}_1$. Since none of the variables in $\tuple u$ occurs in $\vr(\tuple{\vphantom\wedge\smash{t}}_1)$, by locality $\mathcal{M}\true_{X_2^*}P_j\tuple{\vphantom\wedge\smash{t}}_1$. Because $\mathcal{M}\true_{X_1}P_j\tuple{\vphantom\wedge\smash{t}}_1$ and $\mathcal{M}\true_{X_2^*}P_j\tuple{\vphantom\wedge\smash{t}}_1$, by flatness we have $\mathcal{M}\true_{X_1\cup X_2^*}P_j\tuple{\vphantom\wedge\smash{t}}_1$. That is, $\mathcal{M}\true_{X_1\cup X_2^*}((\tuple{\vphantom\wedge\smash{t}}_1\inc\tuple{\vphantom\wedge\smash{t}}_2)_j)'$.

\item The case $\mu = \psi\wedge\theta$ is straightforward to prove.


\item Let $\mu = \psi\vee\theta$. 
Suppose that $\mathcal{M}\true_{X_1}(\psi\vee\theta)'$ and $\mathcal{M}\true_{X_2}(\psi\vee\theta)_i'$. By the first assumption, $\mathcal{M}\true_{X_1}\psi'\vee\theta'$, i.e. there are $Y_1,Y_1'\subseteq X_1$ s.t. $Y_1\cup Y_1'=X_1$, $\mathcal{M}\true_{Y_1}\psi'$ and $\mathcal{M}\true_{Y_1'}\theta'$.

Suppose first that $(\tuple{\vphantom\wedge\smash{t}}_1\!\subseteq\!\tuple{\vphantom\wedge\smash{t}}_2)_i$ occurs in $\psi$. Because then $(\psi\vee\theta)_i'=\psi_i'$, we have $\mathcal{M}\true_{X_2}\psi_i'$, and thus by the inductive hypothesis $\mathcal{M}\true_{Y_1\cup X_2^*}\psi'$. Now
$(Y_1\cup X_2^*)\cup Y_1' = (Y_1\cup Y_1')\cup X_2^* = X_1\cup X_2^*$.
Hence $\mathcal{M}\true_{X_1\cup X_2^*}\psi'\vee\theta'$, i.e. $\mathcal{M}\true_{X_1\cup X_2^*}(\psi\vee\theta)'$. The case when $(\tuple{\vphantom\wedge\smash{t}}_1\subseteq\tuple{\vphantom\wedge\smash{t}}_2)_i$ occurs in $\theta$ is analogous.

Suppose then that $(\tuple{\vphantom\wedge\smash{t}}_1\subseteq\tuple{\vphantom\wedge\smash{t}}_2)_i$ does not occur in $\psi\vee\theta$. Then $\mathcal{M}\true_{X_2}\psi_i'\vee\theta_i'$, i.e. there exist $Y_2,Y_2'\subseteq X_2$ s.t. $Y_2\cup Y_2'=X_2$, $\mathcal{M}\true_{Y_2}\psi_i'$ and $\mathcal{M}\true_{Y_2'}\theta_i'$. By the inductive hypothesis, $\mathcal{M}\true_{Y_1\cup Y_2^*}\psi'$ and $\mathcal{M}\true_{Y_1'\cup Y_2'^*}\theta'$. Now
\begin{align*}
	(Y_1\cup Y_2^*)\cup(Y_1'\cup Y_2'^*) &= (Y_1\cup Y_1')\cup(Y_2^*\cup Y_2'^*) \\
	&= (Y_1\cup Y_1')\cup(Y_2\cup Y_2')^* = X_1\cup X_2^*.
\end{align*}
Hence $\mathcal{M}\true_{X_1\cup X_2^*}\psi'\vee\theta'$, i.e. $\mathcal{M}\true_{X_1\cup X_2^*}(\psi\vee\theta)'$.

\item Let $\mu = \Es x\,\psi$. 
Suppose that $\mathcal{M}\true_{X_1}(\Es x\,\psi)'$ and $\mathcal{M}\true_{X_2}(\Es x\,\psi)_i'$. Hence $\mathcal{M}\true_{X_1}\Es x\,\psi'$ and $\mathcal{M}\true_{X_2}\Es x\,\psi_i'$. Hence there are $F_1:X_1\rightarrow\mathcal{P}^*(M)$ and $F_2:X_2\rightarrow\mathcal{P}^*(M)$ s.t. $\mathcal{M}\true_{X_1[F_1/x]}\psi'$ and $\mathcal{M}\true_{X_2[F_2/x]}\psi_i'$. By the inductive hypothesis $\mathcal{M}\true_{X_1[F_1/x]\cup(X_2[F_2/x])^*}\psi'$. Let
\begin{align*}
	&F_2^*:X_2^*\rightarrow\mathcal{P}^*(M), 
	\quad s\mapsto \bigl\{b\in F_2(s[\tuple a/\tuple u\,])\,\mid\, s[\tuple a/\tuple u\,]\in X_2,\,\tuple a\in M^k\bigr\} \\
	&F:X_1\cup X_2^*\rightarrow\mathcal{P}^*(M), \quad
	\begin{cases}
		s\mapsto F_1(s)\; \text{ if } s\in X_1\setminus X_2^* \\
		s\mapsto F_2^*(s)\; \text{ if } s\in X_2^*\setminus X_1 \\
		s\mapsto F_1(s)\cup F_2^*(s)\; \text{ if } s\in X_1\cap X_2^*.
	\end{cases}
\end{align*}
By the definitions of $F_2^*$ and $F$, we have
\[
	X_1[F_1/x]\cup(X_2[F_2/x])^* = X_1[F_1/x]\cup X_2^*[F_2^*/x] = (X_1\cup X_2^*)[F/x].
\]
Hence $\mathcal{M}\true_{X_1\cup X_2^*}\Es x\,\psi'$, i.e. $\mathcal{M}\true_{X_1\cup X_2^*}(\Es x\,\psi)'$.

\item Let $\mu = \Ae x\,\psi$. 
Suppose $\mathcal{M}\true_{X_1}(\Ae x\,\psi)'$ and $\mathcal{M}\true_{X_2}(\Ae x\,\psi)_i'$. Thus $\mathcal{M}\true_{X_1}\Ae x\,\psi'$ and $\mathcal{M}\true_{X_2}\Es x\,\psi_i'\wedge\Ae x\,\psi'$. Since $\mathcal{M}\true_{X_2}\Ae x\,\psi'$, by locality  $\mathcal{M}\true_{X_2^*}\Ae x\,\psi'$. Thus by flatness $\mathcal{M}\true_{X_1\cup X_2^*}\Ae x\,\psi'$, i.e. $\mathcal{M}\true_{X_1\cup X_2^*}(\Ae x\,\psi)'$. 
\end{itemize}
\vspace{-0,8cm}
\end{proof}

\smallskip

\noindent
Now we are finally ready to prove Claim \ref{AT3}:
\begin{align*}
	\mathcal{M}\true_X\mu\, \;\text{ iff }\;&\text{there exist }
	A_1,\dots,A_n\subseteq M^k \text{ s.t. } \mathcal{M}'\true_X\mu', \\
	&\text{and for any } i\leq n \text{ and } \tuple a\in A_i \text{ there is } s\in X
	\text{ s.t. } \mathcal{M}'\true_{\{s[\tuple a/\tuple u\,]\}}\mu_i',
\end{align*}
where $\mathcal{M}':=\mathcal{M}[\tuple A/\tuple P]$. 

\begin{proof}
We first examine the special case when $X=\emptyset$:
For the other direction of the equivalence, suppose that $\mathcal{M}\true_X\mu$. Let $A_i:=\emptyset$ for each $i\leq n$ and let $\mathcal{M}':=\mathcal{M}[\tuple A/\tuple P]$. Because $X=\emptyset$, we have $\mathcal{M}'\true_X\mu'$, and since $A_i=\emptyset$ for each $i\leq n$, the rest of the right side of the equivalence holds trivially. The other direction is clear since $\mathcal{M}\true_\emptyset\mu$ holds always.  We may thus assume that $X\neq\emptyset$. 
We prove the claim by structural induction on $\mu$:
\begin{itemize}[leftmargin=*]
\item If $\mu$ is a literal, the claim holds trivially (we can choose $A_i:=\emptyset$ for each $i\leq n$ when proving the other direction of the equivalence).

\item Let $\mu=(\tuple{\vphantom\wedge\smash{t}}_1\subseteq\tuple{\vphantom\wedge\smash{t}}_2)_j$ for some $j\leq n$.
Suppose first that $\mathcal{M}\true_X \tuple{\vphantom\wedge\smash{t}}_1\subseteq\tuple{\vphantom\wedge\smash{t}}_2$. Let
\[
	\mathcal{M}' := \mathcal{M}[\tuple A/\tuple P], \text{ where } 
	A_i :=
	\begin{cases}
		X(\tuple{\vphantom\wedge\smash{t}}_1)\, \text{ if } i=j \\
		\emptyset\qquad\, \text{ else.} 
	\end{cases}
\]
Since $X(\tuple{\vphantom\wedge\smash{t}}_1) = A_j = P_j^\mathcal{M'}$, we have $\mathcal{M}'\true_{X}P_j\tuple{\vphantom\wedge\smash{t}}_1$, i.e. $\mathcal{M}'\true_X (\tuple{\vphantom\wedge\smash{t}}_1\subseteq\tuple{\vphantom\wedge\smash{t}}_2)'$.

Let $i\!\in\!\{1,\dots,n\}\setminus\{j\}$ and let $\tuple a\in A_i$. Since $\mathcal{M}'\true_{X}P_j\tuple{\vphantom\wedge\smash{t}}_1$ we can choose any $s\in X$~$(\neq\emptyset)$, and then by flatness $\mathcal{M}'\true_{\{s\}}P_j\tuple{\vphantom\wedge\smash{t}}_1$. By locality we have $\mathcal{M}'\true_{\{s[\tuple a/\tuple u\,]\}}P_j\tuple{\vphantom\wedge\smash{t}}_1$, i.e. $\mathcal{M}'\true_{\{s[\tuple a/\tuple u\,]\}} (\tuple{\vphantom\wedge\smash{t}}_1\subseteq\tuple{\vphantom\wedge\smash{t}}_2)_i'$.

Let then $i=j$ and $\tuple a\in A_j'$. Because $\tuple a\in X(\tuple{\vphantom\wedge\smash{t}}_1)$, there is $s\in X$ s.t. $s(\tuple{\vphantom\wedge\smash{t}}_1)=\tuple a$. Since $\mathcal{M}\true_X\tuple{\vphantom\wedge\smash{t}}_1\inc\tuple{\vphantom\wedge\smash{t}}_2$, there is $s'\in X$ s.t. $s'(\tuple{\vphantom\wedge\smash{t}}_2)=s(\tuple{\vphantom\wedge\smash{t}}_1)$. Now $s'(\tuple{\vphantom\wedge\smash{t}}_2)=\tuple a$, and thus $s'[\tuple a/\tuple u\,](\tuple u)=s'[\tuple a/\tuple u\,](\tuple{\vphantom\wedge\smash{t}}_2)$, i.e. $\mathcal{M}'\true_{\{s'[\tuple a/\tuple u\,]\}} \tuple u\!=\!\tuple{\vphantom\wedge\smash{t}}_2$. Since $\mathcal{M}'\true_{X}P_j\tuple{\vphantom\wedge\smash{t}}_1$, by locality and flatness $\mathcal{M}'\true_{\{s'[\tuple a/\tuple u\,]\}}P_j\tuple{\vphantom\wedge\smash{t}}_1$. Thus $\mathcal{M}'\true_{\{s'[\tuple a/\tuple u\,]\}}\tuple u\!=\!\tuple{\vphantom\wedge\smash{t}}_2\wedge P_j\tuple{\vphantom\wedge\smash{t}}_1$, i.e. $\mathcal{M}'\true_{\{s'[\tuple a/\tuple u\,]\}}(\tuple{\vphantom\wedge\smash{t}}_1\subseteq\tuple{\vphantom\wedge\smash{t}}_2)_i'$.

\medskip
Suppose then that there exist $A_1,\dots,A_n\subseteq M^k$ s.t. $\mathcal{M}'\true_X(\tuple{\vphantom\wedge\smash{t}}_1\!\subseteq\!\tuple{\vphantom\wedge\smash{t}}_2)'$, and for each $i\leq n$ and $\tuple a\in A_i$ there exists $s\in X$ s.t. $\mathcal{M}'\true_{\{s[\tuple a/\tuple u\,]\}} (\tuple{\vphantom\wedge\smash{t}}_1\subseteq\tuple{\vphantom\wedge\smash{t}}_2)_i'$.

For the sake of proving that $\mathcal{M}\true_X\tuple{\vphantom\wedge\smash{t}}_1\subseteq\tuple{\vphantom\wedge\smash{t}}_2$, let $s\in X$. Since $\mathcal{M}'\true_{X}P_j\tuple{\vphantom\wedge\smash{t}}_1$, by flatness we have $\mathcal{M}'\true_{\{s\}}P_j\tuple{\vphantom\wedge\smash{t}}_1$. Now $s(\tuple{\vphantom\wedge\smash{t}}_1)\in P_j^{\mathcal{M}'}=A_j$ and thus there is $s'\in X$ such that $\mathcal{M}'\true_{\{s'[s(\tuple{\vphantom\wedge\smash{t}}_1)/\tuple u\,]\}}\tuple u\!=\!\tuple{\vphantom\wedge\smash{t}}_2\wedge P_j\tuple{\vphantom\wedge\smash{t}}_1$. In particular, $\mathcal{M}'\true_{\{s'[s(\tuple{\vphantom\wedge\smash{t}}_1)/\tuple u\,]\}}\tuple u=\tuple{\vphantom\wedge\smash{t}}_2$, and thus we have 
$s(\tuple{\vphantom\wedge\smash{t}}_1) = s'[s(\tuple{\vphantom\wedge\smash{t}}_1)/\tuple u\,](\tuple u) = s'[s(\tuple{\vphantom\wedge\smash{t}}_1)/\tuple u\,](\tuple{\vphantom\wedge\smash{t}}_2) = s'(\tuple{\vphantom\wedge\smash{t}}_2)$.
%

\item Let $\mu=\psi\wedge\theta$. 
Suppose first that $\mathcal{M}\true_X \psi\wedge\theta$. Hence $\mathcal{M}\true_X\psi$ and $\mathcal{M}\true_X\theta$. By the inductive hypothesis there are $B_1,\dots,B_n\subseteq M^k$ s.t. $\mathcal{M}[\tuple B/\tuple P]\true_X\psi'$ and there are $B_1',\dots,B_n'\subseteq M^k$ s.t. $\mathcal{M}[\tuple B'/\tuple P]\true_X\theta'$. Moreover, for all $i\leq n$ and tuples $\tuple a\in B_i$ and $\tuple a'\in B_i'$ there are $s,s'\in X$ s.t. $\mathcal{M}[\tuple B/\tuple P]\true_{\{s[\tuple a/\tuple u\,]\}}\psi_i'$ and $\mathcal{M}[\tuple B'/\tuple P]\true_{\{s'[\tuple a'/\tuple u\,]\}}\theta_i'$. Let
\[
	\mathcal{M}' := \mathcal{M}[\tuple A/\tuple P], \text{ where }
	A_i :=
	\begin{cases}
		B_i\; \text{ if $P_i$ occurs in $\psi'$} \\
		B_i'\; \text{ if $P_i$ does not occur in $\psi'$.}  
	\end{cases}
\]
Because none of $P_i$ can occur in both  $\psi'$ and~$\theta'$, we clearly have $\mathcal{M}'\true_X\psi'$ and $\mathcal{M}'\true_X\theta'$. Hence $\mathcal{M}'\true_{X}\psi'\wedge\theta'$, i.e. $\mathcal{M}'\true_{X}(\psi\wedge\theta)'$.

Let $i\leq n$ and let $\tuple a\in A_i$. Suppose first that $P_i$ occurs in $\psi'$. Now $\tuple a\in B_i$, and thus there is $s\in X$ s.t. $\mathcal{M}[\tuple B/\tuple P]\true_{\{s[\tuple a/\tuple u\,]\}}\psi_i'$. Relation variables not occurring in $\psi'$ do not occur in $\psi_i'$ either, and thus $\mathcal{M'}\true_{\{s[\tuple a/\tuple u\,]\}}\psi_i'$. Because $(\tuple{\vphantom\wedge\smash{t}}_1\subseteq\tuple{\vphantom\wedge\smash{t}}_2)_i$ does not occur in~$\theta$ and $\mathcal{M}'\true_X\theta'$, by Claim \ref{AT1} we have $\mathcal{M}'\true_{\{s[\tuple a/\tuple u\,]\}}\theta_i'$. Thus $\mathcal{M}'\true_{\{s[\tuple a/\tuple u\,]\}}\psi_i'\wedge\theta_i'$, i.e. $\mathcal{M}'\true_{\{s[\tuple a/\tuple u\,]\}}(\psi\wedge\theta)_i'$. The case when $P_i$ occurs in $\theta'$ is analogous.
Finally suppose that $P_i$ does not occur in $\psi'$ nor $\theta'$, whence $(\tuple{\vphantom\wedge\smash{t}}_1\subseteq\tuple{\vphantom\wedge\smash{t}}_2)_i$ does not occur in $\psi\wedge\theta$. Since $\mathcal{M}'\true_{X}(\psi\wedge\theta)'$, we can choose any $s\in X\;(\neq\emptyset)$, and then by Claim~\ref{AT1} we have $\mathcal{M}'\true_{\{s[\tuple a/\tuple u\,]\}}(\psi\wedge\theta)_i'$.

\medskip
Suppose then that there are $A_1,\dots,A_n\subseteq M^k$ s.t. $\mathcal{M}'\true_X(\psi\wedge\theta)'$, and for every $i\leq n$ and $\tuple a\in A_i$ there exists $s\in X$ s.t. $\mathcal{M}'\true_{\{s[\tuple a/\tuple u\,]\}}(\psi\wedge\theta)_i'$. Now we have  $\mathcal{M}'\true_X\psi'\wedge\theta'$, i.e. $\mathcal{M}'\true_X\psi'$ and $\mathcal{M}'\true_X\theta'$. 
Because $(\psi\wedge\theta)_i'=\psi_i'\wedge\theta_i'$, for every $i\leq n$ and $\tuple a\in A_i$ there exists $s\in X$ such that $\mathcal{M}'\true_{\{s[\tuple a/\tuple u\,]\}}\psi_i'$. Since also $\mathcal{M}'\true_X\psi'$, by the inductive hypothesis $\mathcal{M}\true_X\psi$. Analogously we have $\mathcal{M}\true_X\theta$, and thus $\mathcal{M}\true_X\psi\wedge\theta$.

\item Let $\mu=\psi\vee\theta$.
Suppose first that $\mathcal{M}\true_X\psi\vee\theta$. Thus there are $Y,Y'\subseteq X$ s.t. $Y\cup Y'=X$, $\mathcal{M}\true_Y\psi$ and $\mathcal{M}\true_{Y'}\theta$. By the inductive hypothesis there are $B_1,\dots,B_n,B_1',\dots,B_n'\subseteq M^k$ s.t. $\mathcal{M}[\tuple B/\tuple P]\true_Y\psi'$ and $\mathcal{M}[\tuple B'/\tuple P]\true_{Y'}\theta'$. In addition, for every $i\leq n$, $\tuple  a\in B_i$ and $\tuple a'\in B_i'$ there exist $s\in Y$ and $s'\in Y'$ s.t. $\mathcal{M}[\tuple B/\tuple P]\true_{\{s[\tuple a/\tuple u\,]\}}\psi_i'$ and $\mathcal{M}[\tuple B'/\tuple P]\true_{\{s'[\tuple a'/\tuple u\,]\}}\theta_i'$. Let
\[
	\mathcal{M}' := \mathcal{M}[\tuple A/\tuple P], \text{ where }
	A_i :=
	\begin{cases}
		B_i\; \text{ if $P_i$ occurs in $\psi'$} \\
		B_i'\; \text{ if $P_i$ does not occur in $\psi'$.}
	\end{cases}
\]
Because none of $P_i$ can occur in both  $\psi'$ and~$\theta'$, we clearly have $\mathcal{M}'\true_Y\psi'$ and $\mathcal{M}'\true_{Y'}\theta'$. Therefore $\mathcal{M}'\true_{X}\psi'\vee\theta'$, i.e. $\mathcal{M}'\true_{X}(\psi\vee\theta)'$.

Let $i\leq n$ and $\tuple a\in A_i$. Suppose first that $P_i$ occurs in $\psi'$. Now $A_i=B_i$, and thus, by the inductive hypothesis, there is $s\in Y$ $(\subseteq X)$ such that $\mathcal{M}[\tuple B/\tuple P]\true_{\{s[\tuple a/\tuple u\,]\}}\psi_i'$. Relation variables not occurring in $\psi'$ do not occur in $\psi_i'$ either, and thus $\mathcal{M'}\true_{\{s[\tuple a/\tuple u\,]\}}\psi_i'$, i.e. $\mathcal{M'}\true_{\{s[\tuple a/\tuple u\,]\}}(\psi\vee\theta)_i'$. The case when $P_i$ occurs in $\theta'$ is analogous.
Suppose then that $P_i$ does not occur in $\psi'$ nor $\theta'$, whence $(\tuple{\vphantom\wedge\smash{t}}_1\subseteq\tuple{\vphantom\wedge\smash{t}}_2)_i$ does not occur in $\psi\vee\theta$. Since $\mathcal{M}'\true_{X}(\psi\vee\theta)'$, we can choose any $s\in X\;(\neq\emptyset)$, and then by Claim~\ref{AT1} we have $\mathcal{M}'\true_{\{s[\tuple a/\tuple u\,]\}}(\psi\vee\theta)_i'$.

\medskip

Suppose then that there are $A_1,\dots,A_n\subseteq M^k$ s.t. $\mathcal{M}'\true_X(\psi\vee\theta)'$, and for all $i\leq n$ and $\tuple a\in A_i$ there is $s_{i,\vec a}\in X$ s.t. $\mathcal{M}'\true_{\{s_{i,\vec a}[\tuple a/\tuple u\,]\}}(\psi\vee\theta)_i'$. 
Since $\mathcal{M}'\true_X\psi'\vee\theta'$, there are $Y,Y'\subseteq X$ s.t. $Y\cup Y'=X$, $\mathcal{M}'\true_Y\psi'$ and $\mathcal{M}'\true_{Y'}\theta'$. 
We define the teams $Y_i$ and $Y_i'$, for every $i\leq n$, and the teams $Z,Z'\subseteq X$:
\begin{align*}
	Y_i&:= \{s_{i,\vec a}[\vec a/\vec u\,]\mid \vec a\in A_i\} \; \text{ if } 
		(\tuple{\vphantom\wedge\smash{t}}_1\inc\tuple{\vphantom\wedge\smash{t}}_2)_i \text{ occurs in } \psi,
		\quad \text{ and else } Y_i:=\emptyset. \\
	Y_i'&:= \{s_{i,\vec a}[\vec a/\vec u\,]\mid \vec a\in A_i\} \; \text{ if } 
		(\tuple{\vphantom\wedge\smash{t}}_1\inc\tuple{\vphantom\wedge\smash{t}}_2)_i \text{ occurs in } \theta,
		\quad \text{ and else } Y_i':=\emptyset. \\
	Z&:=Y\cup(\bigcup_{i\leq n}Y_i\upharpoonright\dom(X)), \quad
	Z':=Y'\cup(\bigcup_{i\leq n}Y_i'\upharpoonright\dom(X)). \\[-0,8cm]
\end{align*}
We then show that for every $i\leq n$ it holds that $\mathcal{M}'\true_{Y_i}\psi_i'$. Let $i\leq n$. If $(\tuple{\vphantom\wedge\smash{t}}_1\inc\tuple{\vphantom\wedge\smash{t}}_2)_i$ does not occur in $\psi$, then $Y_i=\emptyset$ whence trivially $\mathcal{M}'\true_{Y_i}\psi_i'$. Suppose then that $(\tuple{\vphantom\wedge\smash{t}}_1\inc\tuple{\vphantom\wedge\smash{t}}_2)_i$ occurs in $\psi$. Now $(\psi\vee\theta)_i'=\psi_i'$ and thus $\mathcal{M}'\true_{\{r\}}\psi_i'$ for every $r\in Y_i$. By flatness we thus have $\mathcal{M}'\true_{Y_i}\psi_i'$.
Since $\mathcal{M}'\true_Y\psi'$ and $\mathcal{M}'\true_{Y_i}\psi_i'$ for every $i\leq n$, we can apply Claim \ref{AT2} for each $i\leq n$ to obtain $\mathcal{M}'\true_{Z}\psi'$. By a symmetric argumentation $\mathcal{M}'\true_{Z'}\theta'$.

We then show that $\mathcal{M}\true_Z\psi$. We may suppose that $Z\neq\emptyset$, since else trivially $\mathcal{M}\true_Z\psi$. Let $i\leq n$ and let $\tuple a\in A_i$. Suppose first that $(\tuple{\vphantom\wedge\smash{t}}_1\inc\tuple{\vphantom\wedge\smash{t}}_2)_i$ occurs in $\psi$. Now $(\psi\vee\theta)_i'=\psi_i'$ and thus there is $s_{i,\vec a}\in X$ s.t. $\mathcal{M}'\true_{\{s_{i,\vec a}[\tuple a/\tuple u\,]\}}\psi_i'$. By the definition of $Y_i$ we must have $s_{i,\vec a}[\tuple a/\tuple u\,]\in Y_i$ and moreover $s_{i,\vec a}\in Z$. Suppose then that $(\tuple{\vphantom\wedge\smash{t}}_1\inc\tuple{\vphantom\wedge\smash{t}}_2)_i$ does not occur in $\psi$. Then we can choose any assignment $s\in Z \,(\neq\emptyset)$, whence by Claim~\ref{AT1} we have $\mathcal{M}'\true_{\{s[\tuple a/\tuple u\,]\}}\psi_i'$. Hence, by the inductive hypothesis, $\mathcal{M}\true_Z\psi$.
We can analogously deduce $\mathcal{M}\true_{Z'}\theta$. Since $Z\cup Z'=X$, we have $\mathcal{M}\true_{X}\psi\vee\theta$.

\item Let $\mu=\Es x\,\psi$. 
Suppose first that $\mathcal{M}\true_X\Es x\,\psi$, i.e. there is $F:X\rightarrow\mathcal{P}^*(M)$, s.t. $\mathcal{M}\true_{X[F/x]}\psi$. By the inductive hypothesis there are $A_1,\dots,A_n\subseteq M^k$ s.t. $\mathcal{M}'\true_{X[F/x]}\psi'$, where $\mathcal{M}'\!:=\!\mathcal{M}[\tuple A/\tuple P]$. Moreover, for all $i\!\leq\! n$ and $\tuple a\!\in\! A_i$ there is $r\in X[F/x]$ s.t. $\mathcal{M}'\true_{\{r[\tuple a/\tuple u\,]\}}\psi_i'$. Since $\mathcal{M}'\true_{X[F/x]}\psi'$, we have $\mathcal{M}'\true_X \Es x\,\psi'$, i.e. $\mathcal{M}'\true_X (\Es x\,\psi)'$.

Let $i\leq n$ and let $\tuple a\in A_i$. Now there is $r\in X[F/x]$ s.t. $\mathcal{M}'\true_{\{r[\tuple a/\tuple u\,]\}}\psi_i'$. Since $r\in X[F/x]$, there is $s\in X$ and $b\in F(s)$ s.t. $r=s[b/x]$. Let $F':\{s[\tuple a/\tuple u\,]\}\rightarrow\mathcal{P}^*(M)$ s.t. $s[\tuple a/\tuple u\,]\mapsto\{b\}$.
Since $\{s[\tuple a/\tuple u\,]\}[F'/x]=\{r[\tuple a/\tuple u\,]\}$, we have $\mathcal{M}'\true_{\{s[\tuple a/\tuple u\,]\}}\Es x\,\psi_i'$, i.e. $\mathcal{M}'\true_{\{s[\tuple a/\tuple u\,]\}}(\Es x\,\psi)_i'$.

\medskip

Suppose then that there are $A_1,\dots,A_n\subseteq M^k$ s.t. $\mathcal{M}'\true_X(\Es x\,\psi)'$, and for every $i\leq n$ and $\tuple a\in A_i$ there is $s_{i,\vec a}\in X$ s.t. $\mathcal{M}'\true_{\{s_{i,\vec a}[\tuple a/\tuple u\,]\}}(\Es x\,\psi)_i'$.
Now there is $F:X\rightarrow\mathcal{P}^*(M)$ s.t. $\mathcal{M}\true_{X[F/x]}\psi'$.  Furthermore, for each $i\leq n$ and $\tuple a\in A_i$ there is $F_{i,\tuple a}:\{s_{i,\tuple a}[\tuple a/\tuple u\,]\}\rightarrow\mathcal{P}^*(M)$ s.t. $\mathcal{M}'\true_{\{s_{i,\tuple a}[\tuple a/\tuple u\,]\}[F_{i,\tuple a}/x]}\psi_i'$. For each $i\leq n$ let
\[
	X_i' := \bigcup_{\tuple a\in A_i} \{s_{i,\tuple a}[\tuple a/\tuple u\,]\}[F_{i,\tuple a}/x].
\]
By flatness $\mathcal{M}'\true_{X_i'}\psi_i'$ for each $i\leq n$. Let $F':X\rightarrow\mathcal{P}^*(M)$ s.t.
\[
	s\mapsto F(s)\cup\{b\in F_{i,\tuple a}(s_{i,\tuple a}[\tuple a/\tuple u\,])\,\mid\, i\leq n,\, \tuple a\in A_i \,
	\text{ s.t. } s=s_{i,\tuple a}\}.
\]
By the definitions of $F'$ and $X_i'$ ($i\leq n$) we have
\[
	X[F/x]\cup \bigl(\bigcup_{i\leq n}X_i'\upharpoonright\dom(X[F/x])\bigr) = X[F'/x].
\]
Thus by applying Claim \ref{AT2} for each $i\leq n$, we obtain $\mathcal{M}'\true_{X[F'/x]}\psi'$. Moreover, now for each $i\leq n$ and $\tuple a\in A_i$ there is $r\in X[F'/x]$ s.t. $\mathcal{M}'\true_{\{r[\tuple a/\tuple u\,]\}}\psi_i'$. Thus, by the inductive hypothesis, $\mathcal{M}\true_{X[F'/x]}\psi$, i.e. $\mathcal{M}\true_X\Es x\,\psi$.

\item Let $\mu=\Ae x\,\psi$.
Suppose first that $\mathcal{M}\true_X \Ae x\,\psi$, i.e. $\mathcal{M}\true_{X[M/x]}\psi$. By the inductive hypothesis there are $A_1,\dots,A_n\subseteq M^k$ s.t. $\mathcal{M}'\true_{X[M/x]}\psi'$. Moreover, for each $i\leq n$ and $\tuple a\in A_i$ there exists $r\in X[M/x]$ s.t. $\mathcal{M}'\true_{\{r[\tuple a/\tuple u\,]\}}\psi_i'$. Now $\mathcal{M}'\true_X\Ae x\,\psi'$, i.e. $\mathcal{M}'\true_X(\Ae x\,\psi)'$.

Let $i\leq n$ and let $\tuple a\in A_i$. Now there is $r\in X[M/x]$ s.t. $\mathcal{M}'\true_{\{r[\tuple a/\tuple u\,]\}}\psi_i'$. Since $r\in X[M/x]$, there is $s\in X$ and $b\in M$ s.t. $r=s[b/x]$. Let $F:\{s[\tuple a/\tuple u\,]\}\rightarrow\mathcal{P}^*(M)$ s.t. $s[\tuple a/\tuple u\,]\mapsto\{b\}$.
Now $\{s[\tuple a/\tuple u\,]\}[F/x]=\{r[\tuple a/\tuple u\,]\}$, and therefore $\mathcal{M}'\true_{\{s[\tuple a/\tuple u\,]\}}\Es x\,\psi_i'$. Since $\mathcal{M}'\true_X\Ae x\,\psi'$, by flatness and locality we have $\mathcal{M}'\true_{\{s[\tuple a/\tuple u\,]\}}\Ae x\,\psi'$. Hence $\mathcal{M}'\true_{\{s[\tuple a/\tuple u\,]\}}\Es x\,\psi_i'\wedge\Ae x\,\psi'$, i.e. $\mathcal{M}'\true_{\{s[\tuple a/\tuple u\,]\}} (\Ae x\,\psi)_i'$.

\medskip
Suppose then that there are $A_1,\dots,A_n\subseteq M^k$ s.t. $\mathcal{M}'\true_X(\Ae x\,\psi)'$, and that for each $i\leq n$ and $\tuple a\in A_i$ there is $s\in X$ s.t. $\mathcal{M}'\true_{\{s[\tuple a/\tuple u\,]\}}(\Ae x\,\psi)_i'$. Now we have $\mathcal{M}'\true_X\Ae x\,\psi'$, i.e. $\mathcal{M}'\true_{X[M/x]}\psi'$.

Let $i\leq n$ and let $\tuple a\in\!A_i$. Now there is $s\in\!X$ s.t. $\mathcal{M}'\true_{\{s[\tuple a/\tuple u\,]\}}\!\Es x\,\psi_i'\wedge\Ae x\,\psi'$ and thus there is $F:\{s[\tuple a/\tuple u\,]\}\rightarrow\mathcal{P}^*(M)$ s.t. $\mathcal{M}'\true_{\{s[\tuple a/\tuple u\,]\}[F/x]}\psi_i'$. Let $b\in F(s[\tuple a/\tuple u])$ and let $r:=s[b/x]$, whence $r\in X[M/x]$ and $r[\tuple a/\tuple u]\in\{s[\tuple a/\tuple u\,]\}[F/x]$. Now by flatness $\mathcal{M}'\true_{\{r[\tuple a/\tuple u]\}}\psi_i'$. Therefore, by the inductive hypothesis, $\mathcal{M}\true_{X[M/x]}\psi$, i.e. $\mathcal{M}\true_X \Ae x\,\psi$.
\end{itemize}
\vspace{-0,7cm}
\end{proof}

\end{document}